\documentclass[12pt,twoside]{article}
\usepackage{amsmath,amssymb,amsthm,mathptmx,amsfonts,mathtools}
\usepackage[english]{babel}

\usepackage[inner=2.2cm,outer=2.6cm,bottom=2.3cm,top=3cm]{geometry}
\usepackage{graphicx}
\usepackage{hyperref}
 \allowdisplaybreaks

\theoremstyle{definition}
\newtheorem{theorem}{Theorem}[section]
\newtheorem{definition}[theorem]{Definition}

\newtheorem{lemma}[theorem]{Lemma}
\newtheorem{proposition}[theorem]{Proposition}
\newtheorem{schemetmp}{Scheme}[section]
\newenvironment{scheme}[1]
  {\renewcommand\theschemetmp{#1}\schemetmp}
  {\endschemetmp}
\usepackage{cleveref}
\crefname{enumi}{}{}
\newcommand{\ov}[1]{\overline{{#1}}}
\newcommand{\un}[1]{\underline{{#1}}}

\theoremstyle{remark}
\newtheorem*{remark}{Remark}

\newcommand*\diff{\mathop{}\!\mathrm{d}}
\newcommand*\discreteDiff{\boldsymbol{d}}
\newcommand{\norm}[1]{\left\lVert#1\right\rVert}
\newcommand{\weakto}{\rightharpoonup}
\newcommand{\weakstarto}{\overset{\ast}{\rightharpoonup}}
\newcommand{\project}{\mathcal{P}_h}
\newcommand{\projectL}{\mathcal{R}_h}
\newcommand{\projectV}{\mathcal{Q}_h}
\newcommand{\interpol}{\mathcal{I}_{h} }
\newcommand{\abs}[1]{\left \lvert #1 \right \rvert}

\DeclareMathOperator{\tr}{tr}

\newcommand{\W}[2]{W^{#1,#2}\left(\Omega \right)}
\renewcommand{\L}[1]{L^{#1} \left(\Omega \right)}

\renewcommand{\H}[1]{H^{#1} \left(\Omega \right) }
\newcommand{\qspace}[1]{H^{1}_0 (\Omega)}
\newcommand{\velspace}{H^2(\Omega)\cap \V }
\newcommand{\sobnorm}[3]{\left\lVert#1\right\rVert_{\W{#2}{#3}}  }
\newcommand{\Hsobnorm}[2]{\left\lVert#1\right\rVert_{\H{#2}}  }

\newcommand{\Hseminorm}[2]{\left \lVert #1\right \rVert_{\H{#2}}  }
\newcommand{\lebnorm}[2]{\left\lVert#1\right\rVert_{\L{#2}}  }

\newcommand{\LTwoNorm}[1]{\lebnorm{#1}{2}}
\newcommand{\hnorm}[1]{\left\lVert#1\right\rVert_{h}}

\DeclareMathOperator{\Rr}{\mathbb{R}}

\DeclareMathOperator{\Cont}{\mathcal{C}}

\DeclareMathOperator{\ra}{\rightarrow}
\newcommand{\de}{\text{d}}

\newcommand{\sym}{\text{sym}}
\newcommand{\skw}{\text{skw}}

\DeclareMathOperator{\BV}{{BV}([0,T])}

\newcommand{\fat}[1]{{\pmb{ #1}}}
\DeclareMathOperator{\di}{\nabla\cdot}

\newcommand{\tu}{\tilde{\fat u}}

\renewcommand{\t}{\partial_t}

\newcommand{\dd}{\tilde{\fat d}}
\newcommand{\vv}{\tilde{\fat v}}
\newcommand{\He}{H^1(\Omega; \mathbb{S}^2)}

\newcommand{\V}{H^1_{0,\sigma}(\Omega)}

\newcommand{\Ha}{{L}^2_{\sigma}(\Omega)}

\DeclareMathOperator{\inter}{\mathcal{I}_h}
\DeclareMathOperator{\interad}{\widetilde{\mathcal{I}}_h}

\newcommand{\E}{\mathcal{E}}

\newcommand{\expon}{\text{\textit{e}}}

\author{
Robert Lasarzik%
\thanks{Weierstrass Institute,
Mohrenstr. 39, 10117 Berlin, Germany,
\texttt{robert.lasarzik@wias-berlin.de}}, Maximilian E.V. Reiter\thanks{Technical University of Berlin,
Straße~des~17.~Juni~135,~10623~Berlin,~Germany,
\texttt{m.reiter@tu-berlin.de}}}
\date{\today}
\title%
{Analysis and numerical approximation of energy-variational solutions to the \\Ericksen--Leslie equations}
\begin{document}
\maketitle

\begin{abstract}
We define the concept of energy-variational solutions for the Ericksen--Leslie equations in three spatial dimensions. 
This solution concept is finer than dissipative solutions and satisfies the weak-strong uniqueness property. 
For a certain choice of the regularity weight, the existence of energy-variational solutions implies the existence of measure-valued solutions and for a different choice, 
we construct an energy-variational solution with the help of an implementable, structure-inheriting  space-time discretization. 
Computational studies are performed in order to provide some evidence of the applicability of the proposed algorithm.  
\end{abstract}
\tableofcontents

\section{Introduction}
Liquid crystals comprise the structural properties of crystals within a fluid.
The fluid flow of the nematic phase of liquid crystals can be  described by  the Ericksen--Leslie system.
In this model, the material behaves like a liquid, \textit{i.e.,} no positional order is present, but the molecules exhibit a long-range self-alignment along a direction (see Figure \ref{fig:nematic_phase}). 
\begin{figure}[h]
     \centering
     \includegraphics[ width=\textwidth]{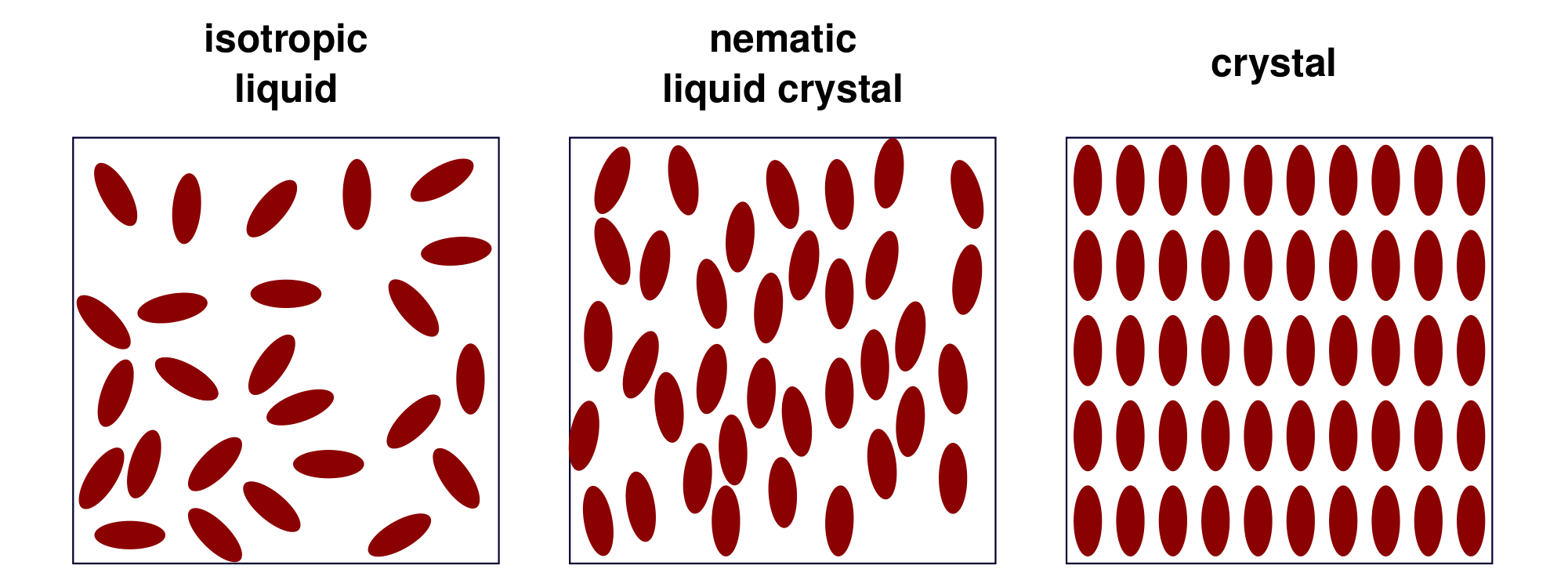}
     \caption{Alignment of the rod-like molecules in an isotropic liquid, the nematic phase of a liquid crystal, and a solid, {\cite[Fig. 1.1]{lasarzik_thesis}}}
     \label{fig:nematic_phase}
 \end{figure}
In this way, liquid crystals sustain anisotropic dynamics such as the polarization of light or the transfer of heat, but at the same time offer the physical flexibility of a fluid, which makes these materials interesting for engineering and sciences~\cite{bild}. 

Ericksen \cite{ericksen2} and Leslie \cite{Leslie1968} derived the Ericksen-Leslie equations during their development of an instationary theory of liquid crystals in the 1960s.
Let $\fat v: \overline{\Omega} \times [0,T] \to \mathbb{R}^3$ denote the velocity of the fluid,  $\fat p: \overline{\Omega} \times [0,T] \to \mathbb{R}$ its pressure and $\fat d: \overline{\Omega} \times [0,T] \to \mathbb{R}^3$ the director.
We consider the system governed by the equations
\begin{subequations}\label{system}
\begin{align}
\begin{split}\label{system_a}
    \partial_t \fat v + (\fat v \cdot \nabla) \fat v  + \nabla \fat p + \di \left ( (\nabla \fat d )^T\nabla \fat d  \right )   - \nabla \cdot \fat T^L &= 0, 
\end{split}
\\
\begin{split}\label{system_b}    
    \nabla \cdot \fat v &= 0, 
\end{split}
\\
\begin{split}\label{system_c}
    \partial_t \fat d + (\fat v \cdot \nabla) \fat d - (\nabla \fat v)_{skw} \fat d +\left ( I - \fat d \otimes \fat d \right ) \left ( \lambda (\nabla \fat v)_{sym} \fat d - \Delta \fat d \right )&=0, 
\end{split}
\\
\begin{split}\label{system_d}
\vert \fat d \vert & = 1 ,
\end{split}
\end{align}
\end{subequations}
where we employ the initial conditions
\begin{align*}
    \fat v(0) = \fat v_0, \quad \fat  d(0) = \fat d_0 \text{ with } \abs{\fat d_0} = 1 \text{ in } \Omega
\end{align*}
and boundary conditions
\begin{align*}
    \fat v = 0, \quad \fat d =\fat  d_{\Gamma}  \text{ with } \abs{\fat d_\Gamma} = 1 \text{ on } \partial \Omega .
\end{align*}
The Leslie stress tensor $T^L$ is defined by
\begin{equation*}
  \fat   T^L \coloneqq
     \fat T^D 
    + \lambda [\fat d \otimes [\fat d]_x^T  [\fat d]_x \cdot\Delta \fat  d]_{sym}
    + [\fat d \otimes\Delta\fat  d]_{skw},
\end{equation*}
where $T^D$ collects the dissipative terms of the Leslie stress tensor, \textit{i.e.}
\begin{align*}
    \fat   T^D \coloneqq
    & (\mu_1 + \lambda^2) (\fat d \cdot (\nabla \fat v)_{sym} \fat d) (\fat d \otimes \fat d)
    + \mu_4 (\nabla \fat v)_{sym} \\&
    +(\mu_5 + \mu_6 - \lambda^2) (\fat d \otimes (\nabla \fat v)_{sym}  \fat d)_{sym}
\end{align*}
with $ \mu_4 > 0, \quad \mu_5 + \mu_6 - \lambda^2\geq 0, \quad \mu_1 +\lambda^2 \geq 0$ in order to ensure the dissipative character of our model.

So far, a vast majority of the mathematical work on the Ericksen--Leslie model considers a simplified system with a relaxed unit-norm constraint that is only enforced approximately by adding a Ginzburg--Landau penalization term $\fat f_\epsilon (d) = \frac{1}{4\epsilon} (\abs{\fat d}^2-1)^2$ to the free energy potential $\frac{1}{2}\vert\nabla \fat d\vert^2$.
With simplified Leslie-stress tensor, the momentum and director equation~\eqref{system} are  replaced by
\begin{align}
\begin{aligned}\label{simplification}
    \begin{split}
    \partial_t \fat v - \frac{1}{Re}\Delta \fat v +(\fat v \cdot \nabla)\fat  v  + \nabla \fat p + \gamma \nabla \cdot \left((\nabla \fat d)^T \nabla \fat d\right) &= 0, \quad
    \di \fat v = 0 \,,
\end{split}\\[2ex]
\begin{split}
   \partial_t \fat d + (\fat v \cdot \nabla) \fat d - \Delta \fat d +  \fat f_\epsilon (\fat d) &=0,
\end{split}
\end{aligned}
\end{align} 
respectively. In the first analysis of this system~\cite{lin_and_liu}, the authors were
able to prove the existence of weak solutions to~\eqref{simplification}. 
A rather general model including the full Leslie stress tensor is considered by \cite{rocca}, where again a Ginzburg--Landau penalization approach is introduced to replace the unit-norm constraint of the director. In this setting the authors prove that weak solutions exist and a blow-up criterion for local strong solutions. 
In \cite{lasarzik-weak_solutions} the existence of weak solutions is generalized to a  larger class of free energy functions.
An overview of the analytical results regarding the Ericksen--Leslie equations and its connection to other models for liquid crystals can be found in \cite{sabine}.
In two spacial dimensions, for the limiting system of~\eqref{simplification}, where $ f_\varepsilon(\fat d)$ is replaced by $-\vert \nabla \fat d \vert^2$, it is known that a unique weak solution exists  that is smooth except for finitely many points in time~\cite{lin_liu_solutions_to_el}.

For a general model of the Ericksen--Leslie equations equipped with the naturally arising an\-iso\-tro\-pic Oseen--Frank energy, the concept of dissipative solutions is applied in \cite{lasarzik_dissipative_solutions}, which are shown to be the local average of measure-valued solutions~\cite{lasarzik_measure_valued} and inherit their weak-strong uniqueness~\cite{weakstrongmeas}. 
In comparison to measure-valued solutions, dissipative solutions have the advantage that they have less degrees of freedom and can be approximated by numerical schemes. Nevertheless, they form only a subset of measure-valued solutions and are thus not as precise. In the work at hand, we introduce energy-variational solutions (\textit{cf.}~\cite{lasarzik_incompressible_fluids}), which have one degree of freedom more than dissipative solutions and can be argued to be as fine as measure-valued solutions, but as we will show, they can also be approximated by numerical schemes and have additional advantages. 
This solution concept is useful where one might either not be able to derive weak solutions for physically relevant models or where weak solutions admit unphysical behaviour, like unphysical non-uniqueness~\cite{Isett}. The solution concept was first introduced in \cite{lasarzik_existence_2021}. 
We will prove that this concept is finer than the so-called dissipative solutions introduced by Lions~\cite{lions}  for the Euler equations and also applied to the Ericksen--Leslie equations by one of the authors~\cite{lasarzik_dissipative_solutions}. 
Additionally, in certain scenarios this concept is finer than the concept of measure-valued solutions. 
Energy-variational solutions do not only fulfill the standard weak-strong uniqueness property of generalized solution concepts (\textit{cf.}~\cite{weakstrongmeas}), but they also fulfill the semi-flow property such that prolongations and restrictions of solutions on larger, and smaller time intervals, respectively, are energy-variational solutions again. 
In~\cite{lasarzik_existence_2021} it was also argued that energy-variational solutions are amenable for different selection criteria. 
The set of energy-variational solutions can be seen as a convex, weakly$^*$-closed superset of weak solutions but a subset of dissipative solutions and for a special choice of the regularity weight also a subset of measure-valued solutions. This may allow to introduce techniques from optimization theory in order to select the physically relevant solution maximizing the dissipation in every point-in-time~\cite{lasarzik_existence_2021}. 

In this work, we define the concept of energy-variational solutions for the Ericksen–Leslie equations in three spatial dimensions.
This definition has some freedom, since it depends on the choice of a certain regularity weight~$\mathcal K$. For one choice of such a regularity weight, we prove the equivalence to a certain class of measure-valued solutions, and for another choice, we construct an energy-variational solution with the help of an implementable, structure-inheriting space-time discretization based on the finite element method.

Concerning the numerical approximation of the Ericksen--Leslie equations, a first  study for the simplified system \eqref{simplification} equipped with the Ginzburg--Landau approximation is conducted in \cite{Liu_Walkington}. They combined an implicit Euler scheme in time with Hermite type finite elements for the director and $Q2$-$Q1$-Taylor--Hood elements for the velocity and pressure.
In their subsequent work \cite{liuwalkingtonmixed}, the authors replace the  demanding Hermite finite elements for the director by piecewise quadratic functions. 
Even a relaxation from $C^1$ finite elements to $C^0$ finite elements is realized in \cite{lin_and_liu_num}.
A different approach is proposed in \cite{badia_saddle_point}, where  the Ginzburg--Landau approximation of the unit-norm constraint is interpreted as a saddle-point structure.
For the simplified and penalized Ericksen--Leslie system also decoupling techniques and mixed methods are examined in \cite{girault-gonzalez,cabralesFEM}.
In \cite{Becker_Prohl_2008} two numerical schemes for the simplified model are proposed. The first one uses the Ginzburg--Landau approximation for the unit-norm constraint and the second proposed scheme does not depend on a regularization parameter and fulfills the unit-norm constraint in the limit. However both schemes do not fulfill the unit-norm constraint at the discrete level exactly. We therefore use an approach from~\cite{lasarzik_main} by implementing a midpoint rule at the finite element level which solves the sphere constraint exactly at every node of the mesh. But we have to refine this approach by introducing a special projection (see Remark~\ref{rem:proj}), which will allow to identify the limit as an energy-variational solution. 
The proposed scheme is the first numerical scheme implementing the main properties of the continuous system including the algebraic norm restriction at every node of the mesh
such that the approximate solutions converge to an  energy-variational solution fulfilling the physically relevant semi-flow property. 
We think that the concept of energy-variational solutions is a strong tool for identifying limits of solutions to numerical schemes and can also serve as such in other related models. 

After providing the considered system and an overview over the existing literature, we
introduce the necessary notation  in Section~\ref{Sec:not}, we provide the definitions and the main results in Subsection~\ref{ch:2}. The weak-strong uniqueness proof and the relation to measure-valued solutions in considered in Subsection~\ref{sec:ana}, whereas the preliminaries are provided in the following subsections on auxiliary lemmata~\ref{sec:aux}, finite elements~\ref{sec:fem_spaces}, and interpolation~\ref{sec:interpol}.  The discrete system is introduced and analysed in Section~\ref{sec:dis}, we provide the necessary \textit{a priori} estimates~\ref{sec:apri}, extract converging subsequences~\ref{sec:extract}, and prove convergence to the director equations~\ref{sec:dir} and the energy-variational inequality~\ref{sec:envar}. Finally, some computational studies are presented in Section~\ref{sec:comp} in order to show some evidence of the applicability of the proposed algorithm.

\section{Preliminaries, main results, and continuous system\label{Sec:not}}
We denote the space of smooth solenoidal functions with compact support by $\mathcal{C}_{c,\sigma}^\infty(\Omega;\Rr^3)$. By $ L^p_{\sigma}( \Omega) $, $\V$,  and $  W^{1,p}_{0,\sigma}( \Omega)$, we denote the closure of $\mathcal{C}_{c,\sigma}^\infty(\Omega;\Rr^3)$ with respect to the norm of $ L^p(\Omega) $, $  H^1( \Omega) $, and $  W^{1,p}(\Omega)$, respectively. By $\He$, we denote the functions $ \fat d \in  H^1(\Omega)$ such that $ \vert\fat d \vert = 1$ a.e.~in $\Omega$. 
The $ L^2 (\Omega)$ inner product is thereby denoted by $(.,.)$.
The Dual space of a Banach space $V$ is denoted by $V^*$, where the dual pairing is denoted as $\langle \cdot , \cdot \rangle$.
In order to define the cross product, we use the Levi-Cita symbol. The cross product of two vectors $x,y \in \mathbb{R}^3$ is then defined as $(x \times y)_i = \sum_{j,k} \epsilon_{i,j,k} x_j y_k$, the cross product of a vector $x \in \mathbb{R}^3$ with a matrix $A \in \mathbb{R}^{d\times d}$ as $(x \times A)_{i,j} = \sum_{l,m} \epsilon_{i,l,m} x_l A_{m,j}$. We further will make use of the matrix notation $ $ of the cross product for three dimensions, \textit{i.e.} $([x]_{\times})_{ik}  =\sum_{j} \epsilon_{i,j,k} x_j $.
We further introduce the discrete derivative $\discreteDiff_t$ as $\discreteDiff_t f^j \coloneqq \frac{f^j -f^{j-1}}{k}$ for a constant time-step size $k>0$ and for a sequence in a normed space $(f_j)_{j=1,...,n}$, where we set $\discreteDiff_t f^0 = 0$.
Usually  functions in the continuous setting are denoted by bold letters in contrast to the discrete functions. 

By $  \mathbb{R}^{d\times d}$ we denote $d$-dimensional quadratic matrices, by $  \mathbb{R}^{d\times d}_{\sym}$ the symmetric subset, and by $  \mathbb{R}^{d\times d}_{\sym,+}$ the symmetric positive semi-definite matrices. 
The symmetric and skew-symmetric part of a matrix $A\in \Rr^{d\times d}$ are denoted by $(A)_{\sym} $ and $(A)_{\skw}$, respectively. The positive and negative semi-definite part of  a matrix $A\in \Rr^{d\times d}$ are denoted by $(A)_{+} $ and $(A)_{-}$, respectively.
We equip the last set with the usual spectral norm $\vert A\vert_2 = \max_{i\in\{ 1,\ldots,d\}} \lambda_i$, where $\lambda _i$ are the nonnegative eigenvalues of the matrix $A\in  \mathbb{R}^{d\times d}_{\sym,+}$. The dual norm of the spectral norm with respect to the Frobenius product ($A : B := \sum_{i,j=1}^d A_{ji}B_{ij}$ for $A,B\in \Rr^{d\times d}$) is given by $\vert A \vert '_2 = \sum_{i=1}^d \lambda _i = A:I = \tr(A)$ for a matrix $A \in  \mathbb{R}^{d\times d}_{\sym,+}$. 
The Radon measures taking values in  $\mathbb{R}^{d\times d}_{\sym}$ are denoted by $\mathcal{M}(\ov\Omega ; \mathbb{R}^{d\times d}_{\sym} ) $, which may be interpreted as the dual space of the continuous functions,\textit{i.e.,}
$\mathcal{M}(\ov\Omega; \mathbb{R}^{d\times d}_{\sym} )
= ( 
\Cont
(\ov 
\Omega; 
\mathbb{R}^{d\times d}_{\sym} ) )^{*}$.
Note that an element $\mu \in \mathcal{M}(\ov\Omega;  \mathbb{M}^{d\times d}_{\sym,+} ) $ is a Radon measure taking values in the symmetric matrices such that for any $\fat \xi \in\Rr^d$ the measure $ \fat \xi \otimes \fat \xi : \mu $ is nonnegative. 
By $I$, we denote the identity matrix in $\Rr^{d\times d}$.

For a given Banach space $\mathbb{X}$, the space $\Cont_w([0,T];\mathbb X )$ denotes the functions on $[0,T]$ taking values in $\mathbb X$ that are continuous with respect to the weak topology of $\mathbb X$. The space $L^\infty_{w^*} ([0,T];\mathbb X^*)$ is the space of all function  on $[0,T]$ taking values in $\mathbb X^*$ that are Bochner measurable with respect to $\mathbb X^*$ equipped with the weak-stark topology and essentially bounded.  
The total variation of a function $E:[0,T]\ra \Rr$ is given by 
$$ \vert E \vert_{\text{TV}([0,T])}= \sup_{0=t_0<\ldots <t_n=T} \sum_{k=1}^n \lvert E(t_{k-1})-E(t_k) \rvert\,, $$
where the supremum is taken over all finite partitions of the interval $[0,T]$. 
We denote the space of all bounded functions of bounded variations on $[0,T]$ by~$\BV$. 
Note that the total variation of a monotone decreasing nonnegative function $E$ only depends on the initial value, \textit{i.e.,}
\begin{align*}
\vert E\vert_{\text{TV}([0,T])} = \sup_{0=t_0<\ldots <t_n=T}\sum_{k=1}^N\lvert E(t_{k-1})-E(t_k)\rvert \leq E(0) - E(T) \leq E(0) \,.
\end{align*}

\subsection{Main results}\label{ch:2}
We define the total energy as
\begin{align}
\mathcal{E}(\fat v , \fat d) 
\coloneqq
 \frac{1}{2}\lVert \fat v\rVert_{L^2(\Omega)}^ 2 + \frac{1}{2}\LTwoNorm{\nabla \fat d}^2 \,.
\end{align}
The main new idea when defining energy-variational solutions is to introduce an auxiliary variable $E$ as an upper bound of the total energy $\E(\fat v, \fat d)$ and add the difference of these two variables to a variational inequality of the energy-dissipation mechanism to close this formulation with respect to the appropriate weak topologies.  The difference between the variable $E$ and the energy $\E(\fat v , \fat d)$ can be interpreted as a measure of the difference between weak and strong convergence of the approximate solutions. The associated error term represents this difference in the limit of vanishing discretization parameters. 

\begin{definition}[Energy-variational solutions]\label{def:envar}
We call  $(\fat v, \fat d , E )$ an energy-variational solution, if 
\begin{align}
\begin{split}
&\fat v \in \Cont_w([0,T];\Ha) \cap L^2(0,T;\V)\,, \quad E \in \BV\,,\\
&\fat d \in  \Cont_w([0,T];\He) \cap H^1(0,T;L^{3/2}(\Omega;\Rr^3)  
\text{ such that }
\\&
\fat d \times \Delta \fat d \in L^2(0,T;L^2(\Omega;\Rr^3))\,,
\end{split}
\label{reg}
\end{align}
and
$E \geq \mathcal{E} (\fat v,\fat d ) $ on $[0,T]$ as well as $ \vert \fat d  \vert =1$ a.e.~in $\Omega \times (0,T)$. The term $ \fat d \times \Delta \fat d$ has to be understood as the weak divergence of $\fat d \times \nabla \fat d $.  The solution fulfills the energy-variational inequality 
\begin{multline}
\left[  E -  \int_\Omega \fat v \cdot\vv  
 \diff x \right ] \Big \vert_{s-}^t + \int_s^t   \int_\Omega \fat v \cdot \t \vv 
  -  ( \fat v \cdot \nabla) \fat v \cdot  \vv 
  \diff x \diff \tau  + \int_s^t \langle \fat T^E(\fat d)  , \vv \rangle \de \tau 
\\
+\int_s^t \int_\Omega \left ((\mu_1 + \lambda^2) (\fat d \cdot (\nabla \fat v)_{sym} \fat d) (\fat d \otimes \fat d)
    + \mu_4 (\nabla \fat v)_{sym} 
  \right ):(\nabla \fat v - \nabla \vv) \diff x \diff \tau 
    \\+\int_s^t \int_\Omega
    (\mu_5 + \mu_6 - \lambda^2) (\fat d \otimes (\nabla \fat v)_{sym}  \fat d)_{sym}: (\nabla \fat v - \nabla \vv) 
+ \vert \fat d \times \Delta \fat d \vert^2 
     \diff x \diff \tau 
    \\
    + \int_s^t \int_\Omega \left ( \lambda [\fat d \otimes [\fat d]_x^T  [\fat d]_x \cdot\Delta \fat  d]_{sym}
    +  [\fat d \otimes \Delta \fat  d]_{skw}\right ) : \nabla \vv \diff x \diff \tau 
    \\
     + \int_s^t 
 \mathcal{K}( \vv ) \left [\mathcal{E} (\fat v,\fat d ) - E \right ]  \de \tau  \leq 0 \,. \label{envarform}
\end{multline}
 for all~$s$, $t\in (0,T)$ and for all $\vv \in \Cont^1([0,T]; (\V)^*) \cap L^2(0,T;\V)$ such that $ \mathcal{K}(\vv) \in L^1(0,T)$. The initial values are fulfilled in a weak sense $\fat v(0) = \fat v_0 $, $ \fat d(0) = \fat d_0$ and $\fat d$ fulfills the inhomogeneuous Dirichlet boundary conditions in the sense of the trace $ \tr(\fat d(t)) = \tr (\fat d_\Gamma)$ for a.e.~$t\in(0,T)$. 
Additionally, it holds 
\begin{align}
\t \fat d + ( \fat v \cdot \nabla ) \fat d - (\nabla \fat v )_{\skw}\fat d + (I - \fat d \otimes \fat d) \left ( \lambda (\nabla \fat v)_{\sym}\fat d  - \Delta \fat d  \right ) = 0 \label{weak:d}
\end{align}
a.e. in $\Omega \times (0,T)$. 
In this work, we  choose $ \mathcal{K}_1(\vv) = \frac{1}{2}\Vert \vv \Vert_{L^\infty(\Omega)}^2$ and $ \mathcal{K}_2(\vv)=2 \Vert( \nabla \vv)_{\sym,-} \Vert_{\Cont(\ov \Omega)}$, where $c_L$ is given in~\eqref{lump_estim_1}, with the associated Ericksen stresses being defined as $$ \langle \fat T_1^E (\fat d) , \vv \rangle =  \int_\Omega ( \fat d \times ( \vv \cdot \nabla ) \fat d) \cdot ( \fat d \times  \Delta \fat d)  \de x \text{ or } \langle \fat T_2^E(\fat d), \vv\rangle  = -\int_\Omega (\nabla \fat d^T \nabla \fat d ): \nabla \vv \de  x \,,$$
respectively.
\end{definition}
\begin{remark}[Ericksen stress]
The main obstacle when passing to the limit in an approximation of the Ericksen--Leslie equations is often the so-called Ericksen stress, which can be seen as a kind of Korteweg stress coupling the elastic deformations of the director field to the evolution of the velocity field. 
This term is the main reason that we have to consider a generalized solution framework different from usual weak solutions. 
Via Lemma~\ref{lem:EulerLagrange} we observe that both formulations $\fat T^E_1$ and $\fat T^E_2$ are formally equivalent via the reformulation
\begin{align*}
    -  \int_\Omega ( \vv \cdot \nabla ) \fat d\cdot  [\fat d ]_{x} ^T & [\fat d ]_{x} \Delta \fat d \de  x
={} -  \int_\Omega ( \vv \cdot \nabla ) \fat d\cdot   (\Delta \fat d + \vert \nabla \fat d \vert ^2 \fat d )  \de  x\\
= {}&  \int_\Omega (\nabla \fat d^T \nabla \fat d ): \nabla \vv \de  x + \frac{1}{2}\int_\Omega ( \vv \cdot \nabla )  \vert \nabla \fat d \vert^2 - \vert \nabla \fat d \vert ^2 ( \vv \cdot \nabla ) \vert  \fat d \vert^2 \de  x\,,
\end{align*} 
where the second term on the right-hand side vanish formally due to the fact that $\vv \in \V$ and $\vert \fat d  \vert =1$. 
Thus, in the case of a regular solution with $\fat d \in L^2(0,T;H^2(\Omega; \mathbb{S} ^2))$ and $E = \mathcal E(\fat v , \fat d )$ a.e. in $ (0,T)$ both solution concepts coincide with usual weak solutions. 
\end{remark}
\begin{remark}[Strong continuity of the initial value]
From the monotony of $t \mapsto E(t)$ and the weak continuity of $t \mapsto (\fat v(t) , \fat d(t))$, we infer by the lower semi-continuity of the energy functional and the inequality $E (t) \geq \mathcal{E}(\fat v(t), \fat d(t))$ for all $t\in [0,T]$ that for any sequence $t\searrow 0$ it holds
\begin{align*}
\mathcal{E}(\fat v_0,\fat d_0) = E (0) \geq \lim_{t\searrow 0}E(t) \geq \liminf_{t \searrow 0} \mathcal{E}(\fat v (t) , \fat d(t)) \geq  \mathcal{E}(\fat v _0 ,\fat d_0)\,.
\end{align*}
This implies that all inequalities are in-fact equations such that we infer from the uniform convexity of $L^2(\Omega)$ that 
\begin{multline*}
\left \{
\begin{matrix}
(\fat v(t), \fat d(t))  \rightharpoonup& (\fat v _0, \fat d_0) \text{ as }t\searrow 0\text{ in } \Ha\times L^2(\Omega;\Rr^3)\\
\mathcal{E}(\fat v (t) , \fat d(t) ) \longrightarrow  & \mathcal{E}(\fat v_0,\fat d_0) \text{ as } t\searrow 0 \text{ in } \Rr 
\end{matrix}
\right \}
\Longrightarrow
\\
(\fat v(t), \fat d(t))  \longrightarrow  (\fat v _0, \fat d_0) \text{ as }t\searrow 0\text{ in } \Ha\times L^2(\Omega;\Rr^3)\,.
\end{multline*}
\end{remark}
\begin{remark}[Properties of generalized solutions]
Energy-variational solutions fulfill several properties, which are desirable for generalized solution concepts. 
That generalized solutions coincide with classical solutions, if the latter exists, is the so-called weak-strong uniqueness property and covered by Theorem~\ref{thm:weakstrong}. 

Furthermore, energy-variational solutions fulfill the so-called semi-flow property such that restrictions and concatinations of solutions are solutions again, which is an important property of a reasonable solution concept (\textit{cf.}~\cite{lasarzik_incompressible_fluids}). 
\end{remark}
\begin{remark}[Different definition]
We could also identify $E (t) = E^1(t) + \frac{1}{2}\Vert \fat v(t) \Vert_{L^2(\Omega)}^2 $ for a.e. $t\in (0,T)$ such that $ E^1(t) \geq \frac{1}{2} \Vert \nabla \fat d(t) \Vert_{L^2(\Omega)}^2 $ for a.e. $t\in (0,T)$ in the inequality~\eqref{envarform}. 
This can be done due to the strong convergence of $\fat v$ a.e.~in $\Omega \times (0,T)$. 
We would also replace $ \E(\fat v, \fat d ) - E = \frac{1}{2}\Vert \nabla \fat d \Vert_{L^2(\Omega)}^2 - E^1$. 
The disadvantage would be that  the adapted inequality~\eqref{envarform} is only fulfilled for a.e.~$s \in(0,T)$ and all $t\in (s,T]$. 
\end{remark}

\begin{theorem}\label{main_result}
Let $\Omega \subset \Rr^3$ be a bounded convex  polyhedral domain. Let $\bigl(\fat v_0,\fat d_0\bigr)
 \in \V \times  H^{2}(\Omega;\Rr^3) $, with $\vert \fat d_0  \vert = 1$ a.e.~in $\Omega$ and $ \fat d _\Gamma  
 = \tr(\fat d_0)$. 
Then there exists an energy-variational solution in the sense of Definition~\ref{def:envar} with $\mathcal{K}(\vv) =\frac{1}{2} \Vert \vv \Vert_{L^\infty(\Omega)}^2$ such that $ E(0)= \E(\fat v_0,\fat d_0)$. 
\end{theorem}
We are going to prove the theorem by the convergence of a fully discrete, implementable, unconditionally solvable scheme in the following sections. 
The proposed scheme in Section~\ref{sec:dis} is a fully implementable numerical scheme for the Ericksen--Leslie system that fulfills the norm restriction $\vert \fat d \vert =1$ for every approximate solution in every node of the mesh. 
We even infer that the norm of the approximate solutions  converges to $1$ with the order $1$ (\textit{cf.} Proposition~\ref{prop:unitconv}). Note that the condition $ \fat d _\Gamma  
 = \tr(\fat d_0)$ is a usual compatibility condition for parabolic problems.

Energy-variational solutions can be seen as a generalized solution concept, which generalizes the well-known weak solutions, but is finer than the usual concept of so-called measure-valued solutions~\cite{lasarzik_measure_valued}. 
\begin{definition}\label{def:meas}
We call the tuple $( \fat v ,\fat d , \fat m)$ a measure-valued solution to the Ericksen--Leslie equations~\eqref{system}, if    $\fat v\in \Cont_w ([0,T]; \Ha) \cap L^2(0,T; \V) $, ${\fat d \in  \Cont_w ( [0,T]; \He) }$, as well as $ \fat m \in  L^\infty_{w^*}(0,T; \mathcal{M}(\ov\Omega;\Rr^{3\times 3}_{\sym,+}))  $, and if~\eqref{weak:d} is fulfilled as well as the measure-valued formulation of the Navier--Stokes equation 
\begin{align}
\begin{split}\label{meas:eq}
\int_0^T\Big[  \int_\Omega \fat T_L : ( \nabla \vv )_{\sym} -  ( \fat v  \otimes  \fat v  + \nabla \fat d ^T & \nabla \fat d   ) :\nabla  \vv 
\\&
- \fat v \cdot \t \vv \de  x - 2\int_{\ov\Omega} (\nabla \vv)_{\sym} : \de \fat m ( x)  \Big]  \de s = 0
\end{split}
\end{align}
for all $\vv \in \Cont^\infty_c(\Omega\times (0,T))$ and the energy-inequality 
\begin{align}
\begin{split}\label{meeas:en}
\frac{1}{2}\left ( \Vert \fat v \Vert_{L^2(\Omega)}^2 + \Vert \nabla \fat d \Vert_{L^2(\Omega)}^2 
 + \int_{\ov{\Omega}} I : \de \fat m(x)   \right ) & \Big  \vert_0^t 
\\ + \int_0^t \int_\Omega \fat T^D :& ( \nabla \fat v )_{\sym} + \vert \fat d \times \Delta \fat d \vert^2 \de  x \de s \leq 0
\end{split}
\end{align}
for a.e.~$t\in(0,T)$. 
\end{definition}

\begin{theorem}\label{thm:meas}
Let $(\fat v, \fat d , E ) $ be an energy-variational solution with ${\mathcal{K}(\vv)=2 \Vert (\nabla \vv)_{\sym,-}\Vert_{\Cont(\ov\Omega)}}$ in the sense of Definition~\ref{def:envar}  such that $E(0)=\E(\fat v_0,\fat d_0)$. 
Then there exists a measure-valued solution $( \fat v ,\fat d , \fat m)$ in the sense of Definition~\ref{def:meas} with  $$ \frac{1}{2} \langle \fat m , I \rangle =
  \frac{1}{2} \int_{\ov\Omega} I : \de \fat m = \frac{1}{2} \int_{\ov\Omega} \de \tr(\fat m)  \leq E - \mathcal{E}(\fat v , \fat d )\quad \text{a.e.~in }(0,T)\,.$$ 
\end{theorem}
\begin{theorem}\label{thm:weakstrong}
Let  $(\fat v, \fat d , E )  
$ be an energy-variational solution in the sense of Definition~\ref{def:envar} with $E(0) = \mathcal{E}(\fat v_0, \fat d_0)$ and let $(\vv, \dd) $ fulfill
\begin{subequations}\label{regvd}
\begin{align}
\begin{split}\label{reg:v}
\vv \in{}& L^2(0,T;\V\cap L^\infty(\Omega)\cap W^{1,3}(\Omega)) \cap L^1(0,T;\Cont^1(\ov \Omega))
\\ & \quad\quad\quad\quad\quad\quad\quad\quad\quad\quad\quad\quad\quad\quad
\cap H^1(0,T;(\V)^*) 
\end{split}
\\
\dd  \in{}& L^1(0,T;W^{2,\infty}(\Omega) ) \cap L^2(0,T;W^{3,3}(\Omega)) \cap H^1(0,T;H^1(\Omega))\label{reg:d} \,
\end{align}
\end{subequations}
such that $\mathcal{K}(\vv) \in L^1(0,T)$.
Then the relative energy inequality
\begin{align}
\begin{split}
 \frac{1}{2}&\left [  \Vert \fat v (t)- \vv(t) \Vert_{L^2(\Omega)}^2 + \Vert \nabla \fat d(t) - \nabla \dd(t) \Vert_{L^2(\Omega)}^2  \right ]
\\
+&\frac{1}{2}\int_0^t \int_\Omega (\mu_1 + \lambda^2) (\fat d \cdot (\nabla \fat v- \nabla \vv)_{sym} \fat d) ^2
    + \mu_4 \vert (\nabla \fat v-\nabla \vv)_{sym} \vert^2
  \diff x \expon^{\int_s^t \mathcal{Q}(\vv,\dd)\de \tau  }\diff s 
    \\+&\frac{1}{2}\int_0^t \int_\Omega
    (\mu_5 + \mu_6 - \lambda^2) \vert (\nabla \fat v-\nabla \vv)_{sym}  \fat d\vert^2 
+ \vert \fat d \times \Delta (\fat d - \dd) \vert^2 
     \diff x  \expon^{\int_s^t \mathcal{Q}(\vv,\dd)\de \tau  }\diff s 
     \\
     +&\int_0^t \left \langle 
      \t \vv + ( \vv \cdot\nabla) \vv + \di ( \nabla \dd^T \nabla \dd) - \di \tilde{\fat T}^L , \fat v -\vv 
     \right \rangle \expon^{\int_s^t \mathcal{Q}(\vv,\dd)\de \tau  }\diff s 
     \\
     +&\int_0^t \left \langle 
    \nabla \left (  \t \dd +(\vv \cdot \nabla) \dd - ( \nabla \vv )_{\skw} \dd 
    \right ) ,\nabla (\fat d - \dd)
     \right \rangle \expon^{\int_s^t \mathcal{Q}(\vv,\dd)\de \tau  }\diff s 
     \\
     +&\int_0^t \left \langle 
    \nabla \left (  
    ( I - \dd \otimes \dd )( \lambda (\nabla\vv)_{\sym} - \Delta \dd ) \right ) ,\nabla (\fat d - \dd)
     \right \rangle \expon^{\int_s^t \mathcal{Q}(\vv,\dd)\de \tau  }\diff s 
     \\
     \leq& 
     \frac{1}{2}\left [  \Vert \fat v _0- \vv(0) \Vert_{L^2(\Omega)}^2 + \Vert \nabla \fat d_0 - \nabla \dd(0) \Vert_{L^2(\Omega)}^2  \right ]
      \expon^{\int_0^t \mathcal{Q}(\vv,\dd)\de s  }
   \,, 
   \end{split}
\label{enineq}
\end{align}
holds for a.e.~$t\in(0,T)$,
where  
\begin{align*}
\mathcal{Q}(\vv,\dd) :=&  \Vert \vv \Vert_{L^\infty(\Omega)} ^2 + \Vert \nabla \vv \Vert_{L^\infty(\Omega)} + \Vert \Delta  \dd \Vert_{L^3(\Omega)}^2+\Vert\nabla \vv \Vert_{L^3(\Omega)}^2\\&+ \Vert \Delta \dd \Vert_{L^\infty(\Omega)}+\Vert  \Delta \dd \Vert_{W^{1,3}(\Omega)}^2+\mathcal{K}(\vv)\,.
\end{align*} 
Additionally, if $(\vv,\dd)$ is
  a classical solution to the Ericksen--Leslie system on $(0,T^*)$, then both solutions coincide on this time-interval, \textit{i.e.},
\begin{align*}
\fat v \equiv \vv \qquad \fat d \equiv \dd \quad\text{ a.e.~on~}(0,T^*)\,.
\end{align*}
\end{theorem}
\begin{remark}[Dissipative solutions]
Theorem~\ref{thm:weakstrong} does not only show the weak-strong uniqueness of energy-variational solutions, it also shows that every energy-variational solution is a dissipative solution. A dissipative solution is defined via the relative-energy inequality~\eqref{enineq}, \textit{i.e.,} $(\fat v , \fat d)$ is called a dissipative solution, if it fulfills the relative energy-inequality~\ref{enineq} for all function~$(\vv,\dd)$ fulfilling the regularity properties~\eqref{regvd}. The concept of energy-variational solutions is thus strictly finer than the concept of dissipative solutions. 
\end{remark}
\subsection{Proofs for the continuous system\label{sec:ana}}
\begin{proof}[Proof of Theorem~\ref{thm:meas}]
Let $(\fat v , \fat d , E)$ be an energy-variational solution according to Definition~\ref{def:envar} with $ \mathcal{K}(\vv) =2 \Vert \nabla \vv \Vert_{\Cont(\ov\Omega)}$.
Then, we may choose $\vv = 0$ in order to infer
\begin{align}
E \Big\vert_0^t + \int_0^t \int_\Omega (\mu_1+\lambda^2) (\fat d \cdot (\nabla \fat v )_{\sym} \fat d )^2 + \mu_4 \vert (\nabla \fat v )_{\sym}\vert^2 \diff  x \diff \tau &\notag \\
+\int_0^t \int_\Omega  (\mu_5+\mu_6 -\lambda^2)\vert (\nabla \fat v )_{\sym} \fat d \vert^2 + \vert \fat d \times \Delta \fat d \vert^2 \diff  x \diff \tau &{}\leq 0 \,.\label{enneu}
\end{align}
Furthermore, we may choose $ \vv = \alpha \tu$ with $\alpha \geq 0$ 
and multiply the energy-variational inequality~\eqref{envarform} by $1/\alpha$. 
In the limit $\alpha \ra \infty$, we find
\begin{multline*}
- \int_\Omega \fat v \cdot \tu \diff x \Big\vert_0^t + \int_0^t\int_\Omega \fat v \cdot \t \tu +  \left[ \fat v \otimes \fat v + \nabla \fat d^T \nabla \fat d - \fat T_L \right] : \nabla \tu \diff x\diff s \\  +\int_0^t 2\Vert(\nabla \tu)_{\sym,-} \Vert_{\Cont(\ov\Omega;\Rr^{3\times 3})} \left [\mathcal{E}(\fat v , \fat d ) - E  \right ]  \diff s \leq 0 \,
\end{multline*}
for all  $\tu \in \Cont^1([0,T])\otimes \Cont^1_0 (\Omega)$.
By the definition of the weak time derivative, we infer that
\begin{align*}
-\int_0^T  \langle \t \fat v ,{}& \tu \rangle\de t  = - \int_\Omega \fat v \cdot \tu \diff x \Big\vert_0^T+  \int_0^T \int_\Omega \fat v \cdot \t \tu \diff x \diff t \\
\leq{}&  \left ( \Vert \fat v \Vert_{L^\infty(0,T;L^2(\Omega))}^2 + \Vert \nabla \fat d \Vert_{L^\infty(0,T;L^2(\Omega))}^2 \right )  \Vert \nabla \tu\Vert_{L^1(0,T;\Cont(\ov\Omega))} 
\\+& 2\Vert\nabla \tu \Vert_{L^1(0,T;\Cont(\ov\Omega))} \left \Vert E- \mathcal{E}(\fat v , \fat d )   \right \Vert_{L^\infty(0,T)}   + \Vert \fat T_L \Vert_{L^2(0,T;L^2(\Omega ))} \Vert \nabla \tu \Vert _{L^2(0,T;L^2(\Omega))}\,.
\end{align*}
First this is only well defined for $\tu \in \Cont^1([0,T])\otimes \Cont^1_0 (\Omega)$. 
Since this is a linear subspace of the Banach space $L^1(0,T;\Cont^1_0(\Omega)) \cap L^2(0,T;H^1(\Omega))$, we infer from the Hahn--Banach theorem~\cite[Thm.~1.1]{brezis} the existence of an element
$$ \t \fat v \in \left ( L^1(0,T;\Cont^1_0(\Omega))\cap L^2(0,T;H^1(\Omega))\right )^* = L^\infty_{w^*} (0,T; (\Cont^1_0 (\Omega))^*) + L^2(0,T;W^{-1,2} (\Omega)) \,.$$
This allows to write the  linear form 
\begin{multline}
\langle \fat l(t) , \tu(t) \rangle_{
L^\infty_{w^*}(0,T;
(\Cont^1_0(\Omega))^*
),L^1(0,T;
\Cont^1_0(\Omega)
)
} : ={}
\\- 
\int_0^T 
\langle \t\fat v(t) , \tu(t) \rangle 
\diff  t 
+ 
\int_0^T
 \int_\Omega \left [ \fat v(t) \otimes \fat v(t) + \nabla \fat d ^T(t) \nabla \fat d(t)- \fat T_L(t)\right  ] : \nabla \tu(t) \diff x  \diff t
  \\ \leq{}  2 \int_0^T\Vert (\nabla \tu (t))_{\sym,-}\Vert_{ L^1(0,T;
  \Cont(\ov\Omega;\Rr^{3\times 3} )
  )
   } 
 \left [ E (t)- \mathcal{E}( \fat v (t), \fat d(t) )\right ]\de t 
    \,\label{def:l}
\end{multline}
for a.e.~$t\in(0,T)$. 
This implies that the linear form $\fat l \in L^\infty_{w^*} (0,T;(\Cont^1_0(\Omega))^*)$ can be identified via considering the  gradient operator $$\nabla  : L^1(0,T;\Cont^1_0 (\Omega ) )\ra L^1(0,T;\Cont(\ov\Omega; \Rr^{3\times 3}  ))\,,$$ we set $ G = \nabla (L^1(0,T;\Cont^1_0(\Omega)))$ the range of the gradient, the set of all gradient fields equipped with the norm of $L^1(0,T;\Cont(\ov \Omega;\Rr^{3\times 3}))$. Note that this mapping is injective due to the prescribed Dirichlet boundary conditions. 
We may set $ S = \nabla ^{-1} : G \ra L^1(0,T;\Cont^1_0(\Omega))$ such that for $h \in G$ the functional $h  \mapsto \langle \fat l , S h \rangle $ is a continuous linear functional on $G$. By the Hahn--Banach theorem~\cite[Thm.~1.1]{brezis}, this linear functional can be extended to a continuous linear functional $\Phi$ on $(L^1(0,T;\Cont(\ov\Omega;\Rr^{3\times 3})))^* = L^\infty_w (0,T;\mathcal{M}(\ov\Omega;\Rr^{3\times 3}))$ with \begin{align}
\langle \Phi (t) ; \fat A \rangle \leq 2 \Vert (\fat A)_{\sym,-}\Vert_{\Cont(\ov\Omega;\Rr^{3\times 3})} [ E (t)- \mathcal{E}( \fat v(t) , \fat d (t))]  \,\label{estMeas}
\end{align} 
for all $\fat A\in \Cont(\ov\Omega;\Rr^{3\times 3})$ a.e.~in $(0,T)$,  where the estimate follows from~\eqref{def:l} and Hahn--Banach's theorem written in its point-wise a.e.-form. 
By the Riesz representation theorem, we know that there exists a measure $\fat m\in L^\infty_{w^*}(0,T;\mathcal{M}(\ov\Omega; \Rr^{3\times 3})$ such that 
\begin{align*}
\langle \fat l(t) , \tu (t)\rangle_{
L^\infty_w(0,T:
(\Cont^1_0(\Omega))^*,
),L^1(0,T;
\Cont^1_0(\Omega)
)
}& = \langle \Phi (t), \nabla \tu(t) \rangle _{
L^\infty_w(0,T;
(\Cont(\ov\Omega))^*,
),L^1(0,T;
\Cont(\ov\Omega)
) 
}\\& =-\int_0^T  \int_{\ov\Omega} \nabla \tu(t) : \diff \fat m_t  \de t \,,
\end{align*}
which implies the formulation~\eqref{meas:eq} by the definition of $\fat l$ in~\eqref{def:l}. 
From the estimate~\eqref{estMeas}, we infer that the measure~$\fat m$ vanishes for all continuous functions with values in the skew-symmetric matrices, such that $ \fat m \in L^\infty_w(0,T;\mathcal{M}(\ov \Omega ; \Rr^{3\times 3}_{\sym}))$. Additionally, \eqref{estMeas} allows to insert $ \fat A = \fat \xi \otimes \fat \xi \phi$ for any $\fat \xi \in \Rr^3$ and $\phi \in \Cont(\ov\Omega; [0,\infty))$ such that $\int_{\ov\Omega}\phi  \fat \xi \otimes \fat \xi : \de \fat m_t \geq 0 $ for all $\phi \in \Cont(\ov\Omega; [0,\infty))$, which is the definition of a positive definite measure $ \fat m \in L^\infty_w(0,T;\mathcal{M}(\ov \Omega ; \Rr^{3\times 3}_{\sym,+}))$.  
From the estimate in~\eqref{estMeas}, we additionally infer by choosing $ \Phi (t) = \phi(t) I$ for $\phi \in\Cont^1_0([0,T])$ with $\phi \geq 0 $ a.e.~in $(0,T)$ that 
\begin{align*}
\int_0^T \phi (t) \int_{\ov\Omega} I:  \de \fat m  (x) \de t=\int_0^T \phi (t) \int_{\ov\Omega}  \tr (\fat m)  ( x) \de t   \leq 2 \int_0^T \phi(t) [  E(t) - \mathcal{E}(\fat v (t) , \fat d (t)) ] \de t \,.
\end{align*}
This implies that $\langle \fat m , I \rangle \leq  2 [E - \E(\fat v, \fat d )] $ 
for a.e.~$t\in(0,T)$, which in turn implies together with~\eqref{enneu} the inequality~\eqref{meeas:en}. 

Note that we use the usual spectral-norm for symmetric matrices $\vert A\vert_2 = \max_{i\in\{1,2,3\}} \vert \lambda_i\vert $, where $\lambda_i$ denote the eigenvalues of the matrix $A\in \Rr^{3\times 3}_{\sym}$. The dual-norm with respect to the Frobenius product $:$ is given by $ \vert A\vert_2' =\sum_{i=1}^3 \vert \lambda_i\vert $ for $A \in \Rr^{3\times 3}$. For a positive definite matrix all eigenvalues are non-negative, such that we may write $ \vert A\vert_2' = \sum_{i=1}^3 \lambda_i = \tr(A)=A:I$. This implies that we may identify the norm $\vert A\vert_2' = I : A $. 

\end{proof}
\begin{proof}[Proof of Theorem~\ref{thm:weakstrong}]
We observe that~\eqref{weak:d} tested with $\Delta\dd$ and integrated-in-time gives
\begin{align}
\int_0^t \int_\Omega \left ( \t \fat d + ( \fat v \cdot \nabla) \fat d - ( \nabla \fat v)_{\skw} \fat d + ( I - \fat d \otimes \fat d ) ( \lambda (\nabla \fat v )_{\sym} \fat d - \Delta \fat d ) \right ) \Delta \dd \de  x \de s = 0 \,.\label{tested:d}
\end{align}
Additionally, the equation~\eqref{system_a} for the strong solution tested with $\vv -\fat v$ and~\eqref{system_c} for the strong solution tested with  $\Delta (\fat d - \dd )$ implies
\begin{align}
0 {}&=\frac{1}{2} \Vert \vv \Vert_{L^2(\Omega)}^2 \Big\vert_0^t - \int_0^t \int_\Omega \t \vv \cdot \fat v +  (\vv \cdot \nabla) \vv\cdot \fat v  - \nabla \dd ^T \nabla \dd : \nabla \fat v  + \mu_4 (\nabla \vv)_{\sym} :(\nabla \fat v - \nabla \vv ) \de  x \de s \notag
\\
&- \int_0^t \int_\Omega (\mu_1+\lambda^2) ( \dd \cdot (\nabla \vv)_{\sym}\dd )(\dd \otimes \dd ):(\nabla \fat v - \nabla \vv ) \de  x \de s \notag
\\
&- \int_0^t \int_\Omega (\mu_5+\mu_6-\lambda^2 ) ( \dd \otimes (\nabla \vv)_{\sym} \dd ):(\nabla \fat v - \nabla \vv ) \de  x \de s \notag
\\
&- \int_0^t \int_\Omega \lambda\left  ( (\nabla \fat v )_{\sym} \dd \cdot (I - \dd \otimes \dd) \Delta \dd + \nabla\left ( (I- \dd \otimes \dd )(\nabla \vv)_{\sym} \dd\right ) :  \nabla \fat d \right ) \de  x \de s \notag
\\
& + \frac{1}{2}\Vert \nabla \dd\Vert_{L^2(\Omega)}^2 \Big\vert_0^t -\int_0^t \int_\Omega\nabla \left( (I-\dd \otimes \dd ) \Delta \dd  \right): ( \nabla \fat d - \nabla \dd )   \diff x \diff t  \notag 
\\& - \int_0^t \int_\Omega \t \nabla \dd : \nabla \fat d + \nabla \left ((\vv \cdot \nabla) \dd\right ):  \nabla \fat d - (\nabla \fat v )_{\skw} \dd \cdot \Delta \dd - \nabla\left ( (\nabla \vv)_{\skw} \dd \right): \nabla  \fat d \de   x \de s\notag
 \\
     &+\int_0^t \left \langle 
      \t \vv + ( \vv \cdot\nabla) \vv + \di ( \nabla \dd^T \nabla \dd) - \di \tilde{\fat T}^L , \fat v -\vv 
     \right \rangle \diff s \notag
     \\
     &+\int_0^t \left \langle 
    \nabla \left (  \t \dd +(\vv \cdot \nabla) \dd - ( \nabla \vv )_{\skw} \dd + ( I - \dd \otimes \dd )( \lambda (\nabla\vv)_{\sym}\dd  - \Delta \dd ) \right ) ,\nabla (\fat d - \dd)
     \right \rangle \diff s 
 \,.\label{tested:strong}
\end{align}
We observe that 
\begin{align*}
- \int_0^t \int_\Omega \t \nabla \dd : \nabla \fat d - \t \fat d \cdot \Delta \dd \de  x \de s  =\int_\Omega \nabla \fat d : \nabla \dd \de  x \Big\vert_0^t\,.
\end{align*}
Adding up~\eqref{envarform},~\eqref{tested:d}, and~\eqref{tested:strong} implies
\begin{align}
 [ E - &\mathcal{E}(\fat v , \fat d )  ]\Big\vert_{0-}^t +\frac{1}{2}\left [  \Vert \fat v - \vv \Vert_{L^2(\Omega)}^2 + \Vert \nabla \fat d - \nabla \dd \Vert_{L^2(\Omega)}^2  \right ]\Big\vert_0^t \notag
\\
+&\int_0^t \int_\Omega (\mu_1 + \lambda^2) (\fat d \cdot (\nabla \fat v- \nabla \vv)_{sym} \fat d) ^2
    + \mu_4 \vert (\nabla \fat v-\nabla \vv)_{sym} \vert^2
  \diff x \diff s \notag
    \\+&\int_0^t \int_\Omega
    (\mu_5 + \mu_6 - \lambda^2) \vert (\nabla \fat v-\nabla \vv)_{sym}  \fat d\vert^2 
+ \vert \fat d \times \Delta (\fat d - \dd) \vert^2 
     \diff x \diff s \notag
      \\
     +&\int_0^t \left \langle 
      \t \vv + ( \vv \cdot\nabla) \vv + \di ( \nabla \dd^T \nabla \dd) - \di \tilde{\fat T}^L , \fat v -\vv 
     \right \rangle\diff s \notag
     \\
     +&\int_0^t \left \langle 
    \nabla \left (  \t \dd +(\vv \cdot \nabla) \dd - ( \nabla \vv )_{\skw} \dd + ( I - \dd \otimes \dd )( \lambda (\nabla\vv)_{\sym} - \Delta \dd ) \right ) ,\nabla (\fat d - \dd)
     \right \rangle \diff s \notag
     \\
     \leq{}&
      \int_0^t \int_\Omega ( \fat v \cdot\nabla ) \fat v \cdot \vv + ( \vv \cdot \nabla) \vv \cdot \fat v 
\de  x \de s \notag
     \\
     &+\int_0^t \int_\Omega     \fat T^E(\fat d) : \nabla \vv  - \nabla \dd ^T \nabla \dd : \nabla \fat v - ( \fat v \cdot \nabla )\fat d \cdot \Delta \dd +\nabla \left ((\vv \cdot \nabla) \dd\right ):  \nabla \fat d \de  x \de s \notag
     \\
     &-\int_0^t \int_\Omega ( \mu_1+ \lambda^2 ) \left (\fat d \cdot (\nabla \vv)_{\sym} \fat d (\fat d \otimes \fat d ) - \dd \cdot ( \nabla \vv)_{\sym\dd } (\dd \otimes \dd )  \right ): ( \nabla \fat v -\nabla \vv)_{\sym} \de  x \de s \notag
     \\
     &- \int_0^t \int_\Omega (\mu_5+\mu_6-\lambda^2) \left ( (\nabla \vv)_{\sym}( \fat d \otimes \fat d - \dd \otimes \dd ) : (\nabla \fat v -\nabla \vv )_{\sym} \right ) \de  x \de s \notag
     \\
     &-\int_0^t \int_\Omega\nabla \left[ \left ( \fat d \otimes \fat d - \dd \otimes \dd \right ) \Delta \dd \right]: \left (  \nabla \fat d - \nabla \dd  \right ) \de  x \de s \notag
     \\
     &-  \lambda\int_0^t \int_\Omega\left ((I - \fat d \otimes \fat d)(\nabla \fat v )_{\sym} \fat d -  (I - \dd \otimes \dd)(\nabla \fat v )_{\sym} \dd \right )\cdot  \Delta \dd \de  x \de s\notag
     \\
      &-  \lambda \int_0^t \int_\Omega\nabla \fat d  : \nabla \left (  (I- \fat d \otimes \fat d ) (\nabla \vv)_{\sym} \fat d -   (I- \dd \otimes \dd ) (\nabla \vv)_{\sym}\dd  \right )\de  x \de s \notag
  \\
     &+\int_0^t \int_\Omega \nabla  ((\nabla \vv)_{\skw}(\fat d - \dd )): \nabla  \fat d +( ( \nabla \fat v )_{\skw} (\fat d - \dd) )\cdot \Delta \dd \de  x \de s\notag \\
     &\hspace*{5cm}
     + 
 \int_0^t \mathcal{K}( \vv ) \left [E- \mathcal{E} (\fat v,\fat d )  \right ]  \de s  \,.\label{together}
\end{align}
We recall the identity, which is a special case of the identity (6.7) in~\cite{weakstrongmeas},
\begin{align*}
0= \int_\Omega \nabla \vv : ( \nabla \dd^T \nabla \fat d ) + \nabla \vv : (\nabla \fat d^T \nabla \dd) + ( \vv \cdot \nabla) \fat d \cdot \Delta \dd - \nabla \left ((\vv \cdot \nabla) \dd\right ):  \nabla \fat d \de  x \,.
\end{align*}
this identity helps to transform the second line on the right-hand side of~\eqref{together} in the case of $\fat T^E = \fat T^E_2$ 
into 
\begin{align*}
-\int_\Omega     & \nabla \fat d ^T \nabla \fat d : \nabla \vv + \nabla \dd ^T \nabla \dd : \nabla \fat v + ( \fat v \cdot \nabla )\fat d \cdot \Delta \dd -\nabla \left ((\vv \cdot \nabla) \dd\right ):  \nabla \fat d \de  x
\\
&=-\int_\Omega  (\nabla \fat d - \nabla \dd) ^T (\nabla \fat d-\nabla \dd)  : \nabla \vv  +( \nabla \dd^T \nabla \dd) : (\nabla \fat v - \nabla \vv) + (( \fat v - \vv ) \cdot \nabla) \fat d \cdot \Delta \dd \de  x 
\\
&=-\int_\Omega  (\nabla \fat d - \nabla \dd) ^T (\nabla \fat d-\nabla \dd)  : \nabla \vv  + (( \fat v - \vv ) \cdot \nabla) (\fat d - \dd ) \cdot \Delta \dd \de  x \,.
\end{align*}
In case of $\fat T^E=\fat T^E_1$ we may observe from the uni-norm restriction  and additional manipulations that
\begin{align*}
    \int_\Omega  &    (\vv\cdot \nabla )\fat d [\fat d ]_{x}[\fat d]_{x}\Delta \fat d   - \nabla \dd ^T \nabla \dd : \nabla \fat v - ( \fat v \cdot \nabla )\fat d \cdot \Delta \dd +\nabla \left ((\vv \cdot \nabla) \dd\right ):  \nabla \fat d \de  x \qquad\qquad\qquad
    \\
    ={}&  \int_\Omega      (\vv\cdot \nabla )\fat d [\fat d ]_{x}^T [\fat d]_{x} \Delta \fat d   + \left[ [\dd]_{x}^T[\dd]_x\nabla \dd   - [ \fat d ]_{x}^T [\fat d ]_{x}\nabla \fat d  \right] \cdot (\Delta \dd\otimes \fat v) \de  x
    \\& +\int_\Omega  \nabla \left ((\vv \cdot \nabla) \dd [\dd]_{x}^T[\dd]_x \right ):  \nabla \fat d \de  x
\\ ={}&  \int_\Omega      (\vv\cdot \nabla )(\fat d-\dd) [\fat d ]_{x}^T [\fat d]_{x} (\Delta \fat d- \Delta \dd)    - ( (\fat v-\vv) \cdot \nabla )(\fat d-\dd)  \cdot [ \fat d ]_{x}^T [\fat d ]_{x}  \Delta \dd  \de  x
\\ & - \int_\Omega     \nabla \left[  (\vv\cdot \nabla )\dd( [\fat d ]_{x}^T [\fat d]_{x} -  [\dd ]_{x}^T [\dd]_{x})\right]: (\nabla \fat d- \nabla \dd) 
\de   x \\ &+ \int_{\Omega} 
( (\fat v-\vv) \cdot \nabla )\dd  \cdot( [\fat d ]_{x}^T [\fat d]_{x} -  [\dd ]_{x}^T [\dd]_{x}) \Delta \dd  \de  x\,.
\end{align*}
The resulting terms can now be estimated in terms of the relative energy and the relative dissipation, \textit{i.e.,} the terms on the left-hand side. Similar arguments have already been applied in~\cite{weakstrongmeas} and~\cite{Lasarzik_opt_control}. 
We remark, that in the term in the fifth line on the right-hand side of~\eqref{together}, we have to apply the product rule to infer
\begin{multline*}
-\int_\Omega \nabla \left[ \left ( \fat d \otimes \fat d - \dd \otimes \dd \right ) \Delta \dd \right]:  \left (\nabla \fat d - \nabla \dd  \right ) \de  x
\\
=-\int_\Omega ( \nabla \fat d - \nabla  \dd ) : \nabla \left ( \fat d \otimes \fat d - \dd \otimes \dd \right ) \cdot \Delta \dd +\left (( \nabla \fat d - \nabla  \dd ) ^T\left ( \fat d \otimes \fat d - \dd \otimes \dd \right ) \right ): \nabla \Delta \dd  \de  x\,.
\end{multline*}

In order to prove the weak-strong uniqueness result, the terms on the right-hand side of~\eqref{together} have to estimated by terms on the left-hand side and in the end Gronwall's inequality is applied. 
Since most of the estimates are rather standard and done using H\"older's and Young's inequality, we refrain from displaying it in full detail here. 
The interested reader may consult the article~\cite{weakstrongmeas} for more details, where these  estimates were performed for a more involved system. 

Estimating all terms appropriately, we find that there exists a constant $C>0$ such that
\begin{align}
 [ E -& \mathcal{E}(\fat v , \fat d )  ]\Big\vert_{0-}^t +\frac{1}{2}\left [  \Vert \fat v - \vv \Vert_{L^2(\Omega)}^2 + \Vert \nabla \fat d - \nabla \dd \Vert_{L^2(\Omega)}^2  \right ]\Big\vert_0^t \notag
\\
+&\frac{1}{2}\int_0^t \int_\Omega (\mu_1 + \lambda^2) (\fat d \cdot (\nabla \fat v- \nabla \vv)_{sym} \fat d) ^2
    + \mu_4 \vert (\nabla \fat v-\nabla \vv)_{sym} \vert^2
  \diff x \diff s \notag
    \\+&\frac{1}{2}\int_0^t \int_\Omega
    (\mu_5 + \mu_6 - \lambda^2) \vert (\nabla \fat v-\nabla \vv)_{sym}  \fat d\vert^2 
+ \vert \fat d \times \Delta (\fat d - \dd) \vert^2 
     \diff x \diff s \notag \\
     +&\int_0^t \left \langle 
      \t \vv + ( \vv \cdot\nabla) \vv + \di ( \nabla \dd^T \nabla \dd) - \di \tilde{\fat T}^L , \fat v -\vv 
     \right \rangle\diff s \notag
     \\
     +&\int_0^t \left \langle 
    \nabla \left (  \t \dd +(\vv \cdot \nabla) \dd - ( \nabla \vv )_{\skw} \dd + ( I - \dd \otimes \dd )( \lambda (\nabla\vv)_{\sym} - \Delta \dd ) \right ) ,\nabla (\fat d - \dd)
     \right \rangle \diff s \notag
     \\
     \leq{}&
    C  \int_0^t \left ( \Vert \vv \Vert_{L^\infty(\Omega)} ^2 + \Vert \nabla \vv \Vert_{L^\infty(\Omega)} + \Vert \Delta  \dd \Vert_{L^3(\Omega)}^2 \right )\left [  \Vert \fat v - \vv \Vert_{L^2(\Omega)}^2 + \Vert \nabla \fat d - \nabla \dd \Vert_{L^2(\Omega)}^2  \right ] \de s \notag
\\
    &+C  \int_0^t \left (  \Vert\nabla \vv \Vert_{L^3(\Omega)}^2+ \Vert \Delta \dd \Vert_{L^\infty(\Omega)}+\Vert  \Delta \dd \Vert_{W^{1,3}(\Omega)}^2  \right )\left [  \Vert \fat v - \vv \Vert_{L^2(\Omega)}^2 + \Vert \nabla \fat d - \nabla \dd \Vert_{L^2(\Omega)}^2  \right ] \de s \notag
    \\
   &   + \int_0^t
 \mathcal{K}( \vv ) \left [E- \mathcal{E} (\fat v,\fat d )  \right ]  \de s  \,.\label{estGron}
\end{align}
Gronwall's inequality implies the relative energy inequality~\eqref{enineq}. Choosing $(\vv,\dd)$ to be a solution, implies that the fourth and fifth line of~\eqref{enineq} vanish. Since all terms on the left-hand side of~\eqref{enineq} are non-negative, if $\fat v_0 = \vv(0) $ and $\fat d _0 = \dd(0)$ the right-hand side of~\eqref{enineq} is zero such that 
$\fat v(t) = \vv(t)$ and $\fat d(t) = \dd(t)$ for a.e.~$t\in(0,T^*)$, which implies the assertion of Theorem~\ref{thm:weakstrong}. 

\end{proof}
\begin{remark}[Weak-strong uniqueness extended]
It may seem strange that all calculations in the above proof are conducted on the time interval $(0,t)$, when the definition of energy-variational solutions~\ref{def:envar} is formulated on every interval $(t,s)$ for all $s<t\in(0,T)$. Indeed, the calculations of the above proof can also be carried out for all intervals $(s,t)$ such that we may infer inequality~\eqref{estGron} with $0$ replaced by $s$. 
Thus, the weak-strong uniqueness result can be sharpened: If there exists an $s\in(0,T)$ such that there exists a solution $(\vv,\dd)$ fulfilling the regularity assumptions~\eqref{regvd} on $(s,T)$ with $ \fat v (s) = \vv(s)$, $ \fat d(s) = \dd(s)$ and additionally $\mathcal E(\vv(s),\dd(s)) = E(s)$, then it holds
\begin{align*}
    \fat v\equiv \vv\,, \qquad \fat d \equiv \dd \qquad \text{on }(s,T)\,.
\end{align*}
\end{remark}

\subsection{Auxiliary Lemmata\label{sec:aux}}
\begin{lemma}\label{lem:EulerLagrange}
Let $\fat d \in L^\infty (\Omega)\cap H^1(\Omega)$ with $\fat d \times \Delta \fat d \in L^2(\Omega)$ and $\vert \fat d ( x ,t)\vert = 1 $ a.e.~in $\Omega\times (0,T)$ such that for the trace $\vert \tr(\fat d)\vert = 1$ a.e.~on $\partial \Omega$.  Then we can Identify the terms $-  \fat d \times ( \fat d \times \Delta  \fat d)= [\fat d ]_{x} ^T [\fat d ]_{x} \Delta \fat d = (I-\fat d \otimes \fat d )\Delta \fat d  = \Delta  \fat d + \vert \nabla \fat d \vert^2 \fat d $ in $L^2(\Omega)$, which implies 
\begin{align*}
    \int_\Omega \vert \fat d \times \Delta\fat d \vert ^2 \de  x = \int_\Omega \vert \Delta  \fat d + \vert \nabla \fat d \vert^2 \fat d \vert ^2 \de  x \,.
\end{align*}
\end{lemma}
\begin{remark}
We note that the notation $\Delta \fat d + \vert \nabla \fat d\vert^2 \fat d $ is the usual formulation in the Euler--Lagrange equation of harmonic maps into the sphere~\cite{sphere}. The norm of this term appears as a dissipative term in the energy-dissipation mechanism of the system (see~\eqref{envarform} below). Formally this implies that the solution to such a system is forced to approximately fulfill the Euler--Lagrange equation of the Dirichlet energy minimized over all functions with values in the sphere, in the large-time limit. 
\end{remark}
\begin{proof}
For all $\fat \varphi \in \He$, we find by an integration-by-parts and  $ \nabla \vert \fat d\vert^2 = 0$ a.e. in~$\Omega $ as well as  a.e. on~$\partial \Omega $  that
\begin{align*}
\int_\Omega \fat \varphi  \cdot  \left ( I - \fat d \otimes \fat d \right ) \Delta \fat d \de x 
= {} &  - \int_\Omega\nabla \fat \varphi : ( I - \fat d \otimes \fat d ) \nabla \fat d - \fat \varphi \cdot \fat d  \vert \nabla \fat d\vert^2 - (\nabla \fat d ^T\fat \varphi)\cdot(\nabla \fat d^T \fat d ) \de  x 
\\&+ \int_{\partial \Omega } \fat \varphi \cdot ( I - \fat d \otimes \fat d ) \nabla \fat d \cdot \fat n \de S 
\\
={}&-\int_\Omega \nabla \fat \varphi : \nabla \fat d -  \fat \varphi \cdot  \vert \nabla \fat d\vert^2  \fat d \de  x -\frac{1}{2}\int_\Omega  \nabla ( \fat \varphi \cdot \fat d) \cdot \nabla \vert \fat d \vert^2 \de  x 
\\&+ \int_{\partial \Omega } \fat \varphi \cdot \nabla \fat d \cdot \fat n - \frac{1}{2}\fat \varphi \cdot \fat d \fat n \cdot \nabla \vert \fat d\vert^2 \de S
\\ ={}& \int_\Omega \fat \varphi \cdot ( \Delta \fat d + \vert \nabla \fat d \vert^2 \fat d ) \de  x
 \,,
\end{align*}
which implies the equivalence of the terms in $L^2(\Omega)$. The unit-norm constraint implies $( I -\fat d\otimes \fat d ) = ( I -\fat d\otimes \fat d )( I -\fat d\otimes \fat d )$ such that we infer the second equality
\begin{multline*}
    \int_\Omega \vert \fat d \times \Delta\fat d \vert ^2 \de  x = \int_\Omega \Delta \fat d \cdot ( I -\fat d\otimes \fat d ) \Delta \fat d  \de  x= \int_\Omega \left ( ( I -\fat d\otimes \fat d ) \Delta \fat d \right) \cdot( I -\fat d\otimes \fat d ) \Delta \fat d  \de  x\\  = \int_\Omega \vert \Delta  \fat d + \vert \nabla \fat d \vert^2 \fat d\vert ^2 \de  x \,.
\end{multline*}
This implies the assertion. 
\end{proof}

\begin{lemma}[{\cite[Lemma 2.4]{lasarzik_incompressible_fluids}}]\label{lem:invar}
Let $f\in L^1(0,T)$ and $g\in L^\infty(0,T)$ with $g\geq 0$ a.e.~in $(0,T)$.
Then the two inequalities 
\begin{align*}
-\int_0^T \phi'(t) g(t) \de t  + \int_0^T \phi(t) f(t) \de t \leq 0 
\end{align*}
for all $\phi \in {\Cont}^1_c ((0,T))$ with $\phi\geq 0$ for all $t\in (0,T)$ and 
\begin{align}
g(t) -g(s) + \int_s^t f(\tau) \de \tau \leq 0 \quad \text{for a.e.~}t,\, s\in(0,T), s\leq t \,\label{ineq2}
\end{align}
are equivalent. 

Moreover, from \eqref{ineq2} we can infer that there exists a function $h\in \BV$ such that $h = g $ a.e.~in $(0,T)$ and the inequality~\eqref{ineq2} holds for every $s,t\in (0,T)$  with $g $ replaced by $h$. 
\end{lemma}

\subsection{Finite element spaces}\label{sec:fem_spaces}
For the rest of this paper we presume the following assumptions.
\begin{enumerate}
    \item\label{ass1} The domain $\Omega \subset \Rr^3$ is a convex polyhedron.
    \item\label{ass2} The subdivision $\mathcal{T}^h$ of 
    tetrahedra 
    of our domain $\Omega \subset \Rr^3$ is quasi-uniform (in the sense of \cite[4.4.15]{Scott_Brenner_FEM}). 
    \item\label{ass3} All $T\in \mathcal{T}^h$ have at least one node that is not on the boundary $\partial \Omega$.
    \item\label{ass4} Every $T\in \mathcal{T}^h$ is equipped with a Lagrangian finite element.
\end{enumerate}
In order to apply well known results for the numerical approximation of incompressible flows, we choose ($P2$-$P1$) Taylor--Hood elements (\textit{cf.} \cite{hood_taylor_original}) for the velocity. By $X_h$ and $M_h$ we denote the finite element spaces for the velocity and pressure,
\begin{align*}
    M_h &\coloneqq \{  q \in   C (\overline{\Omega}) :  \int_{\Omega} q \diff x =0, \quad q \lvert_K \in \mathbb{P}_1 (K) \quad\forall K \in  \mathcal{T}_h  \}
    ,\\
    X_h &\coloneqq \{  v \in C (\overline{\Omega}) : v= 0 \text{ on }\partial \Omega, \quad v \lvert_K \in \mathbb{P}_2 (K) \quad\forall K \in  \mathcal{T}_h 
    \}
    ,
\end{align*}
where $\mathbb{P}_k(\Omega)$ is the set of polynomials of degree $k$ or less on the domain $\Omega$.
For better readability, we denote approximately divergence-free functions in $X_h$ as the space space $V_h$, \textit{i.e.}
\begin{align}\label{discrete_inf_sup}
    V_h \coloneqq \{  v\in [X_h]^3 , \quad
    (\nabla \cdot v, q ) = 0 \text{ for all } q\in M_h 
    \}.
\end{align}
It is a well-known result (\textit{cf.} \cite{verfurth}) that this choice of spaces fulfills the inf-sup condition (sometimes also called \textit{LBB} or \textit{Ladyzhenskaya–Babuska–Brezzi} condition) under our given assumptions (\ref{ass1}- \ref{ass4}), that is
\begin{align*}
     \sup_{v_h \in [X_h]^3} \frac{(\nabla \cdot v_h,  q_h)}{\LTwoNorm{v_h}}
     \geq
     C
     \LTwoNorm{ q_h}
\end{align*}
for all $q_h \in M_h$ and for some constant $C>0$.
The director and discrete Laplacian are approximated using piecewise linear functions,
\begin{align*}
    Z_h \coloneqq \{  v \in C (\overline{\Omega}) : v \lvert_K \in \mathbb{P}_1 (K) \quad\forall K \in  \mathcal{T}_h \}.
\end{align*}
The homogeneous Dirichlet boundary conditions can be implied in the space by
\begin{align*}
    Y_h \coloneqq \{  v \in C (\overline{\Omega}) : v= 0 \text{ on }\partial \Omega, \quad v \lvert_K \in \mathbb{P}_1 (K) \quad\forall K \in  \mathcal{T}_h \}.
\end{align*}
Our final mixed finite element space will be denoted by,
\begin{align*}
    U_h \coloneqq V_h \times [Y_h]^3 
\end{align*}

\begin{lemma}[Inverse estimate, \textit{cf.} {\cite[Thm. 4.5.11]{Scott_Brenner_FEM}}]\label{lem:inverse_estimate}
Under the assumptions (\ref{ass1}-\ref{ass4}), let $p,q \in [1, \infty]$ and $0\leq m \leq l \leq 1$. Then there exists a constant $C$ independent of $h$, such that
\begin{align*}
    \norm{v}_{W^{l,p}(\Omega)}\leq C h^{m-l + \min\{0, 3/p-3/q \}}  \norm{v}_{W^{m,q}(\Omega)}
\end{align*}
holds for all $v \in [Z_h]^3, V_h$.
\end{lemma}
For the finite element spaces $[Y_h]^3$ and $ V_h$ and a function $f \in \L{2}$ the standard $L^2$-projection, is
denoted by $ \project{:}L^2(\Omega)^3 \ra [Y_h]^3$ and $\projectV {:}\Ha \ra V_h$, respectively. 
\begin{lemma}[Projection operator, \textit{cf.}~\cite{wahlbin},~{\cite[Lem.~1.131, Prop.~1.134]{guermond}}]
\label{thm:stability_projection}
Under the previously given assumptions (\ref{ass1} - \ref{ass4}), the $L^2$ projection onto $[Y_h]^3$ is stable in $\L{p}$ for all $1\leq p \leq \infty$ and $\H{1}$, \textit{i.e.} we have
\begin{align*}
    \lebnorm{\project (\fat u_1)}{p}
  &  \leq
   C  \lebnorm{\fat  u_1}{p}, \\
    \lebnorm{\project \fat  u_2 - \fat  u_2}{2} + h \Hsobnorm{\project (\fat u_2)}{1}
   & \leq
    C h\Hsobnorm{ \fat u_2}{1}
\end{align*}
for all $\fat u_1 \in \L{2} \cap \L{p}$, $\fat u_2 \in \H{1}$.
\end{lemma}
For the projection onto the finite element space for the velocity $V_h$, $\projectV$ admits the error and stability estimates 
\begin{align}\label{projection_approximation_vh}
\begin{split}
\LTwoNorm{\projectV (\fat v) - \fat v} &\leq C h \LTwoNorm{\nabla \fat v}\\ \Vert \projectV \vv - \vv \Vert _{H^1(\Omega)} + h \Vert \projectV \vv \Vert _{H^2(\Omega)}& \leq C h  \Vert \vv \Vert  _{H^2(\Omega)} 
\end{split}
\end{align}
for all $\fat v \in H^1_{0,\sigma} ( \Omega)$ and $\vv \in H^2(\Omega)\cap \V$. This follows from~\cite[Lem.~4.3]{projection} and the regularity of the Stokes problem on convex polygons (see~\cite{RegSobDauge}, see also~\cite[Proposition~2]{regStokes3} or~\cite[Cor.~1.8]{RegStokes}).

\subsection{Interpolation and mass-lumping}\label{sec:interpol}
We define the nodal interpolation operator $\interpol : \Cont(\Omega) \to [Z_h]^3$ by
\begin{align*}
    \interpol (f) \coloneqq \sum_{z \in \mathcal{N}_h} f(z) \phi_z,
\end{align*}
for all functions $f \in  C(\Omega) $, where $\mathcal{N}_h$ is the set of all nodes in our mesh. 
For continuous functions $y_1,y_2 \in C(\Omega; \mathbb{R}^3)$ we introduce the lumped inner product (\textit{cf.} \cite{mass_lumping_raviart,ciavaldini}) and the according norm as
\begin{align}
\begin{aligned}\label{def:lumped_norm_product}
    (y_1,y_2)_h &\coloneqq 
    \int_{\Omega}
    \interpol (y_1 \cdot y_2) \diff x
    =\sum_{z \in \mathcal{N}_h} y_1(z) \cdot y_2(z) \int_{\Omega} \phi_z \diff x, \\
    \norm{y_1}_h^2 &\coloneqq (y_1,y_1)_h.
\end{aligned}
\end{align}
For the interpolation operator we observe the following standard error estimate:
\begin{lemma}\label{thm:interpolation_error}
Under the assumptions (\ref{ass1}-\ref{ass4}), let $3/2 < p \leq \infty$. Then there exists a constant $C > 0$ independent of $h$ and the function $f$, such that 
\begin{align*}
    \norm{f - \interpol (f)}_{W^{k,p}(\Omega)}
    \leq
    C h^{2-k}
    \abs{f}_{W^{2,p}(\Omega)}
\end{align*}
holds for all $0 \leq k \leq 2$ and $f\in  W^{2,p}(
\Omega)$.
\end{lemma}
The above theorem is standard and can for example be found in \cite[Thm. 4.4.20]{Scott_Brenner_FEM}.
There exists a generic constant $c_L>0$ independent of $h>0$ such that for functions $y_1,y_2 \in [Z_h]^3$ the estimates
\begin{align}
   \LTwoNorm{y_1} \leq   \norm{y_1}_h &\leq c_L\LTwoNorm{y_1} \label{lump_estim_1} 
\end{align}
hold. 
\begin{lemma}\label{cor:lumping_error_yh_yh}
Let $\frac{1}{p}+\frac{1}{q}=1$. There exists a generic constant $C>0$ independent of $h>0$ such that 
\begin{align}
\begin{aligned}\label{lump_estim_2}
   \vert (y_1,y_2)_h -(y_1,y_2)\vert
   &\leq
   C h \lebnorm{ y_1}{p} \lebnorm{\nabla y_2}{q}  
   \\
      \vert (y_h,g)_h -(y_h,g)\vert
   &\leq
  C h \lebnorm{y_h}{p}
    \sobnorm{ g}{2}{q} 
    ,\, \, 
\end{aligned}
\end{align}
for all $y_1,y_2, y_h \in [Z_h]^3$, $ g\in W^{2,q}(\Omega) $ with $q>3/2$.
\end{lemma}
\begin{proof}
We proceed as in \cite{ciavaldini} and find for $ y_1$, $y_2\in Z_h$ 
\begin{align}
\begin{aligned}\label{proof_lump_convergence}
   \vert (y_1,y_2)_h -( y_1,  y_2)\vert
   &\leq
   \int_{\Omega} \left \vert \interpol (y_1\cdot y_2) - y_1 \cdot y_2 \right \vert \diff x 
  \\& 
  \leq
  \sum_{K\in \mathcal{T}^h}
  \norm{ \interpol (y_1\cdot y_2) - y_1\cdot y_2}_{L^2(K)}
  \norm{1}_{L^2(K)}
  \\&
  \leq
  Ch^2
  \sum_{K\in \mathcal{T}^h}
  \norm{\nabla^2 ( y_1 \cdot  y_2)}_{L^2(K)}
  \norm{1}_{L^2(K)}
  \\&
  \leq
  Ch^2
  \sum_{K\in \mathcal{T}^h}
  \norm{\nabla  y_1 : \nabla   y_2}_{L^2(K)}
  \norm{1}_{L^2(K)}
  \\&
  \leq
  Ch^{2}
  \sum_{K\in \mathcal{T}^h}
  \norm{\nabla  y_1}_{L^p(K)} \norm{\nabla  y_2}_{L^q(K)}
  \\&
  \leq
  Ch^{2}
   \lebnorm{ \nabla  y_1}{p} \lebnorm{\nabla  y_2}{q}
   \\&
  \leq
  Ch
   \lebnorm{  y_1}{p} \lebnorm{\nabla  y_2}{q}.
 \end{aligned}
\end{align}
The first inequality is a triangle inequality, the second one uses H\"older's inequality, the third bounds the interpolation error by Lemma~\ref{thm:interpolation_error}. Since $[Z_h]^3$ consists of piecewise affine-linear functions the second derivative of a function in $[Z_h]^3$ vanishes. This together with H\"older's inequality, local and global inverse estimates~(\textit{cf.}~\ref{lem:inverse_estimate}) allow to conclude. 
Note that the choice of exponents in the third step is necessary in order to make sure that we can indeed apply Lemma~\ref{thm:interpolation_error}. 

In case that one function is not a finite element functions, we have to estimate the difference with the interpolated function, \textit{i.e.,}
\begin{align}
\begin{aligned}\label{s}
    \vert (\interpol  y_h, \interpol g) - (y_h,g)\vert &= \vert ( y_h, \interpol g) - (y_h,g)\vert
    \\
    &
    \leq
    C
    \lebnorm{ y_h}{p}
    \lebnorm{(\interpol - I) g}{q}
    \leq Ch^{2}
    \lebnorm{ y_h}{p}
    \lebnorm{\nabla^2 g}{q}\,,
\end{aligned}
\end{align}
and estimate the last factor on the right-hand side of \eqref{proof_lump_convergence} by
\begin{align*}
    \lebnorm{ \nabla \interpol g}{q}
    &\leq
    \lebnorm{ \nabla (\interpol-I) g}{q}+\lebnorm{ \nabla g}{q}
    \\
    &\leq
    Ch \lebnorm{ \nabla^2 g}{q} + \lebnorm{ \nabla g}{q}
    \leq C \sobnorm{ \nabla g}{1}{q}.
\end{align*}
\end{proof}
We introduce another interpolation operator $\interad: L^1(\Omega; \Rr^3 ) \ra [Z_h]^3$ by 
\begin{equation}
\interad ( g ) : = \sum_{z\in\mathcal{N}_h} \frac{\int_\Omega g (x) \phi_z(x)\diff x }{\int_{\Omega} \phi_z(x)\diff x} \phi_z\,.
\end{equation}
We observe the identity \begin{align}
    \begin{split}
\label{interpoladjiont}
\left (  \inter{f} , g\right )&= \int_\Omega \sum_{z\in\mathcal{N}_h} f(z) \phi_z(x) g(x) \diff x
\\&=   \sum_{z\in\mathcal{N}_h} f(z) \int_\Omega\phi_z(x) g(x) \diff x  = \left ( f , \interad g \right )_h \,.
\end{split}
\end{align}
\begin{lemma}\label{lem:adint}
Let $f \in \Cont(\ov\Omega)$. Then it holds
\begin{align*}
    \Vert f - \interad f \Vert_{\Cont(\ov\Omega)} \to 0 \text{ as }h\ra 0\,.
\end{align*}
\end{lemma}
\begin{proof}
Consider any $y\in \ov\Omega$. We obtain the pointwise estimate
\begin{align*}
    f(y)- \interad f(y) &=
    \sum_z f(y) \phi_z(y) - \sum_z \int_\Omega f(x) \phi_z(x) \diff x \phi_z(y) \frac{1}{\int_\Omega \phi_z(x) \diff x}
    \\&
=
    \sum_z \phi_z(y) \frac{\int_\Omega [f(y)-f(x)] \phi_z(x) \diff x}{\int_\Omega \phi_z(x) \diff x}
    \\&
\leq
    \sum_z \phi_z(y) \max_{x \in supp(\phi_z)} \vert f(x)-f(y)\vert
    \\&
= \max_{z \in \mathcal{N}_h}
   \max_{x,s \in supp(\phi_z)} \vert f(x)-f(s)\vert ,
\end{align*}
where the last term converges to $0$ as $h\to 0$, since $g$ is continuous.
Above, we used that $   \sum_z  \phi_z(y) = 1 $ and that support of the shape functions $\phi_{ z}$ vanishes as $h\to 0$.
Since $y$ was arbitrary the assertion follows. 
\end{proof}
Further, we introduce the discrete Laplacian in a slightly different way than usual, since we further equip it with mass lumping. Then the discrete Laplacian $\Delta_h: \H{1} \to [Y_h]^3$ of a function $f\in \H{1}$ is defined via 
\begin{align}
    (-\Delta_h f,b)_h 
        & =
        (\nabla f,\nabla b) \quad \text{for all }b \in Y_h.\label{discreteLaplace}
\end{align}
Similarly, we define an adapted projection via $\projectL:L^2(\Omega;\Rr^3)\ra [Z_h]^3 $ via 
\begin{equation}
(\projectL f , b)_h = (f , b) \quad\text{for all }b\in [Z_h]^3\,.\label{newproject}
\end{equation}
This choice allows to infer with~\eqref{lump_estim_1} that 
\begin{equation}
\hnorm{\projectL f } ^2 = (\projectL f , \projectL f)_h = (f ,\projectL f) \leq \Vert f \Vert_{L^2(\Omega)} \Vert \projectL f\Vert_{L^2(\Omega)} \leq \Vert f \Vert_{L^2(\Omega)}  \hnorm{\projectL f } \,,\label{projectest}
\end{equation}
which implies together with~\eqref{lump_estim_1} that $\LTwoNorm{\projectL f}\leq  \hnorm{\projectL f }\leq  \Vert f \Vert_{L^2(\Omega)}  $. 
\begin{remark}[Adapted discretization of convection term]\label{rem:proj}
In comparison to the discretization in~\cite{lasarzik_main}, we use the above projection $\projectL$  in the definition of the convection term in comparison to the $L^2$-projection~$\project$. This change is needed in order to show the convergence to the finer concept of energy-variational solutions. In~\cite{lasarzik_main} only the convergence to dissipative solutions was shown, which is a  bigger class of solutions. One key ingredient for this improvement are the above norm-equivalences. 
\end{remark}

\section{Discrete system\label{sec:dis}}
In order to restrict ourselves to finite element spaces with a zero trace, we decompose the spatial approximation of our director into its interior and its non-homogeneous Dirichlet boundary condition, \textit{i.e.}
\begin{align}\label{decomposition_director}
    d^j = {d^j_\circ} + \interpol{\fat d_0},
\end{align}
with $\tr(d^j_\circ)=0$ and $\fat d_0 \in  H^{2}(\Omega)$. 

\begin{scheme}{\ref{discrete_scheme}}
Let $(v^{0}, d^{0})= (\projectV v_0, \interpol d_0)$.
For $1 \leq j\leq J$ and $(v^{j-1}, d^{j-1}_\circ )\in U_h$, we want to find $(v^{j}, d^{j}_\circ )\in U_h$, such that
\begin{subequations}\label{discrete_scheme}
\begin{align}
    \begin{aligned}\label{discrete_scheme_a}
        (\discreteDiff_t v^j,a) 
       +((&v^{j-1}  \cdot \nabla) v^j,a) + \frac{1}{2} ([\nabla \cdot v^{j-1}]v^j,a) \\
        &-([\projectL{\nabla d^{j-1}}]^T [d^{j-1/2}\times (d^{j-1/2} \times \Delta_h d^{j-1/2} )],a)_h 
        +(T_L^j,\nabla a) = 0 \,,
    \end{aligned}&\\
    \begin{aligned}\label{discrete_scheme_c}
        &(\discreteDiff_t d^j,c)_h 
        + (d^{j-1/2}\times [\projectL{\nabla d^{j-1}}v^j], d^{j-1/2} \times c)_h 
        -( (\nabla v^{j})_{skw}d^{j-1/2}, c)_h  
        \\
        &{+}\lambda (d^{j-1/2} \times [(\nabla v^{j})_{sym}d^{j-1/2}],d^{j-1/2} \times c)_h 
        {-}(d^{j-1/2} \times \Delta_h d^{j-1/2},d^{j-1/2} \times c)_h
        = 0,
    \end{aligned}&
\end{align}
\end{subequations}
for all $(a,c)\in U_h$. Hereby we discretize the Leslie stress tensor as
\begin{align}
\begin{aligned}\label{leslie_stress_tensor_num_dischrete}
    (T_L^j, \nabla a) \coloneqq & (T_D^j, \nabla a)
    +\lambda (d^{j-1/2} \times [(\nabla a)_{sym}d^{j-1/2}],d^{j-1/2} \times \Delta_h d^{j-1/2})_h  \\&
    +( (\nabla a)_{skw} \Delta_h d^{j-1/2}, d^{j-1/2})_h,
\end{aligned}
\end{align}
where $T^j_D$ collects the dissipative terms of the discrete Leslie stress tensor again, \textit{i.e.}
\begin{align*}
(T_D^j, \nabla a) \coloneqq & 
    (\mu_1 + \lambda^2) ((d^{j} \cdot (\nabla v^{j})_{sym} d^{j}),
    (d^{j} \cdot (\nabla a)_{sym} d^{j}) ) \\&
    + \mu_4 \left ( (\nabla v^{j})_{sym} , (\nabla a)_{sym}\right) 
    +(\mu_5 + \mu_6 - \lambda^2) \left( 
    (\nabla v^{j})_{sym}  d^{j},
    (\nabla a)_{sym}  d^{j}
    \right)\,.
\end{align*}

\end{scheme}
\begin{remark}[Features of the discrete system]
The mass lumping is applied on the director equation \eqref{discrete_scheme_c} to guarantee the unit-norm restriction of $d^j$ in every node. 
In order to find a discrete energy law, the associated terms in the momentum equation \eqref{discrete_scheme_a} and the discrete Laplacian are mass-lumped as well. Since the gradient of the director is piecewise constant and therefore not well-defined at the nodes of the mesh, we apply the mass-lumped $L^2$ projection $\projectL$ in the convection term. Note that this is computationally not too expensive, since it only effects the gradient of the past iterate. 

Additionally, we note that solving the discrete system amounts in introducing two additional coupled equations, one to calculate the discrete Laplacian~\eqref{discreteLaplace} and one to assure that the solution $v^j$ fulfills the discrete vanishing divergence condition~\eqref{discrete_inf_sup}. 
\end{remark}

\subsection{\textit{A priori} estimates}\label{sec:apri}
The proposed scheme is unconditionally solvable. 
\begin{proposition}[Existence of discrete solutions]
\label{discrete_existence_theorem}
Let $k,h>0$, $j \in \mathbb{N}$ and $1\leq j \leq J = \lfloor T/k \rfloor$.
Then there exists a solution $u^j=(v^j,d^j_\circ )\in U_h$ solving Scheme \eqref{discrete_scheme}.
\end{proposition}
\begin{proof}
 We show that the map $\Pi :U_h \to U_h$ implicitly defined by the above scheme \eqref{discrete_scheme}, whose zero is the next iterate in terms of $u = (v^j, - \Delta_h d^{j-1/2})$, fulfills
\begin{align}\label{inequality_existence_discrete}
( 
    \Pi(u),u )_{U_h}
    \geq 0
\end{align}
for all $\norm{u}_{U_h}\geq R_{j-1}$ for some constant $R_{j-1}>0$.
Where the dependency on $d^j$ has to be interpreted via the discrete Greens function $\mathcal{G}_h$ associated to the discrete Laplacian and given by 
\begin{align*}
    \left( \nabla \mathcal{G}_h (w ) , \nabla b\right) = ( w , b)_h \quad\text{for all }b\in Y_h\,.
\end{align*}
Noting that this operator is a continuous bijection due to Lax-Milgram, we observe that the associated mapping $\Pi$ corresponding to our discrete scheme is actually well defined and continuous. 
Then the existence of solutions to our discrete scheme \eqref{discrete_scheme}  follows from a non-linear version of Brouwer's fix-point theorem (\textit{cf.} \cite{Ruzicka}).
In particular, we evaluate 
\begin{align}
\begin{aligned}\label{eq:discrete_energy_estimate}
( \Pi (v^j,&  - \Delta_h d^{j-1/2}),(v^j, - \Delta_h d^{j-1/2}) )_{U_h}
 \\&=
     \frac{1}{2} \discreteDiff_t \LTwoNorm{v^j}^2
    + \frac{1}{2}\discreteDiff_t \LTwoNorm{\nabla d^j}^2
    +\frac{k}{2}  \LTwoNorm{\discreteDiff_t v^j}^2
    \\& \quad
    +(\mu_1 + \lambda^2) \lebnorm{d^{j} \cdot (\nabla v^{j})_{sym} d^{j}}{2}
    + \mu_4 \lebnorm{(\nabla v^{j})_{sym}}{2} \\&\quad
    +(\mu_5 + \mu_6 - \lambda^2) \lebnorm{(\nabla v^{j})_{sym}  d^{j} }{2}
    +\hnorm{\interpol (d^{j-1/2} \times \Delta_h d^{j-1/2})}^2 
    ,
\end{aligned}
\end{align}
where we used the discrete identities
\begin{align}
\begin{aligned}\label{discrete_identities}
    &
    (\discreteDiff_t v^j,v^j)
    =
    \frac{1}{2} \discreteDiff_t \LTwoNorm{v^j}^2 +\frac{k}{2}  \LTwoNorm{\discreteDiff_t v^j}^2\,,\quad
     (-\Delta_h d^{j-1/2},\discreteDiff_t d^j)_h  
    =\frac{1}{2} \discreteDiff_t \lebnorm{\nabla d^j}{2}^2 
    . 
\end{aligned}
\end{align}
Equation \eqref{eq:discrete_energy_estimate} is indeed positive, if 
\begin{align*}
      \LTwoNorm{v^j}^2
    + \lebnorm{\nabla d^j}{2}^2
    \geq
    \LTwoNorm{v^{j-1}}^2
    +
    \lebnorm{\nabla d^{j-1}}{2}^2
\end{align*}
holds, which only depends on $u^{j-1},k,h$. This is implied by choosing an appropriate norm, since we have
\begin{align*}
  \LTwoNorm{v^j}^2
    + \lebnorm{\nabla d^j}{2}^2
    + \lebnorm{\nabla d^{j-1}}{2}^2
    \geq
    \LTwoNorm{v^j}^2
    +
    \frac{ 2 (c_{l} h)^2}{3} \lebnorm{\Delta_h d^{j-1/2}}{2}^2
\end{align*}
Note that the second term stems from a simple claculation obtained by applying the inverse estimate:
\begin{align*}
    \lebnorm{\Delta_h d^{j-1/2}}{2}^2
    \leq
    c \hnorm{\Delta_h d^{j-1/2}}^2
    &=
    c (\nabla d^{j-1/2}, \nabla (-\Delta_h d^{j-1/2}))
    \\&
    \leq
    (c_l h)^{-1} \LTwoNorm{\nabla d^{j-1/2}} \LTwoNorm{\Delta_h d^{j-1/2}}.
\end{align*}
\end{proof}
\begin{proposition}[Properties of the discrete system]\label{thm_discrete_energy_equality}
Let $u^j=(v^j,d^j_\circ)\in U_h$ be a solution of the discrete scheme \eqref{discrete_scheme} for all $1\leq j\leq n$. Then the follwing energy equality holds,
\begin{align}\small
\begin{split}\label{discrete_energy_equ}
\frac{1}{2}\LTwoNorm{v^n}^2  & + \frac{1}{2}\LTwoNorm{\nabla d^n}^2
+\frac{k^2}{2} \sum_{j=1}^n   \LTwoNorm{\discreteDiff_t v^j}^2 
\\&
    +k \sum_{j=1}^n (\mu_1 + \lambda^2) \lebnorm{d^{j} \cdot (\nabla v^{j})_{sym} d^{j}}{2}^2
    + k \sum_{j=1}^n \mu_4 \lebnorm{(\nabla v^{j})_{sym}}{2}^2 \\&
    + k \sum_{j=1}^n (\mu_5 + \mu_6 - \lambda^2) \lebnorm{(\nabla v^{j})_{sym}  d^{j} }{2}^2
    + k \sum_{j=1}^n \hnorm{ \interpol (d^{j-1/2} \times \Delta_h d^{j-1/2}) }^2 \\& \quad\quad\quad\quad\quad\quad\quad\quad\quad\quad\quad\quad\quad\quad\quad\quad\quad\quad\quad
    =
    \frac{1}{2}  \LTwoNorm{v^0}^2
    + \frac{1}{2}\LTwoNorm{\nabla d^0}^2 
\end{split}
\end{align}
for all $ 1\leq n \leq J$.
The algebraic restriction is fulfilled in every node 
\begin{align}\label{unit_norm_restriction}
    \vert d^j(z) \vert =1 \text{ for all }z \in \mathcal{N}_h,
\end{align}    
where $\mathcal{N}_h$ is the set of all nodes of the mesh.
Additionally, it holds for some constant $C>0$ that
\begin{align}
k \sum_{l=1}^j \left [\norm{\discreteDiff_t d^l}_{L^{3/2}(\Omega)}^2
+ \norm{\discreteDiff_t v^l}_{(\velspace )^*}^2+  \lebnorm{\nabla v^{l}}{2}^2 
\right ]
\leq 
C \left  (  \mathcal{E}(v_0,d_0)  +1 \right )\,.
\label{bounds_temp_variation_inequ}
\end{align}
\end{proposition}
\begin{proof}
The discrete energy estimate~\eqref{discrete_energy_equ} follows simply from iterating \eqref{eq:discrete_energy_estimate}.
In order to show~\eqref{unit_norm_restriction}, we adapt the approach in \cite{lasarzik_main}. Therefore, the director equation \eqref{discrete_scheme_c} is tested with $d^{j-1/2}(\hat{z}) \phi_{\hat{z}}$, where ${\hat{z}} \in \mathcal{N}_h$ is a node of the mesh. All terms except for the one including the time-derivative of the director vanish due to the mass lumping and the skew-symmetry.
Then, we are left with
\begin{align*}
    0 = 
    (\discreteDiff_t d^j,d^{j-1/2}({\hat{z}}) \phi_{\hat{z}})_h  
    =&
    \sum_{z \in \mathcal{N}_h} \discreteDiff_t d^{j-1/2}(z) \cdot d^{j-1/2}({\hat{z}}) \phi_{\hat{z}} \int_{\Omega} \phi_z \diff x \\
    =&
     \frac{1}{2k} ( \vert d^{j}({\hat{z}})\vert^2 - \vert d^{j-1}({\hat{z}})\vert^2) \int_{\Omega} \phi_{\hat{z}} \diff x ,
\end{align*}
which yields the result for all interior nodes $\hat{z}$. At the boundary the unit-norm constraint is fulfilled by the assumption on $\fat d_0$, our inhomogeneous Dirichlet boundary condition.
Korn's inequality yields the estimate for the last addend in~\eqref{bounds_temp_variation_inequ}.

Since $\discreteDiff_t d^j \in Y_h$, the definition of the $L^2$-projection yields for a smooth test function $\fat \phi \in  \Cont^\infty_c(\Omega)$ that
\begin{align*}
    (\discreteDiff_t d^j, \fat \phi )
    =
    (\discreteDiff_t d^j, \project (\fat \phi) ).
\end{align*}
Using a duality argument, we are able to estimate the time derivative of the director in a $L^p$-norm by
\begin{align}
\begin{aligned}\label{eq:dtd_decomp}
    \lebnorm{\discreteDiff_t d^j}{3/2} 
    &=
    \sup_{\fat \phi \in \L{3}}
    \left \lvert \frac{(\discreteDiff_t d^j,\project (\fat \phi) )}{\lebnorm{\phi}{3}} \right \rvert
    \\& 
    \leq
    \sup_{\fat \phi \in \L{3}} \abs{  \frac{(\discreteDiff_t d^j,\project (\fat \phi) )_h}{\lebnorm{\phi}{3}} }
    +
    \sup_{\fat \phi \in \L{3}} \abs{  \frac{(\discreteDiff_t d^j,\project (\fat \phi) ) - (\discreteDiff_t d^j,\project (\fat \phi) )_h}{\lebnorm{\phi}{3}} }.
\end{aligned}
\end{align}
The first term can easily be estimated using our discrete director equation~\eqref{discrete_scheme_c} for $c \in Y_h$:
\begin{align}
\begin{aligned}\label{dtdj_estimate_1}
    \abs{  (\discreteDiff_t d^j,  c )_h} 
    \leq& C
     \lebnorm{d^{j-1/2} }{\infty}^2
     \lebnorm{v^j}{6}
     \lebnorm{\projectL (\nabla d^{j-1})}{2}
     \lebnorm{ c}{3}
     \\&
    + C
    \hnorm{d^{j-1/2} \times \Delta_h d^{j-1/2}}
    \lebnorm{d^{j-1/2} }{\infty} \lebnorm{ c}{2}
    \\&
    +
    C
    \lebnorm{d^{j-1/2} }{\infty}^3
    \lebnorm{ \nabla v^j}{2}  \lebnorm{ c}{2}
   \\
   \leq &
   C \left (\lebnorm{ \nabla v^j}{2} +\hnorm{\interpol \left ( d^{j-1/2} \times \Delta_h d^{j-1/2} \right ) }\right )\lebnorm{ c}{3}.
\end{aligned}
\end{align}
Choosing $c = \project(\fat \phi)$ for $\fat \phi \in H^1(\Omega)$ and using the  embedding $\H{1} \hookrightarrow \L{6}$, the stability of the projection, Lemma~\ref{thm:stability_projection}, and the \textit{a priori} estimates of Proposition~\ref{thm_discrete_energy_equality} allow to obtain a bound for the right-hand side. 
The second term in \eqref{eq:dtd_decomp} describes the error introduced by the mass-lumping. It must be estimated locally as in \eqref{proof_lump_convergence}, \textit{i.e.} by additionally applying the inverse estimate one obtains
\begin{align}
\begin{aligned}\label{inequ:ciavaldini}
{\int_\Omega (I- \interpol) ( \discreteDiff_t d^j \cdot \project (\fat \phi) ) \diff x }
&\leq
\sum_{K \in \mathcal{T}_h} {\int_K (I- \interpol) (\discreteDiff_t d^j \cdot \project (\fat \phi) ) \diff x}
\\&
\leq 
C \sum_{K \in \mathcal{T}_h} \norm{ \discreteDiff_t d^j}_{L^2(K)}\norm{ \project (\fat \phi) }_{L^3(K)} \norm{1}_{L^6(K)}
\\&
\leq 
C h^{1/2} \sum_{K \in \mathcal{T}_h} \norm{ \discreteDiff_t d^j}_{L^2(K)}\norm{ \project (\fat \phi) }_{L^3(K)}\,.
\end{aligned}
\end{align}
For the time derivative of the director $\discreteDiff_t d^j$ in the $L^2$-norm, we observe locally by choosing $c=\sum_{\fat z \in \mathcal{N}_h \cap K }\discreteDiff_t d_j(\fat z) \phi_z $ denotes the  in a local version of~\eqref{dtdj_estimate_1} that
\begin{align*}
    \norm{\discreteDiff_t d^j}_{L^2(K)}^2 
    &\leq
    C
   \int_{K}\interpol\left( \vert \discreteDiff_t d^j\vert^2\right)\diff x 
\\&\leq 
   C\Big  (\norm{  v^j}_{L^6(K)}  \norm{\projectL (\nabla d^{j-1})}_{L^2(K)} \norm{ \discreteDiff_t d^j}_{L^3(K)}  \\&\quad+  \left (\int_{K} \interpol \left [d^{j-1/2} \times \Delta_h d^{j-1/2}\right ] ^2\diff x + \norm{\nabla   v^j}_{L^2(K)}\right ) \norm{\discreteDiff_t d^j}_{L^2(K)}\Big )    \,.
\end{align*}
With the inverse estimate we can also bound the time derivative in the $L^3$-norm by
\begin{align*}
    \norm{ \discreteDiff_t d^j}_{L^3(K)}
    \leq
    C h^{-1/2}
    \norm{ \discreteDiff_t d^j}_{L^2(K)}.
\end{align*}
Putting the above together and reinserting it into \eqref{inequ:ciavaldini} yields
\begin{align*}
\int_\Omega& (I- \interpol) ( \discreteDiff_t d^j \cdot \project (\fat \phi) ) \diff x 
\\& \leq
 C \norm{  v^j}_{L^6(\Omega)}  \norm{\projectL (\nabla d^{j-1})}_{L^2(\Omega)} \norm{ \project (\fat \phi) }_{L^3(\Omega)}\\
 &\quad + C    \left (\lebnorm{ \nabla v^j}{2} +\hnorm{d^{j-1/2} \times \Delta_h d^{j-1/2}}\right ) \norm{ \project (\fat \phi) }_{L^2(\Omega)}\,.
\end{align*}
Multiplying by $k$ and summing up leads to the estimate of $\discreteDiff_t d^j$, due to the \textit{a priori} estimates \eqref{discrete_energy_equ}.
\\
For the velocity, we proceed as in \cite{lasarzik-weak_solutions}: We test the time derivative of the velocity with a smooth solenoidal test function $\fat \phi$,  and apply the projection~$ \projectV$ such that $\langle \discreteDiff_t v^j, \fat \phi \rangle = ( \discreteDiff_t v^j,\projectV \fat \phi )$  and equation~\eqref{discrete_scheme_a}, 
\begin{align*}
    \abs{\langle \discreteDiff_t v^j, \fat \phi \rangle}
    &=
    \vert 
    ((v^{j-1} \cdot \nabla) v^j,\projectV (\fat \phi)) + \frac{1}{2} ([\nabla \cdot v^{j-1}]v^j,\projectV (\fat \phi)) \\&\quad 
        - ([\projectL{\nabla d^{j-1}}]^T [d^{j-1/2}\times (d^{j-1/2} \times \Delta_h d^{j-1/2} )],\projectV (\fat \phi))_h 
        +(T_L^j,\nabla \projectV \fat \phi) 
    \vert .
\end{align*}
The convective terms can be estimated by a constant using \eqref{discrete_energy_equ} and
\begin{multline*}
    \abs{((v^{j-1} \cdot \nabla) v^j,\projectV (\fat \phi)) + \frac{1}{2} ([\nabla \cdot v^{j-1}]v^j,\projectV (\fat \phi))}
   \\ \leq
    (\lebnorm{ v^{j-1}}{2} \lebnorm{\nabla v^{j}}{2}+
    \lebnorm{ v^{j}}{2} \lebnorm{\nabla v^{j-1}}{2})
    \lebnorm{\projectV (\fat \phi)}{\infty}.
\end{multline*}
For the Leslie stress tensor we infer the estimate 
\begin{align*}
    \abs{(T_L^j:\nabla \projectV \fat \phi)}
    \leq&
    \abs{(T_D^j:\nabla \projectV \fat \phi)}
    + \vert (d^{j-1/2}, (\nabla \projectV \fat \phi)_{skw} \Delta_h d^{j-1/2} )_h\vert 
    \\& \quad\quad
    + \lambda \vert (d^{j-1/2} \times [(\nabla \projectV \fat \phi)_{sym} d^{j-1/2}], d^{j-1/2} \times \Delta_h d^{j-1/2})_h\vert 
    \\
    \leq&
    C\Big  [\lebnorm{(\nabla v^{j})_{sym}}{2}\lebnorm{(\nabla \projectV \fat \phi)_{sym}}{2}
    \\&\quad\quad
    +
    \lebnorm{\mathcal I_h \left(d^{j-1/2} \times \Delta_h d^{j-1/2}\right) }{2}\lebnorm{\nabla \projectV \fat \phi}{2}
    \Big ].
\end{align*}
Note that we hereby reformulated the term including the skew-symmetric part of the velocity's gradient on the algebraic level, \textit{i.e.}
\begin{align*}
    (\Delta_h d^{j-1/2}, (\nabla \projectV \fat \phi)_{skw}  d^{j-1/2} )_h
    =
    (d^{j-1/2}\times \Delta_h d^{j-1/2},d^{j-1/2}\times[ (\nabla \projectV \fat \phi)_{skw}  ] d^{j-1/2})_h,
\end{align*}
since $[d ]_x ^T [d]_x (A)_{\skw} d = ( I- d \otimes d) (A)_{skw} d 
= (A)_{skw} d$ holds.
The Ericksen stress tensor can be estimated in a similar fashion
\begin{align*}
   & \abs{([\projectL{\nabla d^{j-1}}]^T  [d^{j-1/2}\times (d^{j-1/2} \times \Delta_h d^{j-1/2} )],\projectV \fat \phi)_h }
    \quad\quad\quad\quad\quad
    \\
    &\leq C 
    \lebnorm{\projectL{\nabla d^{j-1}}}{2}
    \lebnorm{\mathcal I_h \left( d^{j-1/2} \times \Delta_h d^{j-1/2}\right) }{2}
    \lebnorm{\projectV (\fat \phi)}{\infty}.
\end{align*}
Summing over $j$, using the embedding $H^{2} (\Omega) \hookrightarrow L^\infty (\Omega)$ and the stability of the projection onto $V_h$ due to~\eqref{projection_approximation_vh} yields the assertion.
\end{proof}

\begin{proposition}
Let $u^j = (v^j,d^j_{\circ})\in U_h$ be a solution to scheme~\eqref{discrete_scheme}. Then the discrete energy-variational inequality
\begin{align}
\begin{split}\label{discrete_envar}
\discreteDiff_t E^j &
+  \frac{k}{2}\LTwoNorm{\discreteDiff_t v^j}^2 
\\&+ \left ( T_D^j ;( \nabla v^j)_{\sym} - ( \nabla \projectV\vv )_{\sym} \right )  +  \hnorm{  d^{j-1/2} \times \Delta_h d^{j-1/2}}^2 
 \\&
- (\discreteDiff_t v^j,\projectV\vv) - \left ( (v ^{j-1}\cdot \nabla)  v ^j , \projectV\vv \right ) - \frac{1}{2} \left ( ( \di  v ^{j-1} )v^j , \projectV\vv\right )\\ &- \left ( [\projectL{\nabla d^{j-1}}]^T [d^{j-1/2}\times (d^{j-1/2} \times \Delta_h d^{j-1/2} )], \projectV\vv \right ) _h 
\\
        &-\lambda (d^{j-1/2} \times [(\nabla \projectV\vv )_{sym}d^{j-1/2}],d^{j-1/2} \times \Delta_h d^{j-1/2})_h  
   \\& +( (\nabla \projectV\vv)_{skw} \Delta_h d^{j-1/2}, d^{j-1/2})_h
  \\&+ \mathcal{K}(\projectV\vv) \left ( \frac{1}{2}\Vert v^{j-1}\Vert_{L^2(\Omega)}^2 + \frac{1}{2}\Vert \nabla d^{j-1}\Vert_{L^2(\Omega)}^2  -E^{j-1} \right ) = 0
\end{split}
\end{align}
with the energy $E^j$ and the  regularity weight $\mathcal{K}$   given by 
\begin{align*}
 E^j := \frac{1}{2}\LTwoNorm{v^j}^2 + \frac{1}{2}\LTwoNorm{
 \nabla d^j}^2 \qquad \mathcal{K}(\vv ) : ={}& \frac{1}{2}
\Vert  \vv \Vert_{L^\infty (\Omega)} ^2 \,.
\end{align*}
holds for all $\vv\in H^1(0,T;(\V)^*) \cap L^2 (0,T;\V) \cap \Cont(\ov\Omega \times [0,T])$. 
\end{proposition}

\begin{proof}

Adding~\eqref{discrete_scheme_a} tested with $- \projectV \vv^j$ to~\eqref{discrete_energy_equ}, we find the variational-energy inequality
\begin{align*}
\begin{split}
\frac{1}{2k}\Vert v^j & \Vert_{L^2(\Omega)}^2
 + \frac{1}{2k}\Vert \nabla d^j\Vert_{L^2(\Omega)}^2  
+  \frac{k}{2}\LTwoNorm{\discreteDiff_t v^j}^2 
\\&
+ \left ( T_D^j ; ( \nabla v^j)_{\sym} - ( \nabla \projectV\vv )_{\sym} \right ) 
 +  \hnorm{  d^{j-1/2} \times \Delta_h d^{j-1/2}}^2
 \\&  - (\discreteDiff_t v^j,\projectV\vv) - \left ( (v ^{j-1}\cdot \nabla)  v ^j ,\projectV \vv \right ) - \frac{1}{2} \left ( ( \di  v ^{j-1} )v^j ,\projectV \vv\right )\\ &- \left ( [\projectL{\nabla d^{j-1}}]^T [d^{j-1/2}\times (d^{j-1/2} \times \Delta_h d^{j-1/2} )], \projectV\vv \right ) _h 
        \\
        &-\lambda (d^{j-1/2} \times [(\nabla \projectV\vv )_{sym}d^{j-1/2}],d^{j-1/2} \times \Delta_h d^{j-1/2})_h  
 \\&   +( (\nabla \projectV\vv)_{skw} \Delta_h d^{j-1/2}, d^{j-1/2})_h
    =
    \frac{1}{2k}  \LTwoNorm{v^{j-1}}^2
    + \frac{1}{2k} \LTwoNorm{\nabla d^{j-1}}^2 .
\end{split}
\end{align*}
By defining the variable $E^j$ and adding  $ \mathcal{K}(\projectV\vv) \left ( \frac{1}{2}\Vert v^{j-1}\Vert_{L^2(\Omega)}^2 + \frac{1}{2}\Vert \nabla d^{j-1}\Vert_{L^2(\Omega)}^2  -E^{j-1} \right )= 0 $ implies
\eqref{discrete_envar}. 
\end{proof}

\subsection{Converging subsequences}\label{sec:extract}
Let $\{u^j\}_{j=0,...,J}$ be a sequence of measurable functions in space. We define their constant and linear interpolates in time as follows,
\begin{align*}
    \overline{u}^k_h(t) &\coloneqq u^j, &
    \quad \underline{\overline{u}}^k_h (t)&\coloneqq \frac{1}{2} (\overline{u}^k_h (t) + \underline{u}^k_h (t) ), 
    \\
       \underline{u}^k_h(t) &\coloneqq u^{j-1}, &
    \quad {u}^k_h(t)& \coloneqq \frac{u^j-u^{j-1}}{k}(t-jk) + u^j,
\end{align*}
for $(j-1)k < t \leq jk$. This yields the standard relation between the continuous and discrete time derivative, $\partial_t u^k_h (t) = \frac{u^j-u^{j-1}}{k} = \discreteDiff_t u^j$.
Analogously we define the interpolates for a smooth test function $\Tilde{u} \in C([0,T]; \mathbb{Y})$ for $(j-1)k < t \leq jk$
 by,
\begin{align*}
    \overline{\Tilde{u}}^k(t) \coloneqq \Tilde{u} (jk),
    \,\,\,  \underline{\Tilde{u}}^k(t) \coloneqq \Tilde{u}((j-1)k), 
    \,\,\, \hat{\Tilde{u}}^k(t) \coloneqq \frac{\Tilde{u}(jk)-\Tilde{u}((j-1)k)}{k}(t-jk) + \Tilde{u}(jk)\,.
\end{align*}

The \textit{a priori} estimates \labelcref{discrete_energy_equ} allow us to extract converging subsequences  which we do not relabel and $(\fat v, \fat d ,E)$ fulfilling~\eqref{reg} such that  
\begin{align}
\begin{split}\label{converging_subsequences}
\overline{v}^k_h,\underline{v}^k_h,v_h^k &\weakstarto \fat v \text{ in } L^\infty (0,T;\Ha), \\
\overline{v}^k_h,\underline{v}^k_h,v_h^k &\weakto \fat v \text{ in }  L^2(0,T;H^1_0(\Omega)), \\
\partial_t v_k^h &\weakto \partial_t \fat v \text{ in } L^2(0,T; (\velspace )^*) 
,\\
\overline{d}^k_h,\underline{d}^k_h,d_h^k 
&\weakstarto
\fat d \text{ in } L^\infty (0,T;L^\infty(\Omega)) \cap L^\infty (0,T;H^1(\Omega)) 
,\\
\partial_t d_k^h &\weakto \partial_t \fat d \text{ in } L^2(0,T;L^{3/2}(\Omega)),
\end{split}\\[2ex]
\begin{split}\label{converging_subsequences_dissipation}
\overline{d}^k_h \cdot (\nabla \overline{v}^k_h)_{sym} \overline{d}^k_h
&\weakto \fat d \cdot (\nabla \fat v)_{sym} \fat d \text{ in } L^2(0,T;\L{2})
,\\
(\nabla \overline{v}^k_h)_{sym} \overline{d}^k_h
&\weakto (\nabla \fat v)_{sym}\fat  d \text{ in } L^2(0,T;\L{2})
,\\
\mathcal{I}_h(\underline{\overline{d}}^k_h \times \Delta_h \overline{\underline{d}}_k^h)
&\weakto\fat  d \times \Delta \fat d \text{ in } L^2(0,T;\L{2}),
\\
\underline{E}^k_h(t) &\to E(t) \text{ for all } t\in[0,T]\,
,
\end{split}
\end{align}
where the last  convergence follows from Helly's selection theorem (\textit{cf.}~\cite[Ex.~8.3]{brezis}).
We can refine the convergence for the director by using the following compact embeddings, referred to as Aubin--Lions--Simon Lemma \cite{simons}, 
\begin{align}
\begin{aligned}\label{d_lions_aubin}
    \{ \fat d \in L^2(0,T;\H{1}) : \partial_t \fat d \in L^{2}(0,T;\L{3/2}) \}
    &\stackrel{c}{\hookrightarrow} L^2(0,T;\L{2}),
    \\
     \{ \fat d \in L^\infty (0,T;\H{1}) : \partial_t \fat d \in L^{2}(0,T;\L{3/2}) \}
    &\stackrel{c}{\hookrightarrow} C^0(0,T;\L{2}),
\end{aligned}
\end{align}
which yield strong convergence for the director
\begin{align*}
    d_k^h 
    \to\fat  d \text{ in } L^2 (0,T;L^2(\Omega)) .
\end{align*}
A standard interpolation inequality (\textit{cf.} \cite[p. 192 ff.]{bennett_sharpley}) and the uniform bound in $ L^\infty (0,T;L^\infty(\Omega))$ yields strong convergence,
\begin{equation}\label{director_strong_convergence}
    d_k^h 
    \to \fat d \text{ in } L^p (0,T;L^p(\Omega)) \text{ for all }p\in [1,\infty).
\end{equation}
For the term $\fat d \times \Delta \fat d$, note that due to the above convergences and \eqref{lump_estim_2}, we have
\begin{align*}
    \int_0^T& (\mathcal{I}_h(\underline{\overline{d}}^k_h \times \Delta_h \overline{\underline{d}}_k^h), \phi) \diff t\\
    &=
    \int_0^T (\mathcal{I}_h(\underline{\overline{d}}^k_h \times \Delta_h \overline{\underline{d}}_k^h), \phi) - (\underline{\overline{d}}^k_h \times \Delta_h \overline{\underline{d}}_k^h, \phi)_h \diff t
    +\int_0^T (\underline{\overline{d}}^k_h \times \Delta_h \overline{\underline{d}}_k^h, \phi)_h \diff t
    \\
    &\leq
    C h \norm{\mathcal{I}_h(\underline{\overline{d}}^k_h \times \Delta_h \overline{\underline{d}}_k^h)}_{L^2(0,T;L^2(\Omega))} \norm{\nabla \phi }_{L^2(0,T;H^1(\Omega))}
    -
    \int_0^T ( \nabla \overline{\underline{d}}_k^h,  \nabla \phi \times \underline{\overline{d}}^k_h ) \diff t
\end{align*}
for $\phi \in C_c^\infty ([0,T]\times \Omega; \mathbb{R}^3)$. Now note that the first term vanishes as $h \to 0$ and the last term converges to
\begin{align*}
    -\int_0^T ( \nabla \fat d, \nabla \phi \times \fat d) \diff t
    =
    -\int_0^T (\fat d \times \nabla \fat d, \nabla \phi ) \diff t
    = 
    \int_0^T (\fat d \times \Delta \fat d,  \phi ) \diff t
\end{align*}
where the last equality, $\nabla \cdot (\fat d \times \nabla \fat d) = \fat d\times \Delta \fat d$ , has to be understood in the distributional sense.
The fact that all interpolates in \eqref{converging_subsequences} converge against the same limit can be confirmed as in \cite{Prohl_Schmuck_2010}. Consider examplary for the velocity,
\begin{align*}
\begin{aligned}
    \norm{v^k_h - \underline{v}^k_h}_{L^2(0,T;\L{2})}^2
    &\leq
    \sum_{j=1}^J
    \lebnorm{v^j-v^{j-1}}{2}^2 
    \int_{j(k-1)}^{jk} \left(\frac{t-jk}{k} \right)^2 \diff t
\\&    \leq k
     \sum_{j=1}^J
    \frac{k^2}{3}
    \lebnorm{\discreteDiff_t v^j}{2}^2 ,
\end{aligned}
\end{align*}
which vanishes for $k \to 0$, since the term on the right-hand side is bounded by our energy estimate~\eqref{discrete_energy_equ}.
In order to deduce that the    velocity is solenoidal, we  use standard arguments (e.g. as in \cite{crouzeix_raviart}). 
We test the divergence of the velocity with a test function $\fat \phi \in C_c^\infty(\Omega \times (0,T])$ with $\int_\Omega \fat \phi \diff x =0$ and use our discrete divergence-zero condition to obtain
\begin{align*}
    \int_0^T ( \nabla \cdot \fat v(s) , \fat \phi) \diff s
    =
    \int_0^T 
    (\nabla \cdot (\fat v(s) -v^h_k(s)) , \fat \phi)
    +
     ( \nabla \cdot v^h_k(s), \fat \phi-\Pi_h \fat \phi)
     \diff s,
\end{align*}
where $\Pi_h$ is the $L^2$-projection onto $M_h$.
The first term vanishes due to the weak convergence of the velocity in \eqref{converging_subsequences}. The second summand vanishes for $h\to 0$ simply due to the uniform boundedness $v^h_k$ and a standard approximation property of the space $M_h$ (see e.g. \cite[Lemma 12.4.3]{Scott_Brenner_FEM}), by
\begin{align*}
    \int_0^T( \nabla \cdot v^h_k, \fat \phi-\Pi_h (\fat  \phi))\diff s
    &\leq
    \norm{v^h_k}_{L^2 (0,T; H^1_0 (\Omega))} \norm{\fat \phi-\Pi_h\fat \phi}_{L^2(0,T; L^2 (\Omega))}
    \\& 
    \leq
  C h^2 \norm{\nabla^2 \fat \phi}_{L^2(0,T; L^2 (\Omega))}.
\end{align*}
The convergence for our non-homogeneous Dirichlet boundary conditions of the director, follows from the trace theorem (\textit{cf.} \cite[Thm. 6.3.10]{atkinson_han}) and the interpolation error estimate \eqref{thm:interpolation_error},
\begin{align}
\begin{aligned}\label{boundary_interpolation_convergence}
    \norm{\tr(\fat d_0)- \tr (d^h_k (t))}_{L^2 (\partial \Omega)}
    =&
     \norm{\tr (\fat d_0)  - \tr (\interpol \fat d_0 ) -\tr(d^h_{k,\circ} (t) )}_{L^2 (\partial \Omega)}
     \\
     =&
    \norm{\tr ((I-\interpol) \fat d_0)}_{L^2 (\partial \Omega)} 
     \\
     \leq&
     C
     \norm{ (I-\interpol) \fat d_0}_{H^1 ( \Omega)} 
     \leq
     C h \Hseminorm{\fat d_0}{2}\, ,
\end{aligned}
\end{align}
for every $t\in [0,T]$.
Analogously, we observe for the initial condition of the director
\begin{align}\label{strong_convergence_IC}
\Hsobnorm{d^h_k (0) -\fat d_0}{1}
=
  \Hsobnorm{\interpol (\fat d_0) -\fat d_0}{1}
  \leq C h \Hseminorm{\fat d_0}{2}.
\end{align}
For the initial condition of the velocity, we can derive an analogous estimate -- however in the $L^2$ norm instead -- by making use of \eqref{projection_approximation_vh}, since we approximated the initial condition with the respective projection operator, $v^0 = \projectV (v_0)$.
\subsection{Unit-norm restriction and convergence for the director equations}\label{sec:dir}
From~\eqref{unit_norm_restriction}, we can infer that the unit-norm restriction $\vert d(t,x) \vert =1$ holds at every node $z \in \mathcal{N}_h$. In the limit as $h \to 0$, this restriction will be fulfilled almost everywhere on $\Omega$, since $\interpol ( \vert d^j\vert^2 -1) = 0 $ holds. By subsequently applying the interpolation error bound (\textit{cf.}Lemma~\ref{thm:interpolation_error}) and the inverse estimate~(\textit{cf.} Lemma~\ref{lem:inverse_estimate}), one obtains
\begin{align*}
    \lebnorm{\vert d^j\vert^2 -1 }{2}
    &=
    \lebnorm{(I- \interpol)(\vert d^j\vert^2 -1 ) }{2}
    \\&=
    \sqrt{\sum_{K \in \mathcal{T}_h} \norm{(I- \interpol) ( \vert d^j\vert^2 -1) }_{L^2 (K)}^2}
    \\&
    \leq Ch^2
     \sqrt{\sum_{K \in \mathcal{T}_h} \norm{\nabla^2 \vert d^j\vert^2  }_{L^2 (K)}^2}
    \\&
    \leq Ch
    \lebnorm{\nabla \vert d^j\vert^2 }{2}
    \\&
    \leq Ch
    \lebnorm{\nabla d^j }{2} \lebnorm{d^j}{\infty},
\end{align*}
where the right-hand side is bounded uniformly due to our \textit{a priori} energy estimate \eqref{discrete_energy_equ}. This gives us the  result of Proposition~\ref{prop:unitconv}. 
\begin{proposition}\label{prop:unitconv}
Let $d^k_h$ be the approximate solution in the sense of the discrete scheme \eqref{discrete_scheme} that converges to a limit as in \eqref{converging_subsequences}. Then the norm of $d^k_h$ converges to $1$ in a linear order, \textit{i.e.}
\begin{align*}
    \lebnorm{\vert d^k_h (t)\vert^2 -1 }{2}
    \in  \mathcal{O} (h) .
\end{align*}
\end{proposition}
In order to prove the convergence of the approximate solution to~\eqref{discrete_scheme_c} to a solution of~\eqref{weak:d}, we consider equation \eqref{discrete_scheme_c} tested with the projection $\project \fat c$ of a smooth test function $\fat c\in C^\infty_c (\Omega \times (0,T)) $ and observe, for instance by Lemma~\ref{cor:lumping_error_yh_yh} that
\begin{align*}
    \int_0^T  (\discreteDiff_t d^k_h,\project{\fat c})_h  -  (\discreteDiff_t d^k_h,\project{\fat c}) \diff t& \leq 
Ch
\norm{\discreteDiff_t d^k_h}_{L^2(0,T;L^2(\Omega))}
\norm{ \project \fat c}_{L^2(0,T;H^1(\Omega))}
\\
& \leq 
Ch^{1/2}
\norm{\discreteDiff_t d^k_h}_{L^2(0,T;L^{3/2}(\Omega))}
\norm{ \project \fat  c}_{L^2(0,T;H^1(\Omega))}
.
\end{align*}
This term vanishes as $h\ra 0$ since the last factor is bounded for smooth $\fat c$ due to~\ref{thm:stability_projection}. 
For the convection term, we infer due to the definition~\ref{newproject}
\begin{align*}
&\left (\overline{\underline{d}}^k_h\times [\projectL{\nabla \underline{d}^k_h} \overline{v}^k_h] , \overline{\underline{d}}^k_h \times \project{\fat c}\right )_h
\\&=-  \left ( \projectL{\nabla \underline{d}^k_h}  , \interpol \left [\overline{\underline{d}}^k_h\times \left (\overline{\underline{d}}^k_h \times \project{\fat c}\right ) \otimes\overline{v}^k_h\right ]  \right )_h
\\&=-  \left ( {\nabla \underline{d}^k_h}  ,  \left [\overline{\underline{d}}^k_h\times \left (\overline{\underline{d}}^k_h \times \project{\fat c}\right ) \otimes\overline{v}^k_h\right ]  \right )-  \left ( {\nabla \underline{d}^k_h}  , (\interpol -I)\left [\overline{\underline{d}}^k_h\times \left (\overline{\underline{d}}^k_h \times \project{\fat c}\right ) \otimes\overline{v}^k_h\right ]  \right )
\,.
\end{align*}
For the last term, we infer
\begin{align*}
  &  \left\vert \left ( {\nabla \underline{d}^k_h}  , (\interpol -I)\left [\overline{\underline{d}}^k_h\times \left (\overline{\underline{d}}^k_h \times \project{\fat c}\right ) \otimes\overline{v}^k_h\right ]  \right ) \right\vert 
    \\
  &  \leq{} \LTwoNorm{ {\nabla \underline{d}^k_h} } \LTwoNorm{ (\interpol -I)\left [\overline{\underline{d}}^k_h\times \left (\overline{\underline{d}}^k_h \times \project{\fat c}\right ) \otimes\overline{v}^k_h\right ]  }
   \end{align*}
   such that we find 
 \begin{align*}
&\LTwoNorm{ (\interpol -I)\left [\overline{\underline{d}}^k_h\times \left (\overline{\underline{d}}^k_h \times \project{\fat c}\right ) \otimes\overline{v}^k_h\right ]  }^2\\
  &\leq{}  C   h^4\sum_{K\in \mathcal{T}_h} \norm{\nabla^2\left [\overline{\underline{d}}^k_h\times \left (\overline{\underline{d}}^k_h \times \project{\fat c}\right ) \otimes\overline{v}^k_h\right ]}_{L^2(K)}^2 
  \\
  &\leq{}   C   h^4 \sum_{K\in \mathcal{T}_h} \left(\norm{\nabla\overline{\underline{d}}^k_h\times \left (\nabla \overline{\underline{d}}^k_h \times \project{\fat c}\right ) \otimes\overline{v}^k_h+\nabla\overline{\underline{d}}^k_h\times \left ( \overline{\underline{d}}^k_h \times \nabla\project{\fat c}\right ) \otimes\overline{v}^k_h}_{L^2(K)} \right)^2
  \\
  &\quad+ C   h^4 \sum_{K\in \mathcal{T}_h} \left(\norm{\nabla\overline{\underline{d}}^k_h\times \left ( \overline{\underline{d}}^k_h \times \project{\fat c}\right ) \otimes\nabla\overline{v}^k_h+\overline{\underline{d}}^k_h\times \left ( \nabla\overline{\underline{d}}^k_h \times \nabla\project{\fat c}\right ) \otimes\overline{v}^k_h}_{L^2(K)} \right)^2
   \\
  &\quad+ C   h^4 \sum_{K\in \mathcal{T}_h} \left(\norm{\overline{\underline{d}}^k_h\times \left ( \nabla\overline{\underline{d}}^k_h \times \project{\fat c}\right ) \otimes\nabla\overline{v}^k_h+\overline{\underline{d}}^k_h\times \left ( \overline{\underline{d}}^k_h \times \nabla\project{\fat c}\right ) \otimes\nabla\overline{v}^k_h}_{L^2(K)} \right)^2
 \\
  &\quad+ C   h^4 \sum_{K\in \mathcal{T}_h}\norm{\overline{\underline{d}}^k_h\times \left ( \overline{\underline{d}}^k_h \times \project{\fat c}\right ) \otimes\nabla^2\overline{v}^k_h}_{L^2(K)}^2
\\
&\leq C \left[h\LTwoNorm{ \nabla \overline{\underline{d}}^k_h }^2 \lebnorm{\overline{v}^k_h}{6}^2{+}h^2\lebnorm{  \overline{\underline{d}}^k_h }{\infty}^2 \lebnorm{\nabla\overline{v}^k_h}{2}^2\right] \lebnorm{\overline{\underline{d}}^k_h}{\infty}^2 \lebnorm{\project \fat c }{\infty }^2\,.
\end{align*}
The first inequality is due to Lemma~\ref{thm:interpolation_error}, in the second inequality, we calculate the derivative, where we used that the second  
derivative of $P1$ finite elements vanishes on every tetrahedron. In order to infer the last inequality, we use H\"older's inequality and inverse estimates. 

In order to pass from the mass-lumped inner product to the standard $L^2(\Omega)$ inner product,  other occurring terms have to be estimated in the same fashion, but locally. Regarding the second term for example, we observe
\begin{align*}
     \int_\Omega & (I - \interpol)(\overline{\underline{d}}^k_h\times [(\nabla \overline{v}^k_h)_{\sym} \overline{\underline{d}}^k_h ] \cdot \overline{\underline{d}}^k_h \times \project{\fat c}) \diff x
     \\&\leq    
     C   h^2 \sum_{K\in \mathcal{T}_h} \norm{\nabla^2(\overline{\underline{d}}^k_h\times [(\nabla \overline{v}^k_h)_{\sym} \overline{\underline{d}}^k_h] \cdot \overline{\underline{d}}^k_h \times \project{\fat c})}_{L^2(K)} \norm{1}_{L^2(K)}
\\&\leq C h^{2+3/2} \sum_{K\in\mathcal T_h}    \sum_{i,j=1}^3 \norm{ \partial_{x_i} \overline{\underline{d}}^k_h\times    [\partial_{x_j}(\nabla \overline{v}^k_h)_{\sym} \overline{\underline{d}}^k_h] \cdot (\overline{\underline{d}}^k_h \times \project{\fat c})}_{L^2(K)} 
   \\
   &\quad + C h^{2+3/2} \sum_{K\in\mathcal T_h} \sum_{i,j=1}^3
   \norm{ \partial_{x_i} \overline{\underline{d}}^k_h\times   [(\nabla \overline{v}^k_h)_{\sym} \partial_{x_j}\overline{\underline{d}}^k_h  ] \cdot( \overline{\underline{d}}^k_h \times \project{\fat c})}_{L^2(K)} 
    \\
   &\quad + C h^{2+3/2} \sum_{K\in\mathcal T_h} \sum_{i,j=1}^3
   \norm{ \partial_{x_i} \overline{\underline{d}}^k_h\times   [(\nabla \overline{v}^k_h)_{\sym} \overline{\underline{d}}^k_h] \cdot \partial_{x_j} ( \overline{\underline{d}}^k_h \times \project{\fat c})}_{L^2(K)} 
  \\
      &\quad + C h^{2+3/2} \sum_{K\in\mathcal T_h} \sum_{i,j=1}^3
   \norm{  \overline{\underline{d}}^k_h\times   [\partial_{x_i}(\nabla \overline{v}^k_h)_{\sym}\partial_{x_j} \overline{\underline{d}}^k_h] \cdot( \overline{\underline{d}}^k_h \times \project{\fat c})}_{L^2(K)} 
    \\
   &\quad + C h^{2+3/2} \sum_{K\in\mathcal T_h} \sum_{i,j=1}^3
   \norm{ \overline{\underline{d}}^k_h\times   [\partial_{x_i}(\nabla \overline{v}^k_h)_{\sym} \overline{\underline{d}}^k_h] \cdot \partial_{x_j} ( \overline{\underline{d}}^k_h \times \project{\fat c})}_{L^2(K)} 
      \\
   &\quad + C h^{2+3/2} \sum_{K\in\mathcal T_h} \sum_{i,j=1}^3
   \norm{\overline{\underline{d}}^k_h\times   [  (\nabla \overline{v}^k_h)_{\sym}\partial_{x_i} \overline{\underline{d}}^k_h] \cdot \partial_{x_j} ( \overline{\underline{d}}^k_h \times \project{\fat c})}_{L^2(K)} 
  \\
   &\quad + C h^{2+3/2} \sum_{K\in\mathcal T_h} \sum_{i,j=1}^3
   \norm{ \overline{\underline{d}}^k_h\times   [(\nabla \overline{v}^k_h)_{\sym} \overline{\underline{d}}^k_h] \cdot \partial^2_{x_ix_j}( \overline{\underline{d}}^k_h \times \project{\fat c})}_{L^2(K)} 
   \,,
\intertext{where we used that the second derivatives of the Lagrangian finite element functions vanish on the set $K$. 
Using H\"older's inequality and inverse estimates for all appearing terms, we  infer}
     \int_\Omega & (I - \interpol)(\overline{\underline{d}}^k_h\times [(\nabla \overline{v}^k_h)_{\sym} \overline{\underline{d}}^k_h] \cdot \overline{\underline{d}}^k_h \times \project{\fat c}) \diff x
     \\&\leq
 C h \LTwoNorm{(\nabla \overline{v}^k_h)_{\sym}}\lebnorm{\overline{\underline{d}}^k_h}{\infty}^2 \LTwoNorm{\nabla \overline{\underline{d}}^k_h}\lebnorm{\project \fat c }{\infty } 
 \\&\quad+
 C h \LTwoNorm{(\nabla \overline{v}^k_h)_{\sym}}\lebnorm{\overline{\underline{d}}^k_h}{\infty}^2\LTwoNorm{\nabla\project \fat c }\lebnorm{\overline{\underline{d}}^k_h }   \,.
\end{align*}
The right-hand side is bounded due to the $H^1$-stability of the $L^2$ projection~ (\textit{cf.}~Lemma~\ref{thm:stability_projection}). 
The estimates for the difference of the mass-lumped and $L^2(\Omega)$ inner product for the other terms follow analogously. 
Thus, it remains to pass to the limit in~\eqref{discrete_scheme_c} with $L^2(\Omega)$- inner products instead of the lumped-inner products. Note that by standard approximation arguments $ \project{\fat c} \to \fat c $ in $L^3(\Omega \times (0,T))$ for $\fat c\in\Cont^\infty_c(\Omega \times (0,T))$, \textit{cf.} Lemma~\ref{thm:stability_projection}. 
The weak convergences of~\eqref{converging_subsequences} and~\eqref{converging_subsequences_dissipation}  allow to pass to the limit in all occurring terms, and we infer 
\begin{align*}
\int_0^T \int_\Omega \t \fat d \cdot \fat c
- (\nabla \fat v )_{\skw} \fat d \cdot \fat c + ( \fat d \times (\nabla \fat d \fat v +  \lambda ( \nabla \fat v )_{ \sym} \fat d - \Delta \fat d ) ) \cdot ( \fat d \times \fat c) \diff x \diff t = 0 
\end{align*}
for all $\fat c \in L^2(0,T;L^3(\Omega))$ by density arguments. The regularity of $\fat d$ suffices to infer $2 \fat d \cdot ( \nabla \fat d \fat v ) = ( \fat v \cdot \nabla ) \vert \fat d\vert^2 $ such that $ ( \fat d \times \nabla \fat d \fat v ) \cdot ( \fat d \times \fat c ) = ( \fat v \cdot \nabla ) \fat d \cdot \fat c $. Since the relation holds for all $\fat c \in L^2(0,T;L^3(\Omega))$ it also holds a.e.~in $\Omega \times (0,T)$, which is nothing else than~\eqref{weak:d}.

\subsection{Convergence to the energy-variational formulation}\label{sec:envar}
Multiplying the discrete energy-variational formulation~\eqref{discrete_envar} by ${\fat \phi}^j$ for a $\fat \phi \in C_c^\infty( 0,T)$ applying an discrete integration as well as a discrete integration-by-parts, we find
\begin{align}
\begin{split}\label{discrete_envar_sum}
-{}& \int_0^T (\partial _t \hat {\fat \phi}^k ) \un{E} ^k \de t 
+   \int_0^T \ov  {\fat \phi}^k  \frac{1}{2}\LTwoNorm{\partial_t \hat{v}^k_h}^2 \de t 
\\&+  \int_0^T \ov {\fat \phi}^k\left [\left ( \ov{T_D}^k_h ; (  \nabla \ov{v}^k_h)_{\sym} - ( \nabla \projectV\ov{\vv}^k )_{\sym} \right )  +  \left \lVert \ov{\un{d}}^k_h \times \Delta_h \ov{\un{d}}^k_h \right \rVert 
^2 _h\right ]\de t 
\\&
+  \int_0^T \left [ ( \un{v}^k_h ,\partial_t  \projectV\ov{\vv}^k \ov {\fat \phi}^k + \projectV\un{\vv}\partial_t \ov{\fat \phi}^k ) - \ov{\fat \phi} ^k \left ( (\un v ^{k}_h\cdot \nabla) \ov{v}^k_h ,\projectV \ov{\vv}^k  \right ) \right ] \de t \\ &-\int_0^T \ov{\fat \phi}^k\left [\frac{1}{2} \left ( ( \di  \un v ^{k}_h )\ov{v}^k_h , \projectV\ov{\vv}^k\right )+ \left ( [\projectL{\nabla \un{d}^k}]^T [\ov{\un{d}}^k_h\times (\ov{\un{d}}^k_h \times \Delta_h \ov{\un{d}}^k_h )], \projectV\ov{\vv}^k \right ) _h \right ]\de t 
\\
        &- \int_0^T\ov {\fat \phi}^k\left [\lambda (\ov{\un{d}}^k_h \times [(\nabla \projectV\ov{\vv}^k )_{sym}\ov{\un{d}}^k_h],\ov{\un{d}}^k_h \times \Delta_h \ov{\un{d}}^k_h)_h  
    -( (\nabla \projectV\ov{\vv}^k)_{skw} \Delta_h \ov{\un{d}}^k_h, \ov{\un{d}}^k_h)_h\right ]\de t 
  \\&+ \int_0^T \ov {\fat \phi}^k\mathcal{K}(\projectV\ov{\vv}^k) \left ( \frac{1}{2}\Vert \un{v}^k_h\Vert_{L^2(\Omega)}^2 + \frac{1}{2}\Vert \nabla \un d^{k}_h\Vert_{h}^2  - \un E^k \right  )\de t \leq  0\,.
\end{split}
\end{align}
Since ${\fat \phi}\geq 0$ for all $t \in [0,T]$, the second term can be estimated from below by zero. 
We observe that the terms $\left (\ov{ T_D}^k_h ; (\nabla \ov{v}^k_h)_{\sym} \right )  $ are non-negative and quadratic in $\nabla \ov v ^k_h$ such that they are convex and weakly-lower continuous. 
Furthermore, we observe that 
\begin{equation*}
- \left ( [\projectL{\nabla {\un{d}}^k_h}]^T [\ov{\un{d}}^k_h\times (\ov{\un{d}}^k_h \times \Delta_h \ov{\un{d}}^k_h )], \projectV \ov{\vv}^k \right ) _h  = 
 \left ( \ov{\un{d}}^k_h\times[\projectL{\nabla {\un{d}}^k_h}]  \projectV \ov{\vv}^k , \ov{\un{d}}^k_h \times \Delta_h \ov{\un{d}}^k_h\right ) _h \,,
\end{equation*} 
and from from inequality~\eqref{projectest} that $\LTwoNorm{\nabla{\un{d}}^k_h  } ^2 \geq \left \Vert \projectL{\nabla {\un{d}}^k_h}\right \Vert_{h}^2$
such that we may find 
\begin{align*}
& \left( \ov{\un{d}}^k_h \times \Delta_h \ov{\un{d}}^k_h,\ov{\un{d}}^k_h \times \Delta_h \ov{\un{d}}^k_h  \right)_h   
+ \mathcal{K}(\projectV \ov{\vv}^k)  \frac{1}{2}\Vert {\nabla {\un{d}}^k_h}\Vert_{L^2(\Omega)}^2 
\\&- \left ( [\projectL{\nabla {\un{d}}^k_h}]^T [\ov{\un{d}}^k_h\times (\ov{\un{d}}^k_h \times \Delta_h \ov{\un{d}}^k_h )], \projectV \ov{\vv}^k \right ) _h 
\\
&= {} \left \Vert (\ov{\un{d}}^k_h \times \Delta_h \ov{\un{d}}^k_h ) + \frac{1}{2}\ov{\un{d}}^k_h\times[\projectL{\nabla {\un{d}}^k_h}]  \projectV \ov{\vv}^k  \right\Vert_h^2
\\&\quad + \frac{1}{4}\Vert \projectV \ov{\vv}^k\Vert_{L^\infty(\Omega)}^2\Vert  \projectL{\nabla {\un{d}}^k_h}\Vert_{h}^2- \frac{1}{4}\Vert \ov{\un{d}}^k_h\times[\projectL{\nabla {\un{d}}^k_h}]  \projectV \ov{\vv}^k  \Vert_{h}^2 \\
&  ={} \left \Vert \mathcal{I}_h(\ov{\un{d}}^k_h \times \Delta_h \ov{\un{d}}^k_h ) + \frac{1}{2} \mathcal{I}_h( \ov{\un{d}}^k_h\times[\projectL{\nabla {\un{d}}^k_h}]  \projectV \ov{\vv}^k ) \right\Vert_{L^2(\Omega)}^2 
\\
&\quad +\frac{1}{4} \int_{\Omega} \mathcal{I}_h \left [ \left ( \Vert \projectV \ov{\vv}^k\Vert_{L^\infty(\Omega)}^2 I - \projectV \ov{\vv}^k \otimes \projectV \ov{\vv}^k \right ) [\projectL{\nabla {\un{d}}^k_h}]^T [\projectL{\nabla {\un{d}}^k_h}]
  \right ]\de  x 
  \\&\quad +\frac{1}{4} \int_{\Omega} \mathcal{I}_h \left [ \left ( \ov{\un{d}}^k_h\cdot [\projectL{\nabla {\un{d}}^k_h}]  \projectV \ov{\vv}^k \right )^2  \right ]\de  x 
  \,.
 \end{align*}

By writing $ \Vert \fat a\Vert_{L^\infty(\Omega)}^2 I - \fat a  \otimes \fat a   =  \left( \sqrt{\Vert \fat a \Vert_{L^\infty(\Omega)}^2 - \vert \fat a \vert } \right)^2I  + \left( \vert \fat a \vert I - \frac{\fat a  \otimes \fat a }{\vert \fat a \vert } \right)^2 $ for $\fat a\in L^\infty(\Omega)$, we may use~\eqref{lump_estim_1} to infer 
\begin{align*}
    & \int_{\Omega} \mathcal{I}_h \left [ \left ( \Vert \projectV \ov{\vv}^k\Vert_{L^\infty(\Omega)}^2 I - \projectV \ov{\vv}^k \otimes \projectV \ov{\vv}^k \right ) [\projectL{\nabla {\un{d}}^k_h}]^T [\projectL{\nabla {\un{d}}^k_h}]
  \right ]\de  x 
  \\& + \int_{\Omega} \mathcal{I}_h \left [ \left ( \ov{\un{d}}^k_h\cdot [\projectL{\nabla {\un{d}}^k_h}]  \projectV \ov{\vv}^k \right )^2  \right ]\de  x \\
   & \geq  \hnorm{  \sqrt{\Vert \projectV \ov{\vv}^k\Vert_{L^\infty(\Omega)}^2 - \vert \projectV \ov{\vv}^k\vert } [\projectL{\nabla {\un{d}}^k_h}]  }^2
   \\&\quad+ \hnorm{ \left( \vert \projectV \ov{\vv}^k\vert I - \frac{\projectV \ov{\vv}^k \otimes \projectV \ov{\vv}^k}{\vert \projectV \ov{\vv}^k\vert } \right)  [\projectL{\nabla {\un{d}}^k_h}] }^2
   + 
   \hnorm{\ov{\un{d}}^k_h\cdot [\projectL{\nabla {\un{d}}^k_h}]  \projectV \ov{\vv}^k  }^2
   \\
   & \geq  \lebnorm{ \mathcal{I}_h \left[ \sqrt{\Vert \projectV \ov{\vv}^k\Vert_{L^\infty(\Omega)}^2 - \vert \projectV \ov{\vv}^k\vert^2} [\projectL{\nabla {\un{d}}^k_h}]  \right]}{2}^2
  \\&\quad+ \lebnorm{ \mathcal{I}_h \left[ \left( \vert \projectV \ov{\vv}^k\vert I - \frac{\projectV \ov{\vv}^k \otimes \projectV \ov{\vv}^k}{\vert \projectV \ov{\vv}^k\vert } \right)  [\projectL{\nabla {\un{d}}^k_h}]  \right]}{2}^2
  \\& \quad + 
   \lebnorm{ \mathcal{I}_h\left [\ov{\un{d}}^k_h\cdot [\projectL{\nabla {\un{d}}^k_h}]  \projectV \ov{\vv}^k  \right ] }{2}^2\,.
\end{align*}
From~\eqref{converging_subsequences_dissipation}, we already know that 
$\mathcal{I}_h(\ov{\un{d}}^k_h \times \Delta_h \ov{\un{d}}^k_h ) \rightharpoonup \fat d \times \Delta \fat d $ in $L^2(\Omega \times (0,T))$.  Additionally, we observe that $$ \mathcal{I}_h( \ov{\un{d}}^k_h\times[\projectL{\nabla {\un{d}}^k_h}]  \projectV \ov{\vv}^k ) \rightharpoonup \fat d \times \nabla  \fat d \cdot  \ov{\vv}^k  \quad\text{in }L^2(\Omega\times (0,T) )$$ for all $ \ov{\vv}^k \in \Cont([0,T];H^2 (\Omega)\cap \V)$ by observing that for all $ \fat \varphi \in \Cont^\infty_c(\Omega)$ and a.e.~$t\in(0,T)$, it holds
\begin{align*}
\Big ( \mathcal{I}_h&( \ov{\un{d}}^k_h\times[\projectL{\nabla {\un{d}}^k_h}]  \projectV \ov{\vv}^k )  , \fat \varphi \Big ) \\={}&
\left ( \ov{\un{d}}^k_h\times[\projectL{\nabla {\un{d}}^k_h}]  \projectV \ov{\vv}^k   , \interad (\fat \varphi) \right )_h 
\\={}& 
\left ( [\projectL{\nabla {\un{d}}^k_h}]     ;\interpol\left [ [\ov{\un{d}}^k_h]_\times^T\interad (\fat \varphi) \otimes \projectV \ov{\vv}^k\right ]\right )_h 
\\={}& 
\left ( {\nabla {\un{d}}^k_h}    ;\interpol\left [ [\ov{\un{d}}^k_h]_\times^T\interad (\fat \varphi) \otimes \projectV \ov{\vv}^k\right ]\right )
\\={}& 
\left ( {\nabla {\un{d}}^k_h}    ;  [\ov{\un{d}}^k_h]_\times^T\interad (\fat \varphi) \otimes \projectV \ov{\vv}^k\right )+
\left ( {\nabla {\un{d}}^k_h}    ;(\interpol-I)\left [ [\ov{\un{d}}^k_h]_\times^T\interad (\fat \varphi) \otimes \projectV \ov{\vv}^k\right ]\right )
\\\leq{}&
\left ( {\nabla {\un{d}}^k_h}    ;  [\ov{\un{d}}^k_h]_\times^T\interad (\fat \varphi) \otimes  \ov{\vv}^k\right )
\\
&{+ \LTwoNorm{ {\nabla {\un{d}}^k_h} } 
\lebnorm{\ov{\un{d}}^k_h}{\infty}
\LTwoNorm{ \interad (\fat \varphi)}
\lebnorm{(I-\projectV ) \ov{\vv}^k}{\infty}
} \\ &
+\LTwoNorm{ {\nabla {\un{d}}^k_h} } 
\LTwoNorm{(\interpol-I)\left [ [\ov{\un{d}}^k_h]_\times^T\interad (\fat \varphi) \otimes \projectV \ov{\vv}^k\right ]}\,.
\end{align*}
The first equality holds due to the definition of the interpolation operator in~\eqref{interpoladjiont}. The second equality is a transformation and the third is the definition of~\eqref{newproject}.  In the fourth equation, a zero is added and the last inequality follows from H\"older's inequality. 
Furthermore, we estimate the second term on the right-hand side of the previous inequality by 
\begin{align*}
&\LTwoNorm{(\interpol-I)\left [ [\ov{\un{d}}^k_h]_\times^T\interad (\fat \varphi) \otimes \projectV \ov{\vv}^k\right ]}^2\\
&\leq C h^{4}\sum_{K\in\mathcal{T}_h}  \left \lVert\nabla^2 \left [ [\ov{\un{d}}^k_h]_\times^T\interad (\fat \varphi) \otimes \projectV \ov{\vv}^k\right ] \right \rVert _{L^2(K)} ^2
\\
&\leq C h^{4}\sum_{K\in\mathcal{T}_h} \left ( \left \lVert \left [ [\nabla \ov{\un{d}}^k_h]_\times^T\nabla\interad (\fat \varphi) \otimes \projectV \ov{\vv}^k
+[\nabla \ov{\un{d}}^k_h]_\times^T\interad (\fat \varphi) \otimes \nabla \projectV \ov{\vv}^k\right ] \right \rVert _{L^2(K)} \right )^2
\\
&\quad +C h^{4}\sum_{K\in\mathcal{T}_h} \left ( \left \lVert \left [ [ \ov{\un{d}}^k_h]_\times^T\nabla\interad (\fat \varphi) \otimes \nabla \projectV \ov{\vv}^k
+[ \ov{\un{d}}^k_h]_\times^T\interad (\fat \varphi) \otimes \nabla^2 \projectV \ov{\vv}^k\right ] \right \rVert _{L^2(K)} \right )^2
\\
& \leq C \left ( h \left[ \norm{\ov{\un{d}}^k_h}_{L^\infty(\Omega)}+ \norm{\nabla \ov{\un{d}}^k_h}_{L^2(\Omega)} \right] \norm{\projectV \ov{\vv}^k}_{H^{2}(\Omega)} \norm{ \interad (\fat \varphi)}_{L^\infty(\Omega)}\right )^2\,.
\end{align*}
The first inequality follows from the local estimate of the interpolation error (\textit{cf.} Lemma~\ref{thm:interpolation_error}) and the quasi-uniformity of the mesh. In order to infer the second inequality, we calculate the derivatives via the product rule and observe that the second derivative vanishes for all function in $[Z_h]^3$. In the third estimate, we use H\"older's inequality and local inverse estimates. 

We note that the interpolation~$\interad$ converges  in $\Cont(\ov\Omega)$ (\textit{cf.}~Lemma~\ref{lem:adint}). 
Moreover, from~\eqref{projection_approximation_vh} and Gagliardo--Nirenberg's inequality~\cite{nirenberg} we infer
\begin{align*}
    \Vert \vv -  \projectV\vv \Vert_{L^\infty(\Omega)} \leq C \lVert \vv -  \projectV\vv \rVert_{H^1(\Omega)}^{1/2} \lVert\vv -  \projectV\vv  \rVert_{H^{2}(\Omega)}^{1/2}  \leq C h^{1/2} \lVert\vv -  \projectV\vv  \rVert_{H^{2}(\Omega)}
\end{align*}
for all $\vv \in H^2(\Omega ) \cap \V $. 
The weak convergence~\eqref{converging_subsequences} and the strong convergence~\eqref{director_strong_convergence} implies the weak convergence
\begin{align*}
\mathcal{I}_h( \ov{\un{d}}^k_h\times[\projectL{\nabla {\un{d}}^k_h}]  \projectV \ov{\vv}^k ) &\rightharpoonup \fat d \times (\nabla \fat d \cdot {\vv} ) \quad \text{in } L^2(\Omega \times (0,T); \Rr^3 )\,.
\intertext{In a similar fashion, we may infer the convergences }
\mathcal{I}_h \left[ \sqrt{\Vert \projectV\ov{\vv}^k\Vert_{L^\infty(\Omega)}^2 - \vert \projectV\ov{\vv}^k\vert } [\projectL{\nabla {\un{d}}^k_h}]  \right]  
&\rightharpoonup  \sqrt{\Vert \vv\Vert_{L^\infty(\Omega)}^2 - \vert \vv\vert } \nabla \fat d 
\,,\\
\mathcal{I}_h \left[ \left( \vert \projectV\ov{\vv}^k\vert I - \frac{\projectV\ov{\vv}^k \otimes \projectV\ov{\vv}^k}{\vert \projectV\ov{\vv}^k\vert } \right)  [\projectL{\nabla {\un{d}}^k_h}]  \right]  &\rightharpoonup 
 \left( \vert \vv\vert I - \frac{\vv \otimes \vv}{\vert \vv\vert } \right) \nabla\fat d 
 \,,\\
  \mathcal{I}_h\left [\ov{\un{d}}^k_h\cdot [\projectL{\nabla {\un{d}}^k_h}]  \projectV {\vv}^k  \right ] &\rightharpoonup 
  \fat d \cdot (\nabla \fat d   {\vv}) 
\end{align*}
in $L^2(\Omega \times (0,T); \Rr^{3\times 3})$ and $  L^2(\Omega \times (0,T); \Rr)$, respectively. 
The weakly-lower semi-continuity of the $L^2(\Omega)$ norm implies 
\begin{align*}
  \liminf_{k ,h\ra 0 } & \int_0^T \ov{\fat \phi}^k \Big[   \left( \ov{\un{d}}^k_h \times \Delta_h \ov{\un{d}}^k_h,\ov{\un{d}}^k_h \times \Delta_h \ov{\un{d}}^k_h  \right)_h   
\\
&
+ \mathcal{K}({\ov{\vv}^k})  \frac{1}{2}\Vert {\nabla {\un{d}}^k_h}\Vert_{L^2(\Omega)  }^2 - \left ( [\projectL{\nabla {\un{d}}^k_h}]^T [\ov{\un{d}}^k_h\times (\ov{\un{d}}^k_h \times \Delta_h \ov{\un{d}}^k_h )], \ov{{\vv}}^k \right ) _h \Big] \de t 
\\
& \geq  \int_0^T {\fat \phi} \left [ \left \Vert \fat d \times \Delta \fat d  + \frac{1}{2} \fat d \times ( \nabla \fat d \vv  ) \right\Vert_{L^2(\Omega)}^2 
+ \frac{1}{4}  \lebnorm{  \sqrt{\Vert \vv \Vert_{L^\infty(\Omega)}^2 - \vert \vv \vert }\nabla \fat d }{2}^2\right ]\diff t  \\& \quad 
   +  \frac{1}{4} \int_0^T {\fat \phi} \left[ \lebnorm{  \left( \vert \vv \vert I - \frac{\vv  \otimes \vv }{\vert \vv \vert } \right)  \nabla \fat d }{2}^2
   + 
   \lebnorm{ \fat d\cdot \nabla \fat d   \vv   }{2}^2 \right] \diff t 
   \\& = \int_0^T {\fat \phi} \left [ \lebnorm{\fat d \times \Delta \fat d }{2}^2 + \left( \fat d \times \Delta \fat d, \fat d \times ( \nabla \fat d \vv  )\right) \right ]\diff t
    \\ & \quad + \frac{1}{4} \int_0^T {\fat \phi}\left [ \lebnorm{\fat d \times ( \nabla \fat d \vv  ) }{2}^2 +  \int_\Omega  \left( \Vert \vv  \Vert_{L^\infty(\Omega)} ^2 - \vert \vv \vert^2 \right) I: [ \nabla \fat d^T] \nabla \fat d \de  x \right ]\diff t \\ & \quad  +\int_0^T {\fat \phi} \int_{\Omega} \left[ \left( \vert \vv  \vert^2 I - \vv  \otimes \vv   \right)  [ \nabla \fat d^T] \nabla \fat d + ( \fat d \cdot \nabla \fat d \vv  )^2\right]\de  x \diff t
    \\&=\int_0^T {\fat \phi} \left [ \lebnorm{\fat d \times \Delta \fat d }{2}^2 + \left( \fat d \times \Delta \fat d, \fat d \times ( \nabla \fat d \vv  )\right) + \frac{1}{4} \Vert \vv  \Vert_{L^\infty(\Omega)} ^2 \lebnorm{\nabla \fat d }{2}^2\right ] \diff t \,.
\end{align*}
The pointwise convergence of $\un{E}^{k}_h\to E$ allows to pass to the limit in the occuring terms, even thought the prefactors are possibly negative. 
In all other terms, the quantities in~\eqref{discrete_envar_sum} $(\nabla v^j_h, \nabla d^j_h, \mathcal{I}_h( \ov{\un{d}}^k_h\times \Delta_h \ov{\un{d}}^k_h) )$ only occur linearly and can thus, be handled in a standard fashion. 
The switch from mass-lumping to the $L^2$ inner product and the convergence of the projection of the gradient can be dealt with as in the previous section.

Passing to the limit in~\eqref{discrete_envar_sum}  leads to the formulation
\begin{align*}
- \int_0^T & {\fat \phi} '(  E-  \int_\Omega \fat v \cdot \vv \de x) \de t   - \int_0^T {\fat \phi} \mathcal{K}(\vv) \left (  E - \mathcal{E}( \fat v , \fat d )  \right )\de t 
\\&
+ \int_0^T{\fat \phi}  \int_\Omega \fat v \cdot \t \vv - (\fat v \cdot \nabla ) \fat v \cdot \vv + \left( \fat d \times \Delta \fat d \right ) \cdot \left ( \fat d \times ( \nabla \fat d \vv  )\right)\de  x \de t 
\\&
+\int_0^T {\fat \phi}  \int_\Omega \left ((\mu_1 + \lambda^2) (\fat d \cdot (\nabla \fat v)_{sym} \fat d) (\fat d \otimes \fat d)
    + \mu_4 (\nabla \fat v)_{sym} 
   \right ) (\nabla \fat v - \nabla \vv) \diff x \diff t 
    \\&+\int_0^T{\fat \phi} \int_\Omega
    (\mu_5 + \mu_6 - \lambda^2) (\fat d \otimes (\nabla \fat v)_{sym}  \fat d)_{sym} (\nabla \fat v - \nabla \vv) 
+ \vert \fat d \times \Delta \fat d \vert^2 
     \diff x \diff t 
    \\
    &+ \int_s^t \int_\Omega \left ( \lambda [\fat d \otimes [\fat d]_x^T  [\fat d]_x \cdot\Delta \fat  d]_{sym}
    +  [\fat d \otimes \Delta \fat  d]_{skw}\right ) : \nabla \vv \diff x \diff \tau 
\leq 0 \,.
\end{align*}
From Lemma~\eqref{lem:invar}, we infer the formulation~\eqref{envarform} with~$\mathcal{K}(\vv)=\frac{1}{2} \Vert \vv\Vert_{L^\infty(\Omega)}^2$ and $$\langle \fat T^E_1(\fat d) , \fat v \rangle = \int_\Omega   \left( \fat d \times \Delta \fat d \right ) \cdot \left ( \fat d \times ( \nabla \fat d \vv  ) \right )\diff x $$ as given in Definition~\ref{def:envar}. 

\section{Computational studies\label{sec:comp}}
An efficient implementation of our discrete scheme \eqref{discrete_scheme} faces two challenges: The  non-linearity in  Scheme~\ref{discrete_scheme} does not allow to rely on well-known and effective solvers for the resulting linear systems. Secondly, all equations are coupled which leads to a high-dimensional mixed problem formulation.
Therefore, by fully linearizing and decoupling all equations we make use of an iterative fixed-point algorithm which was presented in \cite[ch. 5]{thesis_max} and follows the ideas of \cite[Algorithm $A_1$]{Prohl_Schmuck_2010}. 
This approach can be described by the following scheme.
\begin{scheme}{\ref{fp_scheme}}
(1) Let $(v^{j,0},d^{j,0}_{\circ})^T
    =
    (v^{j-1},d^{j-1}_\circ
    )^T \in U_h$ for $j> 1$. For $j=1$ let the initial values be given by $(v^{0}, d^{0})= (\projectV \fat v_0, \interpol \fat  d_0 )$.
\\
(2) For $1 \leq l\leq J$ and $(v^{j,l-1}, d^{j,l-1})\in U_h$, we want to find $(v^{j,l}, d^{j,l}_\circ)\in U_h$, such that
\begin{subequations}\label{fp_scheme}
\begin{align}
    \begin{aligned}\label{fp_scheme_a}
        (\discreteDiff_t  & v^{j,l},a)
        +((v^{j-1} \cdot \nabla) v^{j,l},a) + \frac{1}{2} ((\nabla \cdot v^{j-1} ) v^{j,l},a) 
        +(T_L^{j,l},\nabla a)  &\\&
        - v_{el}A([\projectL{\nabla d^{j-1}}]^T [d^{j-1/2,l-1}\times (d^{j-1/2,l-1} \times \Delta_h d^{j-1/2,l-1} )],a)_h = 0 
        &
    \end{aligned}\\[2ex]
    \begin{aligned}\label{fp_scheme_c}
        (\discreteDiff_t d^{j,l} &,c)_h 
        + v_{el} (d^{j-1/2,l-1}\times [\projectL{\nabla d^{j-1}}v^{j,l-1}], d^{j-1/2,l} \times c)_h & \\&
        - A(d^{j-1/2,l-1} \times \Delta_h d^{j-1/2,l-1},d^{j-1/2,l} \times c)_h & \\&
        -v_{el} ((\nabla v^{j,l-1})_{skw}d^{j-1/2, l}, c)_h & \\&
        +v_{el} \lambda (d^{j-1/2, l-1} \times [(\nabla v^{j,l-1})_{sym}d^{j-1/2,l-1}],d^{j-1/2,l} \times c)_h = 0, &
    \end{aligned}
\end{align}
\end{subequations}
 for all $(a,c)\in U_h$, where the constants $v_{el}, A$ describe the intensity of the coupling of the different physical properties.
 \\ 
 (3) Stop if $\LTwoNorm{v^{j,l}-v^{j,l-1}}+\Hsobnorm{d^{j,l}-d^{j,l-1}}{1} \leq \theta$, where $\theta$ describes the tolerance of the fixpoint-solver.
 \end{scheme}
 For the above scheme, one would solve the discrete Laplace equation computationally after solving \eqref{fp_scheme_a} and \eqref{fp_scheme_c}. This is possible since \eqref{fp_scheme_c} does not depend on $\Delta_h d^{j,l}$. Note that \eqref{fp_scheme_a} and \eqref{fp_scheme_c} could even be solved in parallel.
The Leslie stress tensor is thereby discretized as
\begin{align*}
    T_L^{j,l} \coloneqq & v_{el}(\mu_1 + \lambda^2) (d^{j,l-1} \cdot (\nabla v^{j,l})_{sym} d^{j,l-1}) (d^{j,l-1} \otimes d^{j,l-1})
    + \mu_4 (\nabla v^{j,l})_{sym} \\&
    +v_{el}(\mu_5 + \mu_6 - \lambda^2) (d^{j,l-1} \otimes (\nabla v^{j,l})_{sym}  d^{j,l-1})_{sym} \\&
    + \lambda  v_{el} A [d^{j-1/2,l-1} \otimes [d^{j-1/2,l-1}]_x^T  [d^{j-1/2,l-1}]_x \cdot \Delta_h d^{j-1/2,l-1}]_{sym}
    \\&
    + v_{el} A [d^{j-1/2,l-1} \otimes \Delta_h d^{j-1/2,l-1}]_{skw}\,.
\end{align*}
Any fixed-point of scheme \eqref{fp_scheme} solves the highly non-linear scheme \eqref{discrete_scheme}.
By standard arguments, one can derive conditional convergence of our fixpoint algorithm, i.e. $u^{j,l} \to u^j \in U_h$ as $l\to \infty$ for a sufficiently small $k,h$. For more details, we refer to \cite{thesis_max}.

Since we introduced the constants $v_{el},A$ for the physical coupling of the quantities, the energy law changes slightly and we rescale the total energy to
\begin{equation*}
  \mathcal{E}(\fat v , \fat d) 
\coloneqq
 \frac{1}{2}\lVert \fat v\rVert_{L^2(\Omega)}^ 2 + \frac{A}{2}\LTwoNorm{\nabla  \fat d}^2 \,.
\end{equation*}
The computational studies consist of three numerical experiments which have been standard benchmarks for past numerical methods of nematic liquid crystals (see e.g. \cite{lasarzik_main,Becker_Prohl_2008,Liu_Walkington,cabralesFEM}).
As a simplified version of our discrete scheme we consider one analogous to the Ericksen--Leslie interactions in \cite{lasarzik_main} and similar to \cite[Algorithm 4.1]{Becker_Prohl_2008}.
The simplified scheme is attained by choosing a reduced (discrete) version of the Leslie stress tensor, \textit{i.e.}
\begin{align}\label{def:simple_model}
    (T_L^j, a) = \nu (\nabla v^j, \nabla a).
\end{align}
Accordingly the fourth and last term in the director equation \eqref{discrete_scheme_c} of our scheme are also neglected.
Both schemes can be solved by the fixed-point iteration described above. The tolerance of the iterative fixpoint solver was set to $\theta = 10^{-6}$.
The python implementation of the fixed-point solver relies on the API of the finite element package FEniCS (\textit{cf.} \cite{Fenics1,Fenics2}). 
The implementation can be found in \cite{fenics_code}.

If not mentioned otherwise, we employ homogeneous Dirichlet boundary (no-slip) conditions for the velocity and non-homogeneous Dirichlet boundary conditions for the director as in \eqref{system}.
The parameter choices for the following experiments can be found in table \ref{fig:parameter_table}. Our choice of parameters is rather exemplary and analogous to the experiments considered in \cite{lasarzik_main,Becker_Prohl_2008,Liu_Walkington,cabralesFEM}. In particular regarding the choice of $\mu$, other choices are most definitely conceivable.
Note that the dissipative character of the system is fulfilled, \textit{i.e.,} $\mu_4 >0, \quad \mu_5+\mu_6 -\lambda^2>0, \quad \mu_1 + \lambda^2 >0$. Further, $\nu = \mu_4$ is fulfilled which shall allow a comparison of the models.
Since the convergence of our fixed-point solver relies on a contraction property, a smaller choice of our temporal discretization parameter $k$ might sometimes still lead to faster computation time, if this leads to less total iterations of the fixed-point solver itself.

\subsection{A smooth example in two spatial dimensions}\label{experiment_1}
We consider the domain $\Omega = (-1,1)^2$ equipped with the initial conditions
\begin{align*}
    d_0 (x) &=
    \begin{pmatrix}
     &\sin \left( 2 \pi \left(\cos(x_1)- \sin(x_2) \right) \right) \\
   & \cos \left( 2 \pi \left(\cos(x_1)-\sin(x_2) \right) \right) 
    \end{pmatrix}
    \text{ for }x\in \Omega,
    \\
    v_0 &=0.
\end{align*}
Although the above analysis is only done for three spatial dimensions, the results can be transferred to the two-dimensional case as well. In this case, the cross product has to be understood in the sense of the matrix $ I- \fat d \otimes \fat d$ (see \eqref{weak:d}), which in three dimensions fulfills $ I- \fat d \otimes \fat d = [\fat d]^T_x [\fat d]_x $.
In this setting, at least for the simplified model, a smooth solution exists except for finitely many points in time (\textit{cf.} \cite[Thm. 1.3]{smooth_solutions_el}).
The results for our simplified model (see \eqref{def:simple_model}, Fig. \ref{fig:smooth}) deviate from the results of previous authors (\textit{cf.} \cite{Becker_Prohl_2008}) due to our Dirichlet boundary conditions. The director exhibits a long-term self-alignment which causes a fluid flow. In opposite to e.g. \cite{Becker_Prohl_2008}, this alignment of the director is not uniform. The same holds for our full Ericksen--Leslie model (see Fig. \ref{fig:smooth_general}) except for the fact that induced velocity field exhibits less symmetry due to the anisotropic dynamics of our Leslie-stress tensor.
\begin{table}[t]
    \centering
    \begin{tabular}{c|c|c|c|c|c| c|c|c|c}
         & $\Omega$& $h$ & $k$ & $\mu_1$ & $\mu_4$ & $\nu$   & $\mu_5$, $\mu_6$, $\lambda$ & A & $\nu_{el}$ 
        \\
        \hline
        Experiment 1 & 
        $(-1,1)^2$& 
        $2^{-5}$ & 0.00025 & 1 &  0.1& 0.1& 1&  1& 1
        \\
        Experiment 2 & 
        $(-0.5,0.5)^3$&
         $2^{-4}$ & 0.00025 & 1  & 1& 1& 1&   1& 0.25 
        \\
        Experiment 3 & 
        $(-0.5,0.5)^3$&
         $2^{-4}$ & 0.00025 & 1  & 1& 1&   1& 0.1& 1
    \end{tabular}
    \caption{Parameter choices for the experiments}
    \label{fig:parameter_table}
\end{table}

\begin{figure}[t]
    \centering
    \includegraphics[width = .49 \textwidth]{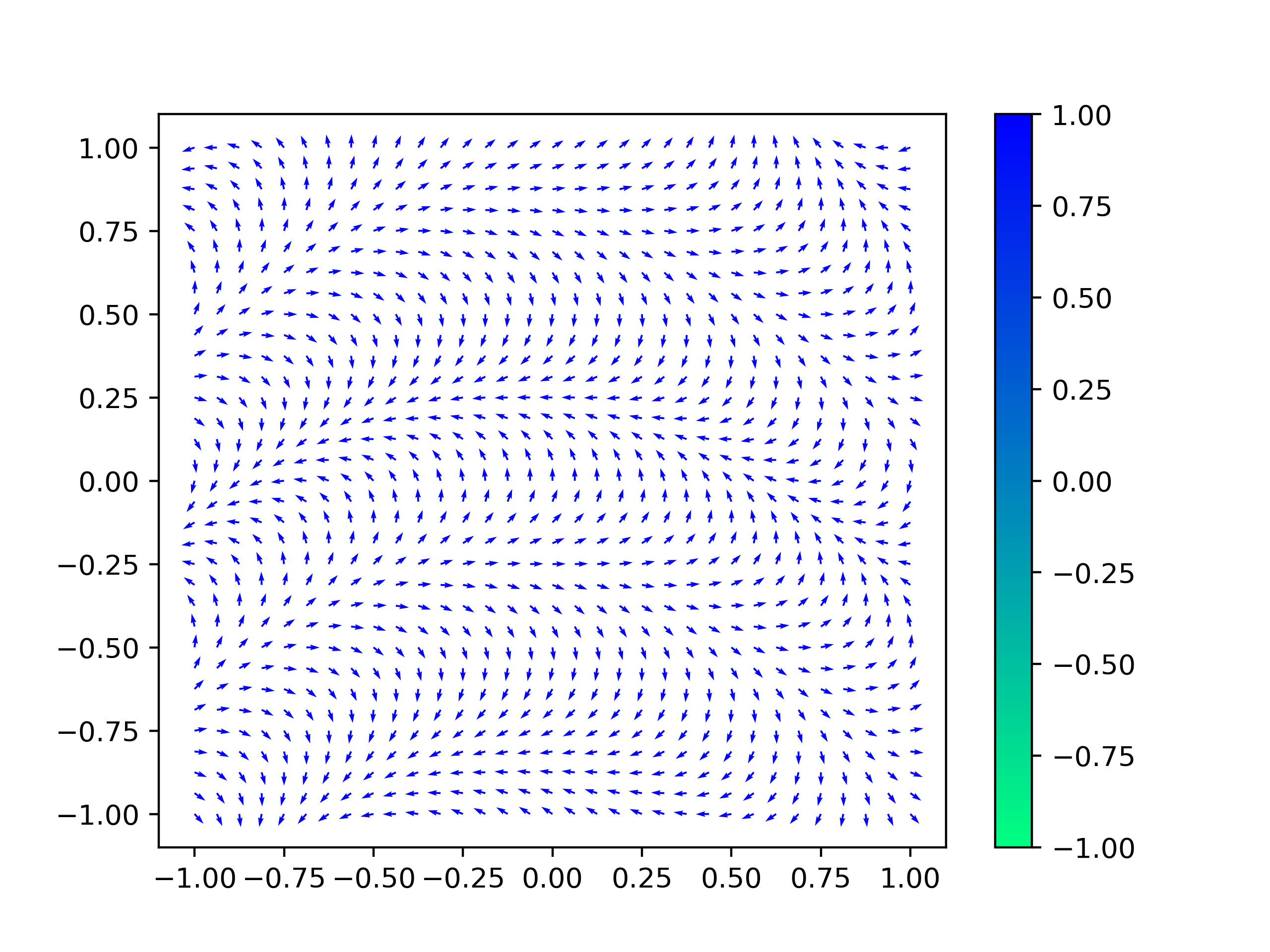}
    \includegraphics[width = .49 \textwidth]{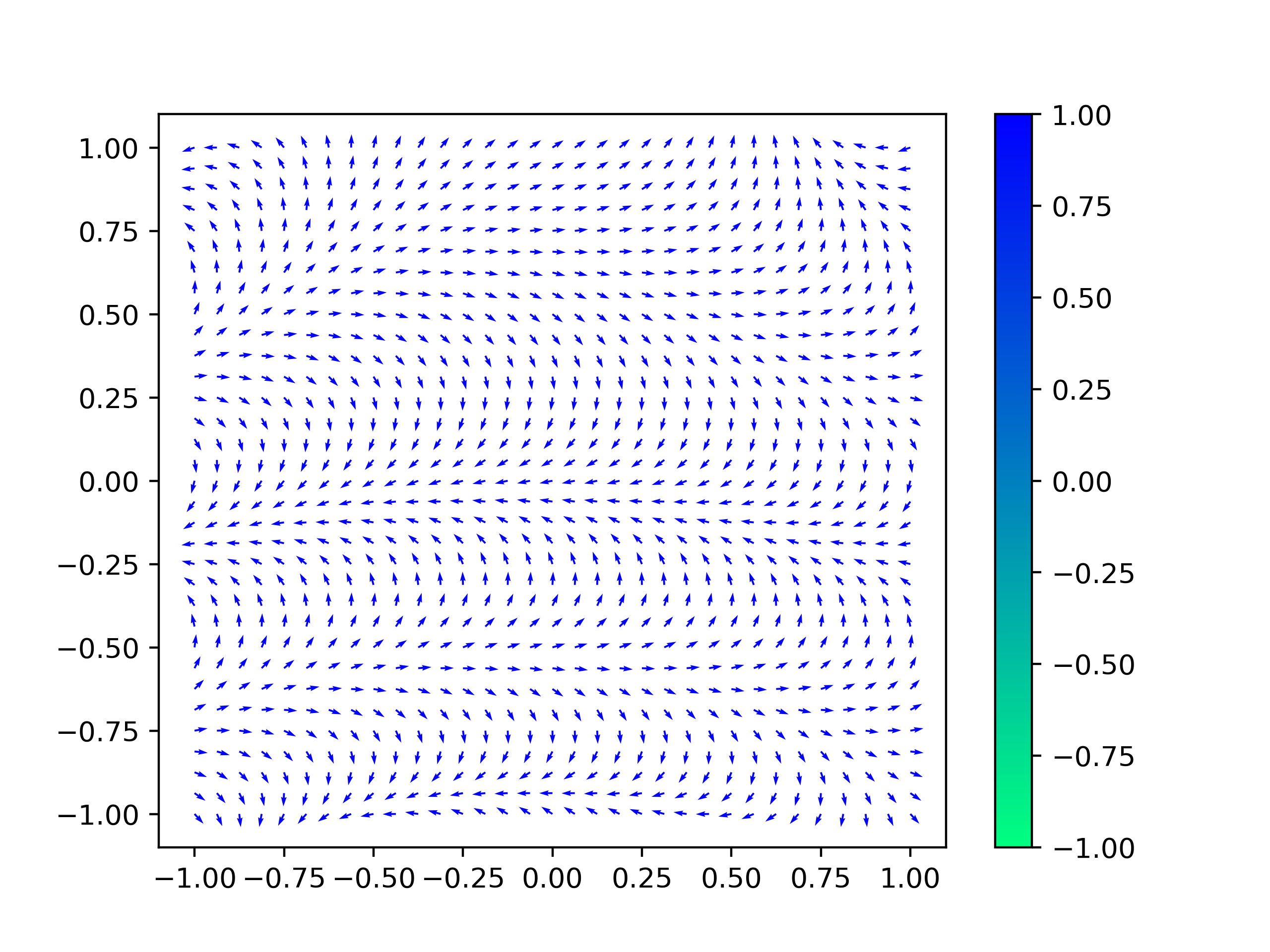}
    \includegraphics[width = .49 \textwidth]{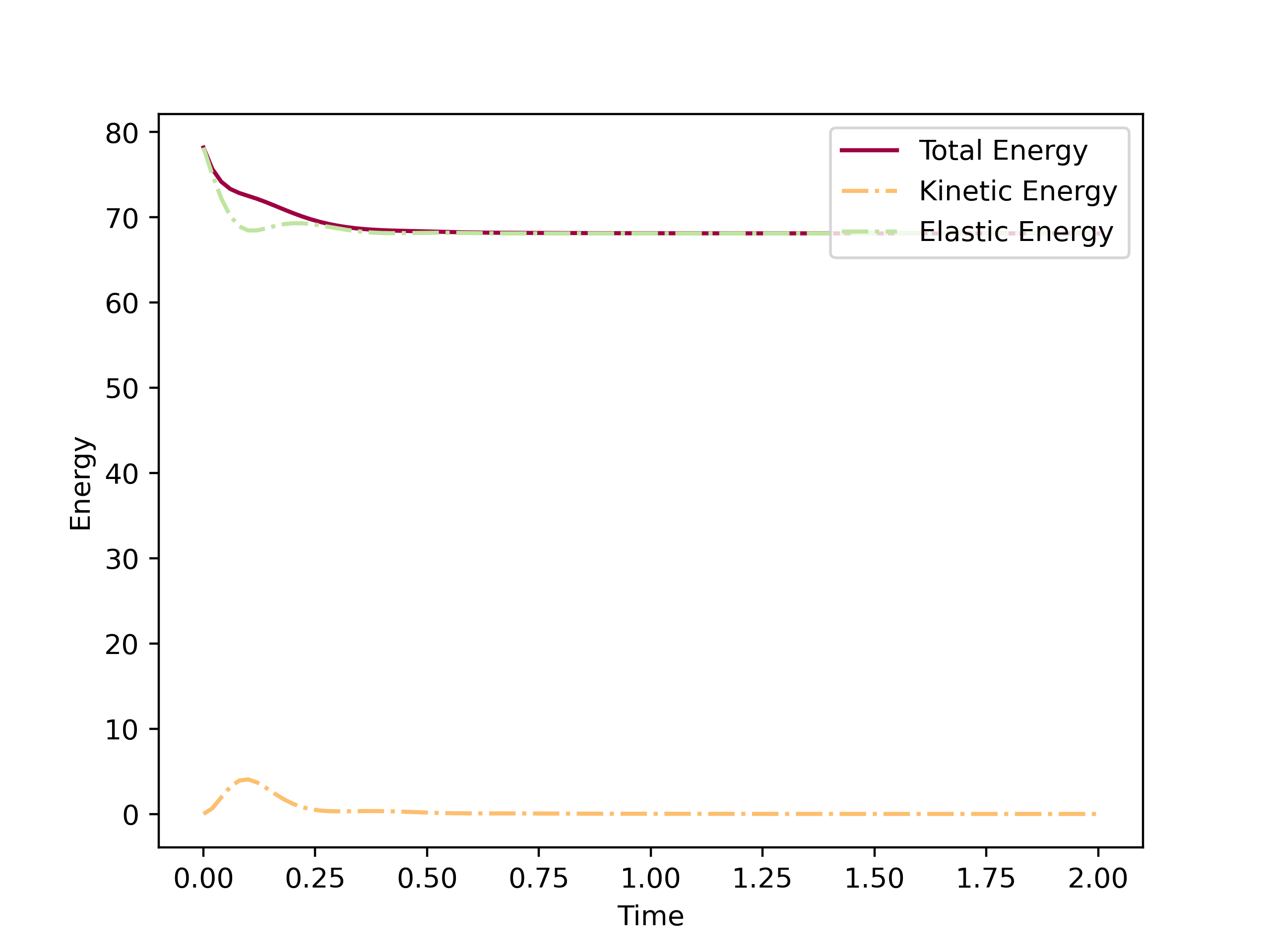}
    \includegraphics[width = .49 \textwidth]{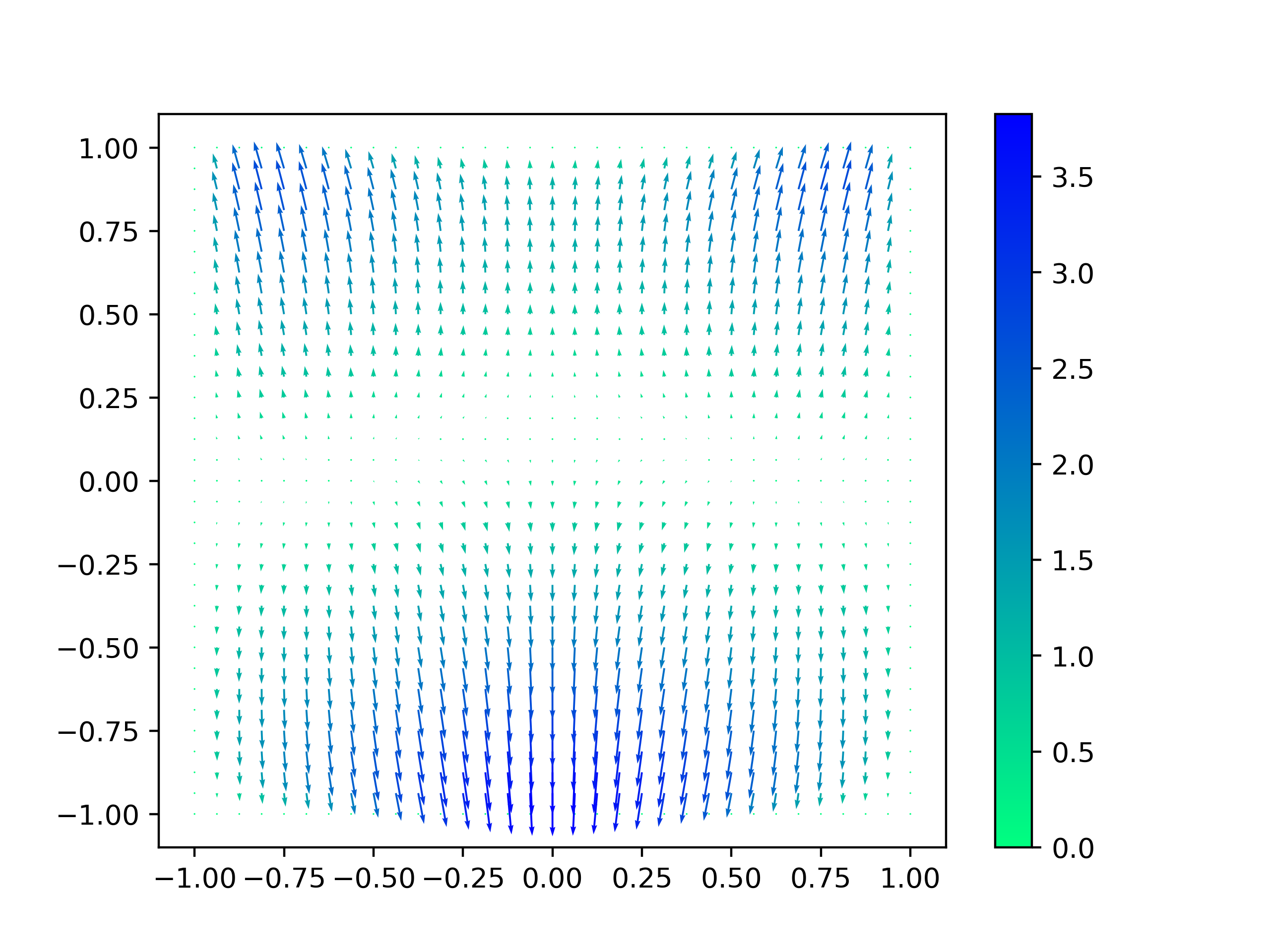}
    \caption{Experiment \ref{experiment_1}, simplified model \eqref{def:simple_model}: Evolution of the director at time $t=0, 2.0$ (from left to right), evolution of the energy (bottom left), velocity field at times $t=0.2$ (bottom right). The colour marks the magnitude of the vectors.
    }
    \label{fig:smooth}
\end{figure}
\begin{figure}[t]
    \centering
    \includegraphics[width = .49 \textwidth]{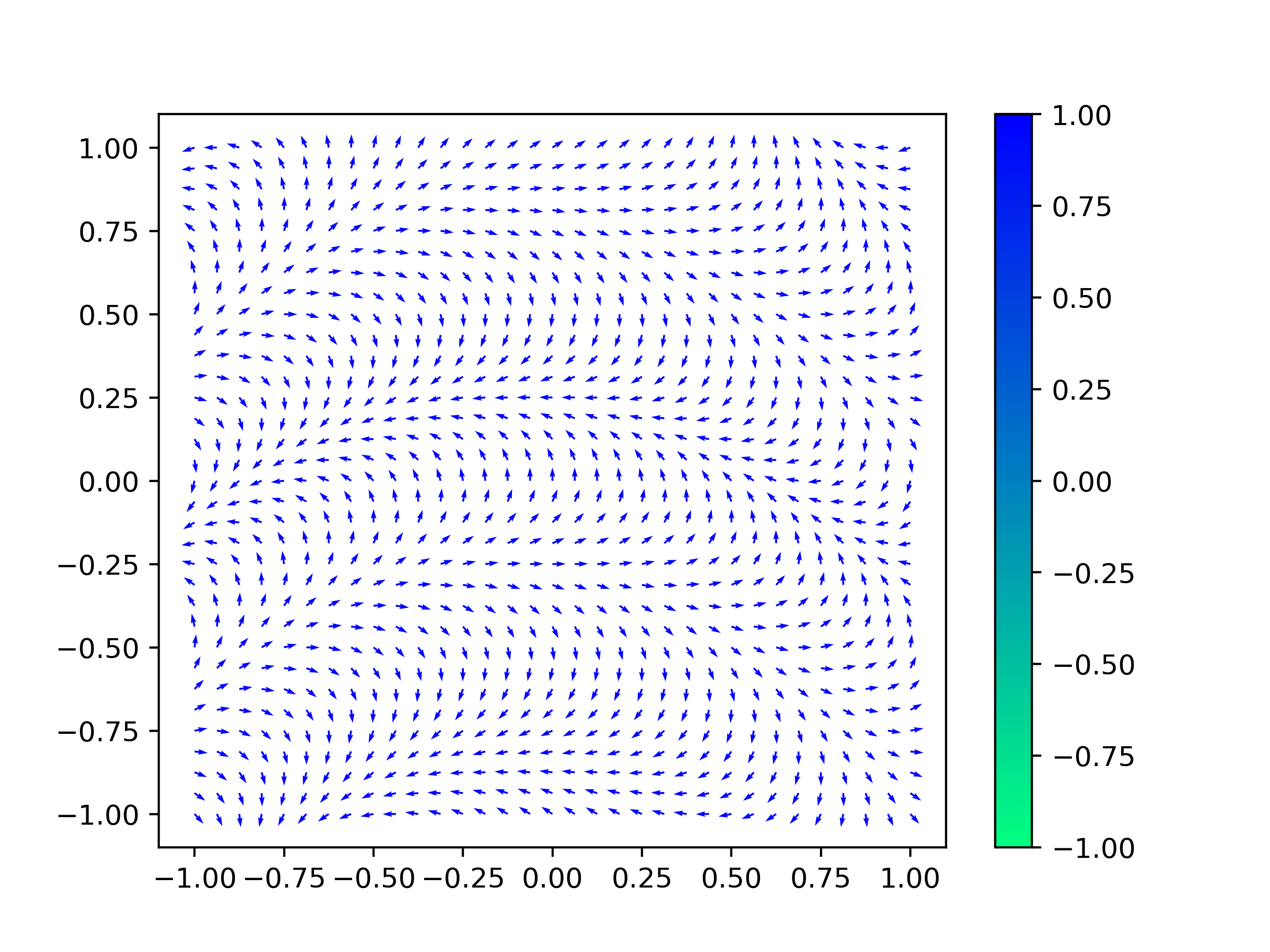}
    \includegraphics[width = .49 \textwidth]{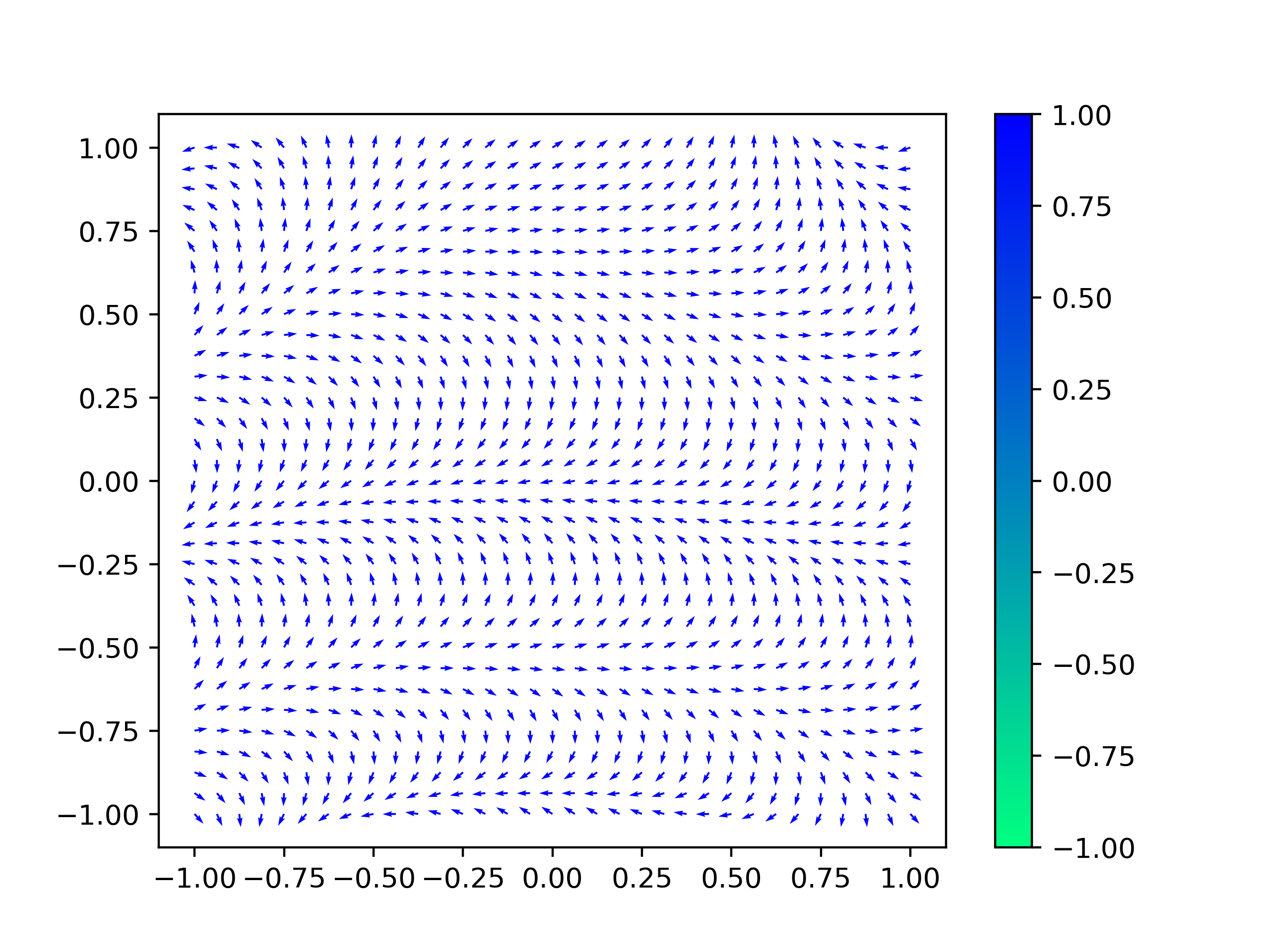}
    \includegraphics[width = .49 \textwidth]{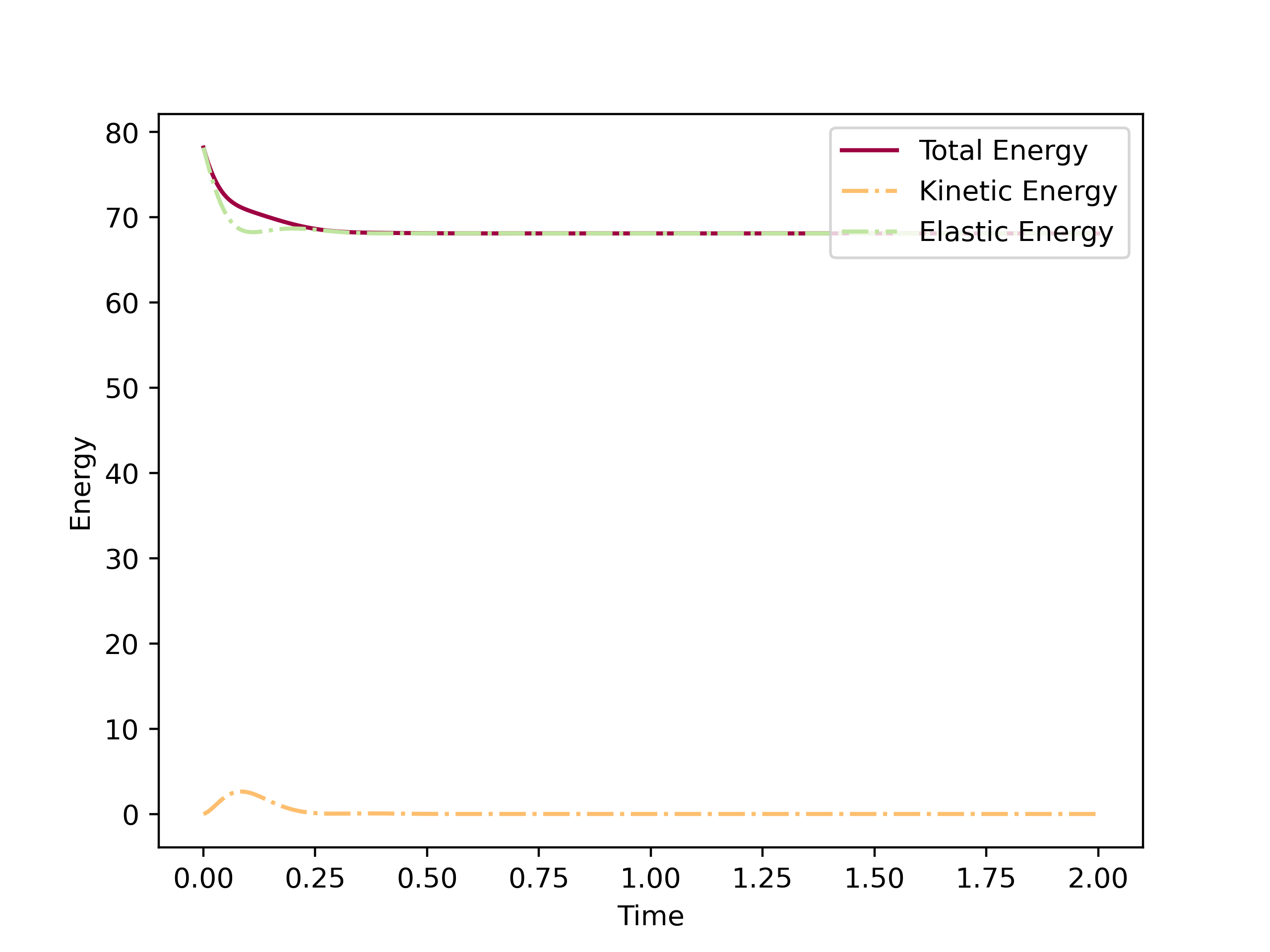}
    \includegraphics[width = .49 \textwidth]{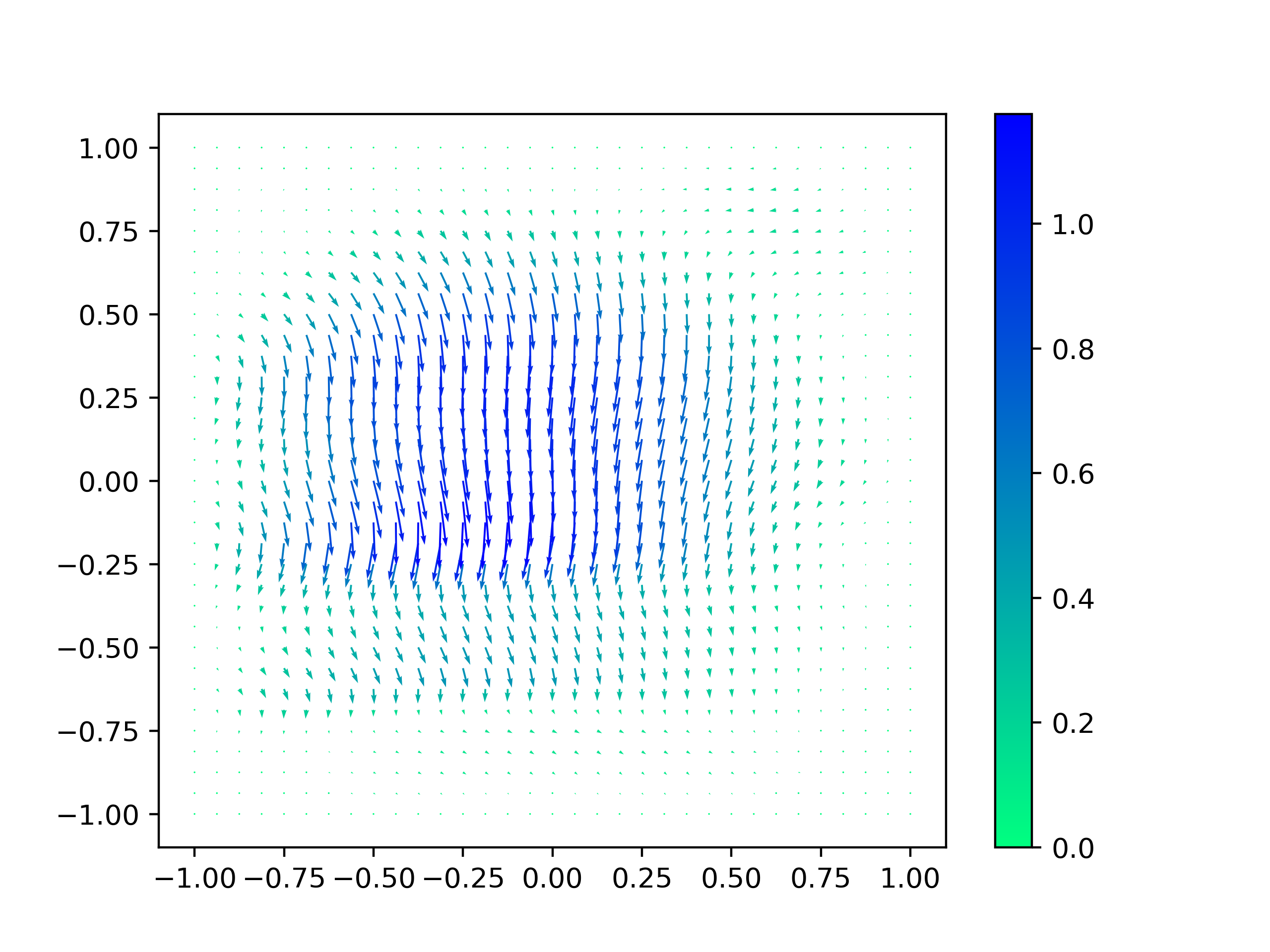}
    \caption{Experiment \ref{experiment_1}, model \eqref{discrete_scheme}: Evolution of the director at time $t=0, 2.0$ (from left to right), evolution of the energy (bottom left), velocity field at times $t=0.2$ (bottom right). The colour marks the magnitude of the vectors.
    }
    \label{fig:smooth_general}
\end{figure}

\subsection{Annihilation of two defects}\label{experiment_2}
On the three-dimensional domain $\Omega = (-0.5 ,0.5)^3$, we set the initial director  to
\begin{align*}
d_0 &=
    \begin{cases}
    \frac{
    \Tilde{d}_0(x)}{
    \norm{\Tilde{d}_0(x)}}
    , \text{ if }  \norm{\Tilde{d}_0(x)}>0 \\
    (0,0,1)^T , \text{ otherwise},
    \end{cases}
\\
\text{for  }
\Tilde{d}_0(x) &=(4 x_1^2+4 x_2^2-0.25, 2x_2, 0)^T , 
\end{align*}
where the locations with $d_0 = (0,0,1)^T$ are considered \textit{defects} of the liquid crystal.
The initial velocity is set to $v_0 = 0$.

It can be observed that for both models the defects slowly annihilate by rotation of the director (see Fig. \ref{fig:annihilation} and Fig. \ref{fig:annihilation_general}). This causes rough swirls to develop in the velocity field. The evolution of the director qualitatively agrees with the results in \cite{Becker_Prohl_2008, lasarzik_main, Liu_Walkington, cabralesFEM}\footnote{Since the norm conservation is fulfilled precisely at every node in our case, we cannot simulate this experiment in two dimensions as done in \cite{Becker_Prohl_2008} because the defects would otherwise be conserved.}.
For the full Ericksen--Leslie model, we observe again anisotropic dynamics transferred to the velocity field as well.

\begin{figure}[t]
    \centering
    \includegraphics[width= .49 \textwidth]{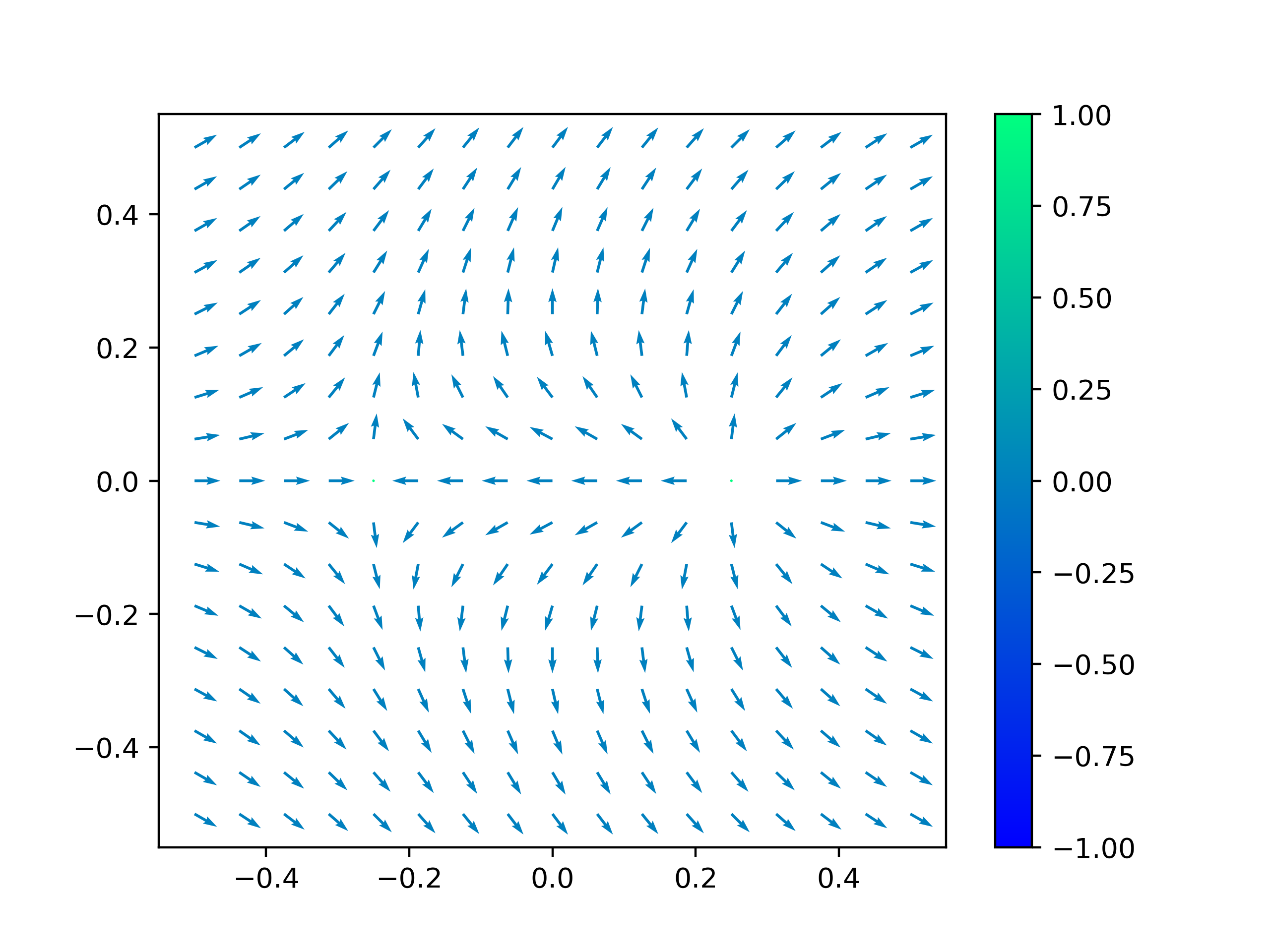}
    \includegraphics[width= .49 \textwidth]{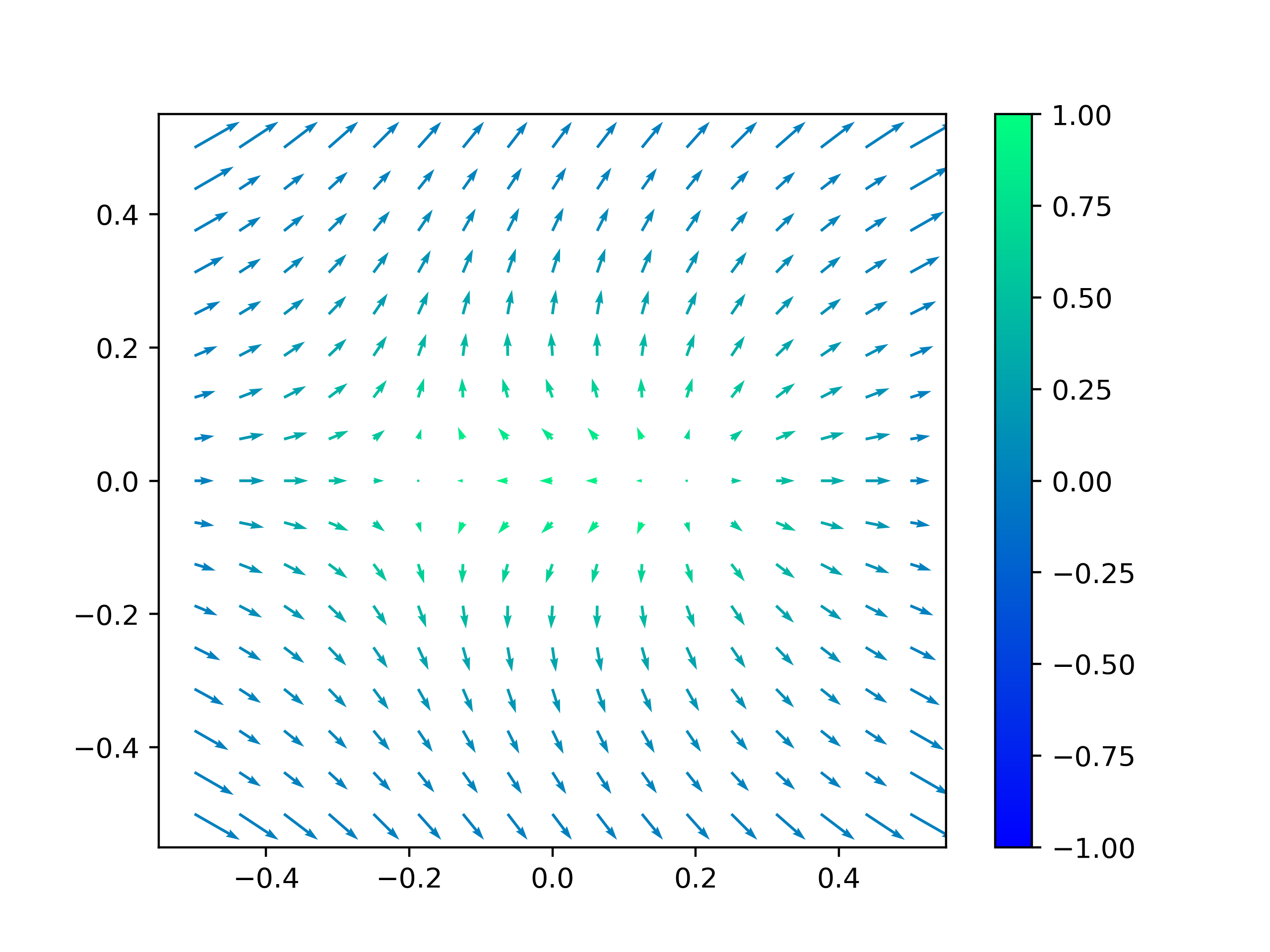}
    \includegraphics[width= .49 \textwidth]{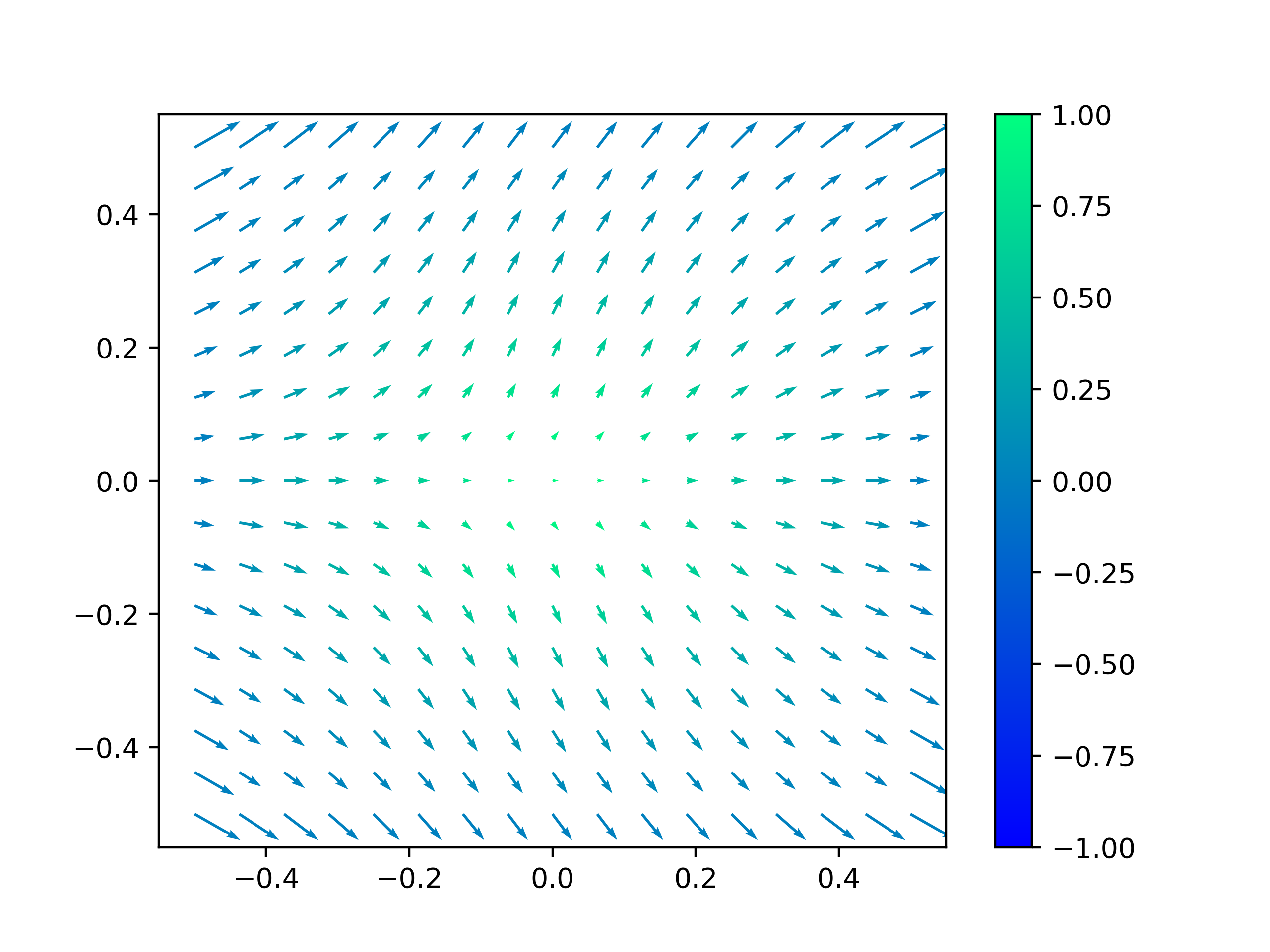}
    \includegraphics[width= .49 \textwidth]{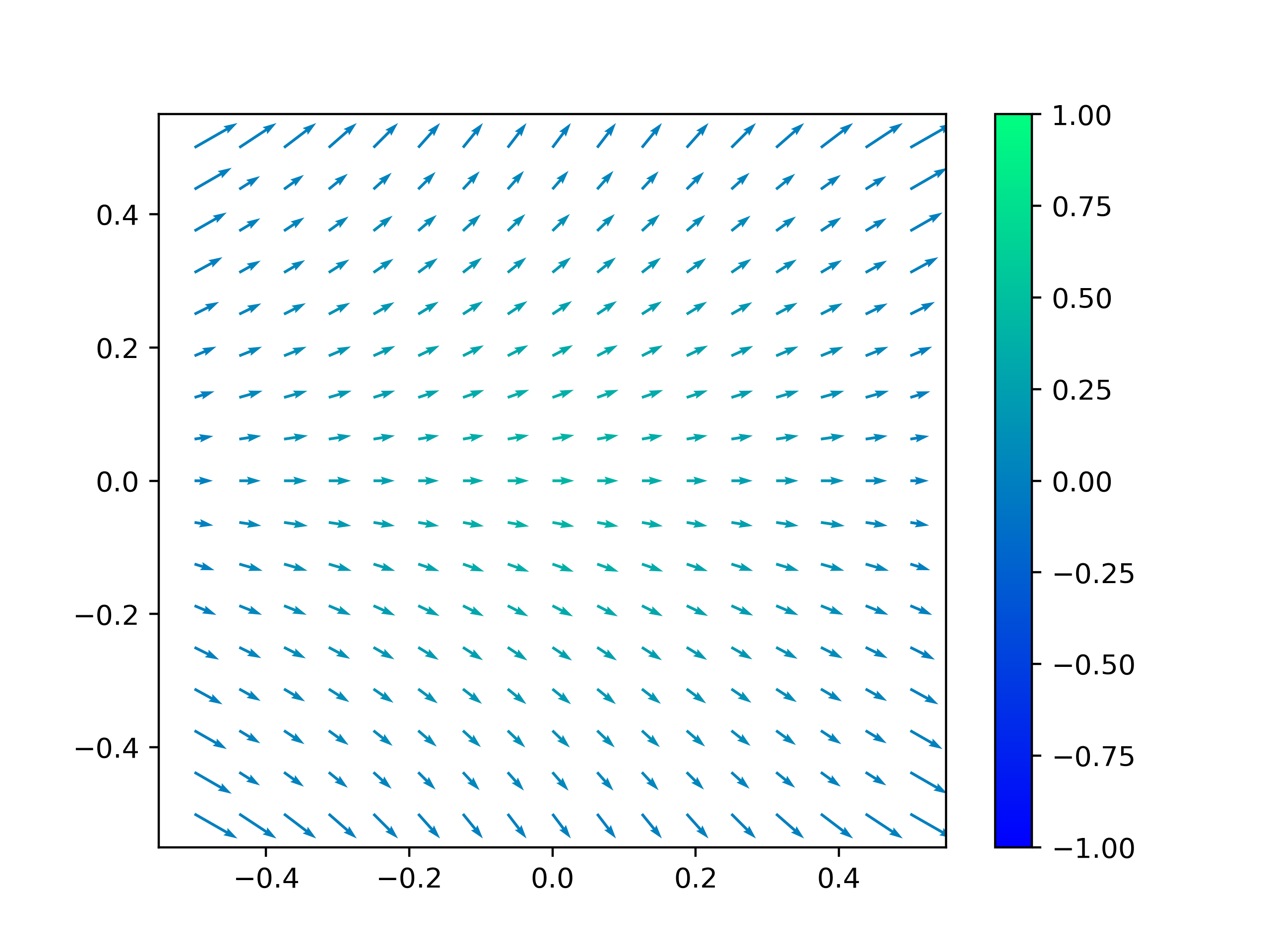}
    \includegraphics[width= .49 \textwidth]{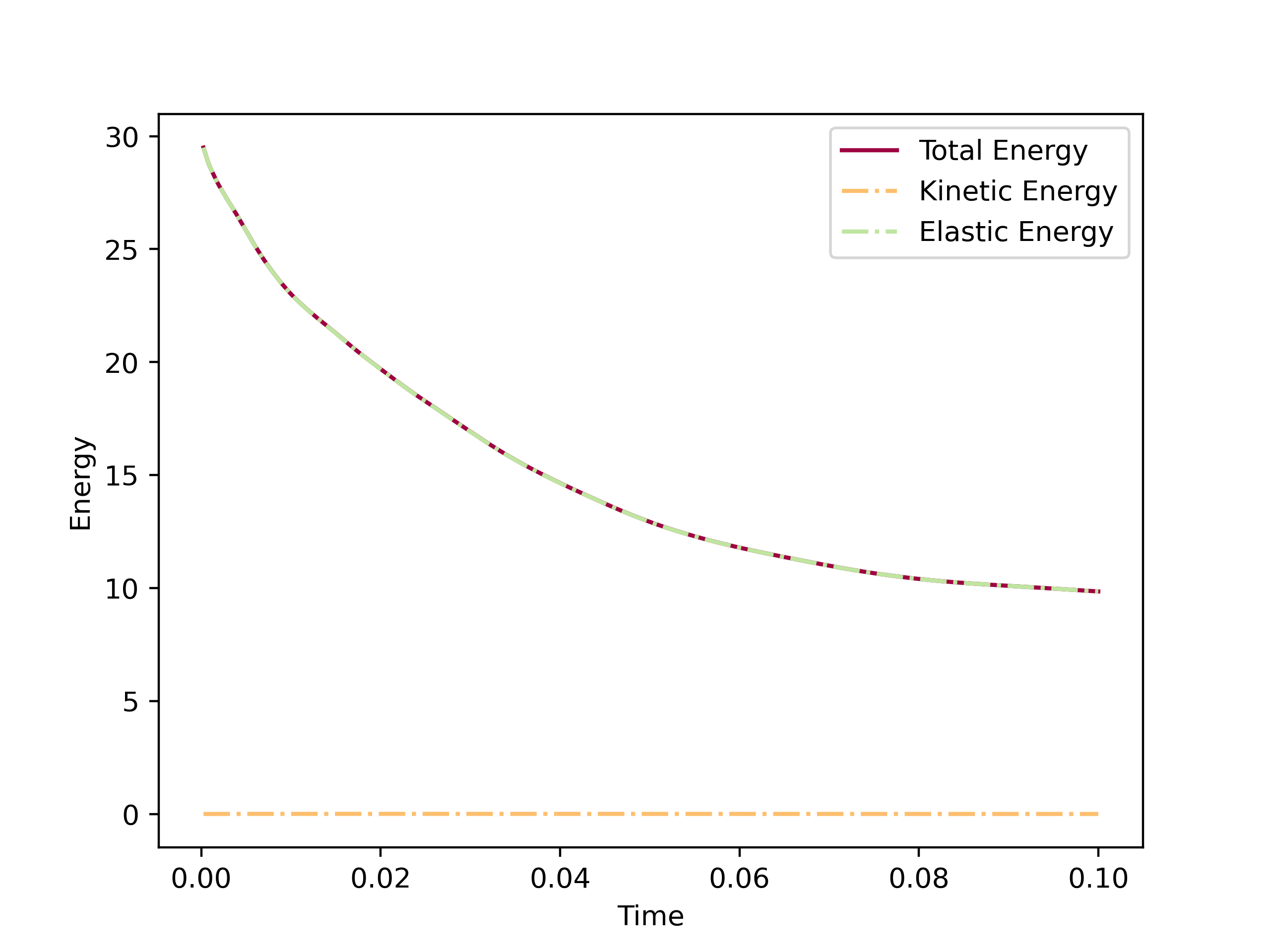}
    \includegraphics[width= .49 \textwidth]{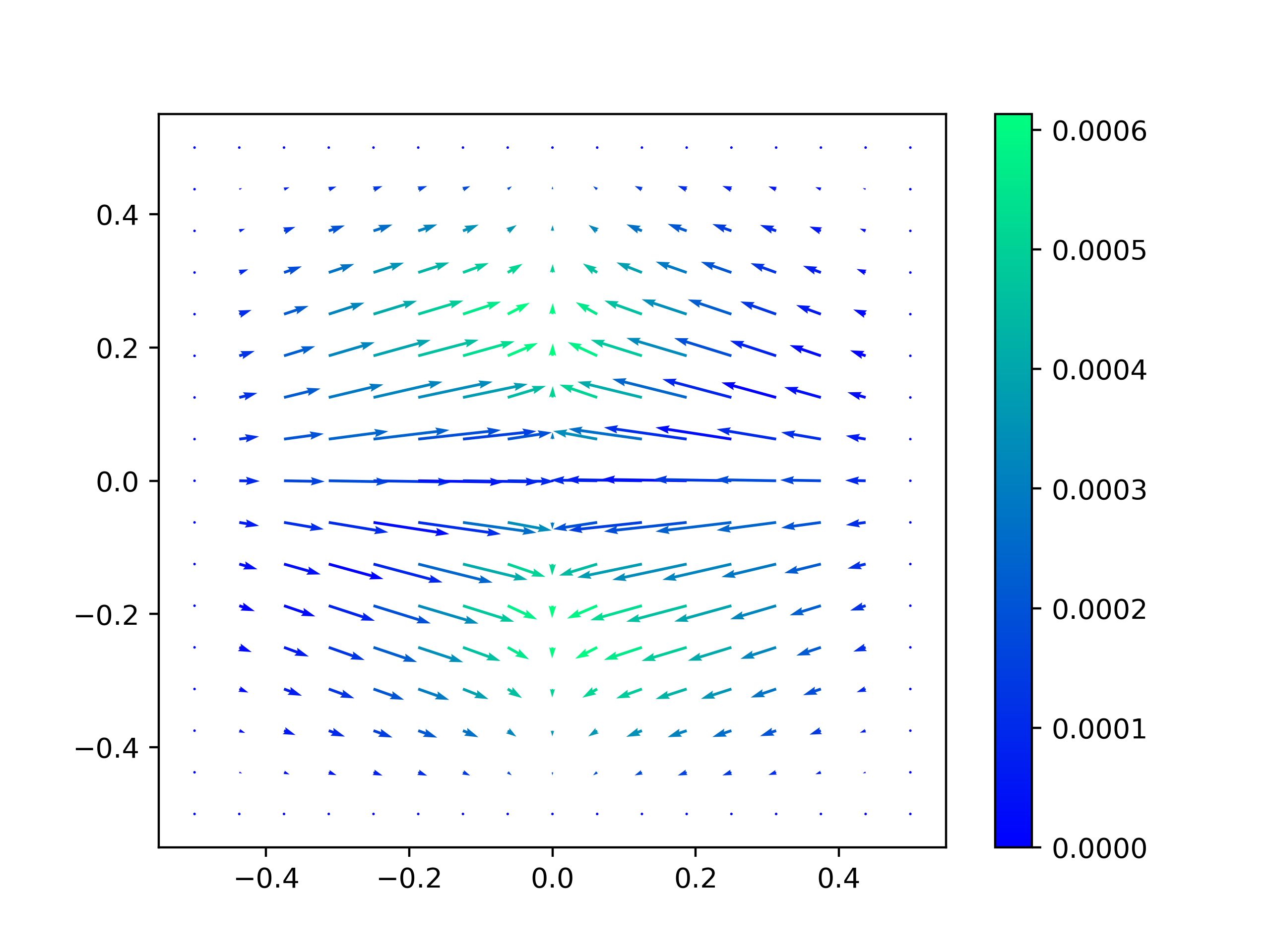}
    \caption{Experiment \ref{experiment_2}, simplified model \eqref{def:simple_model}: Evolution of the director in the plane $x_3=0$ at time $t=0, 0.03, 0.05, 0.1$ (from left to right, from top to bottom), evolution of the energy (bottom left), velocity field in the plane $x_3=0$ at time $t= 0.05$ (bottom right).
    }
    \label{fig:annihilation}
\end{figure}

\begin{figure}[t]
    \centering
    \includegraphics[width= .49 \textwidth]{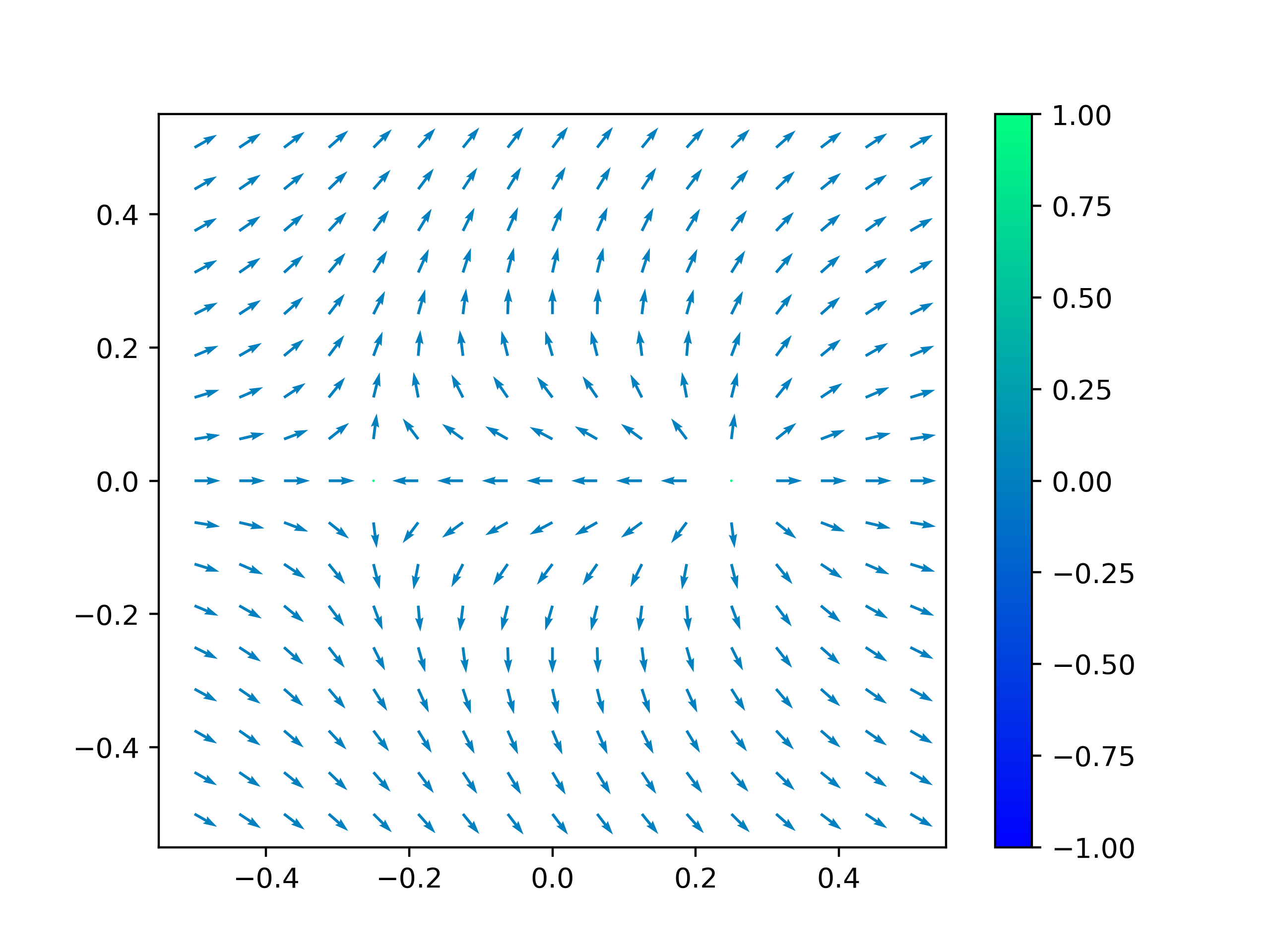}
    \includegraphics[width= .49 \textwidth]{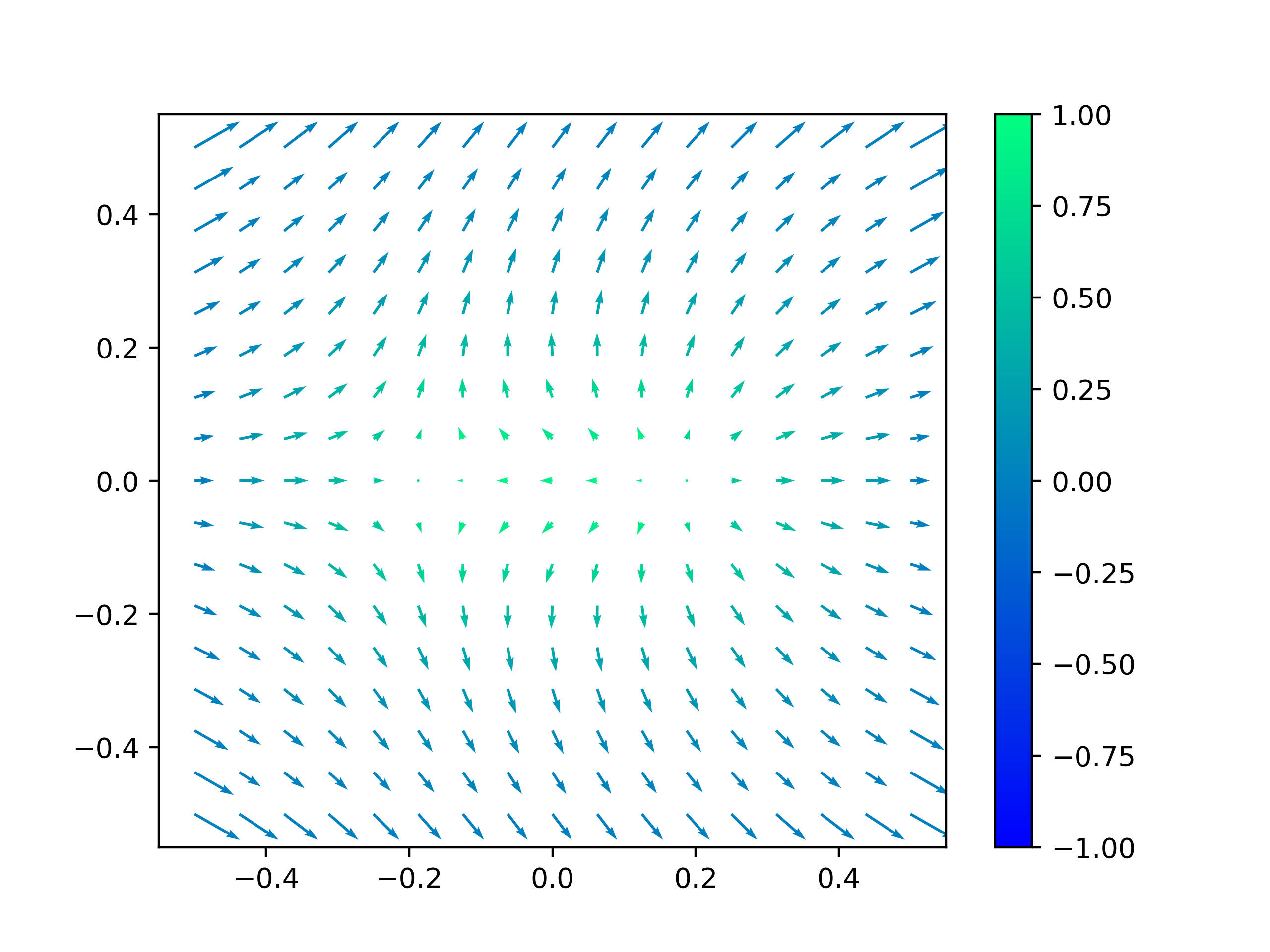}
    \includegraphics[width= .49 \textwidth]{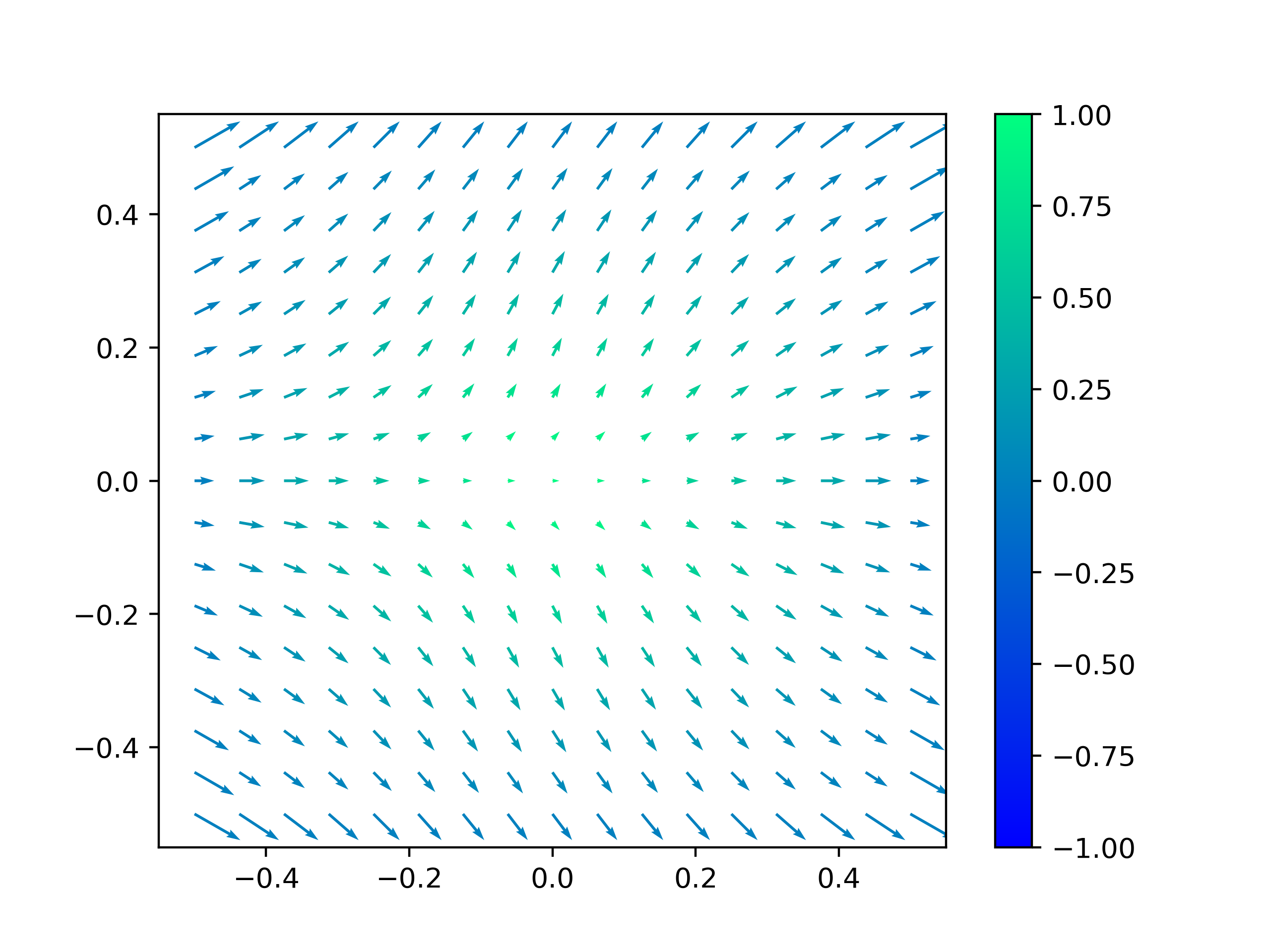}
    \includegraphics[width= .49 \textwidth]{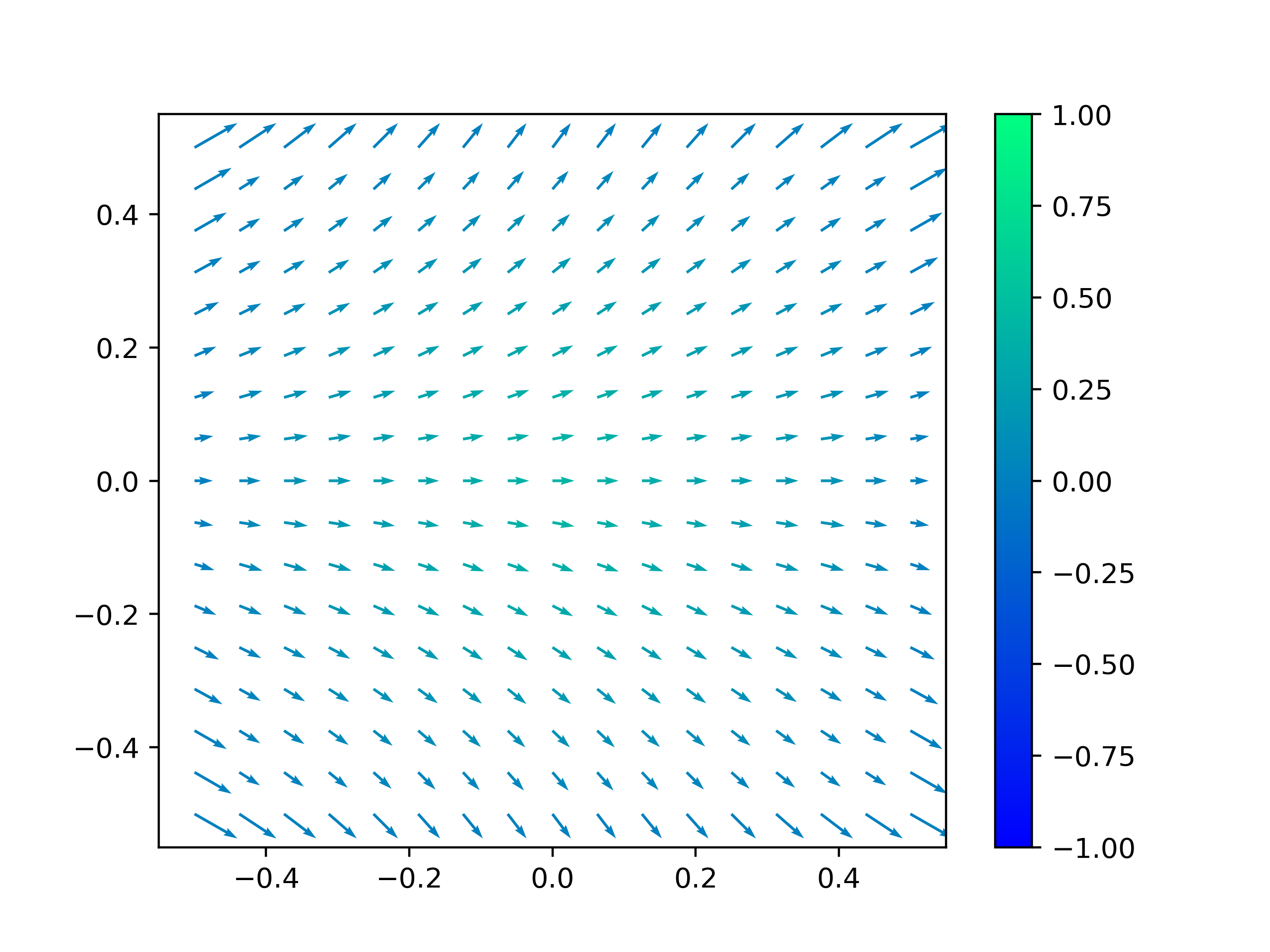}
    \includegraphics[width= .49 \textwidth]{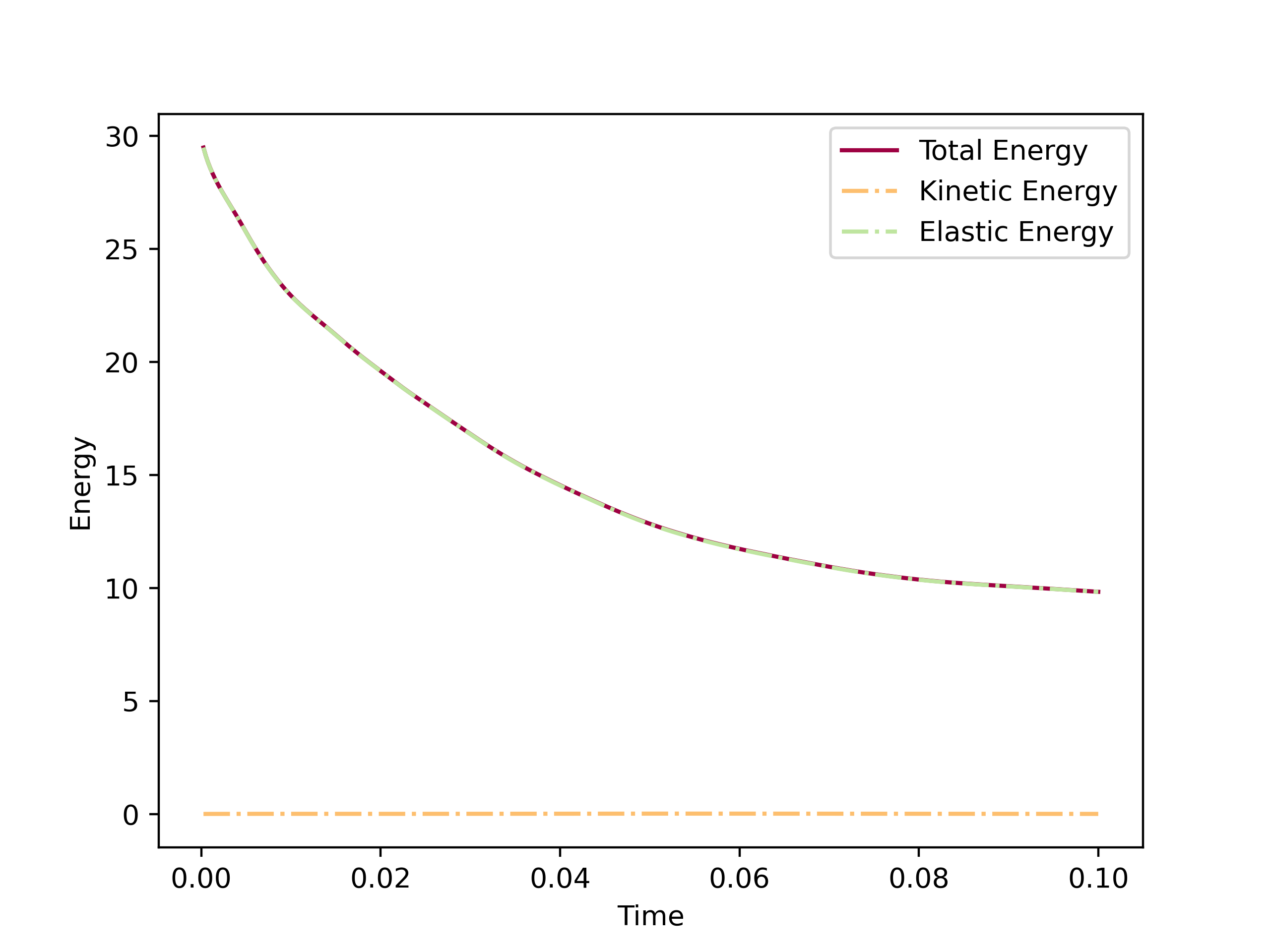}
    \includegraphics[width= .49 \textwidth]{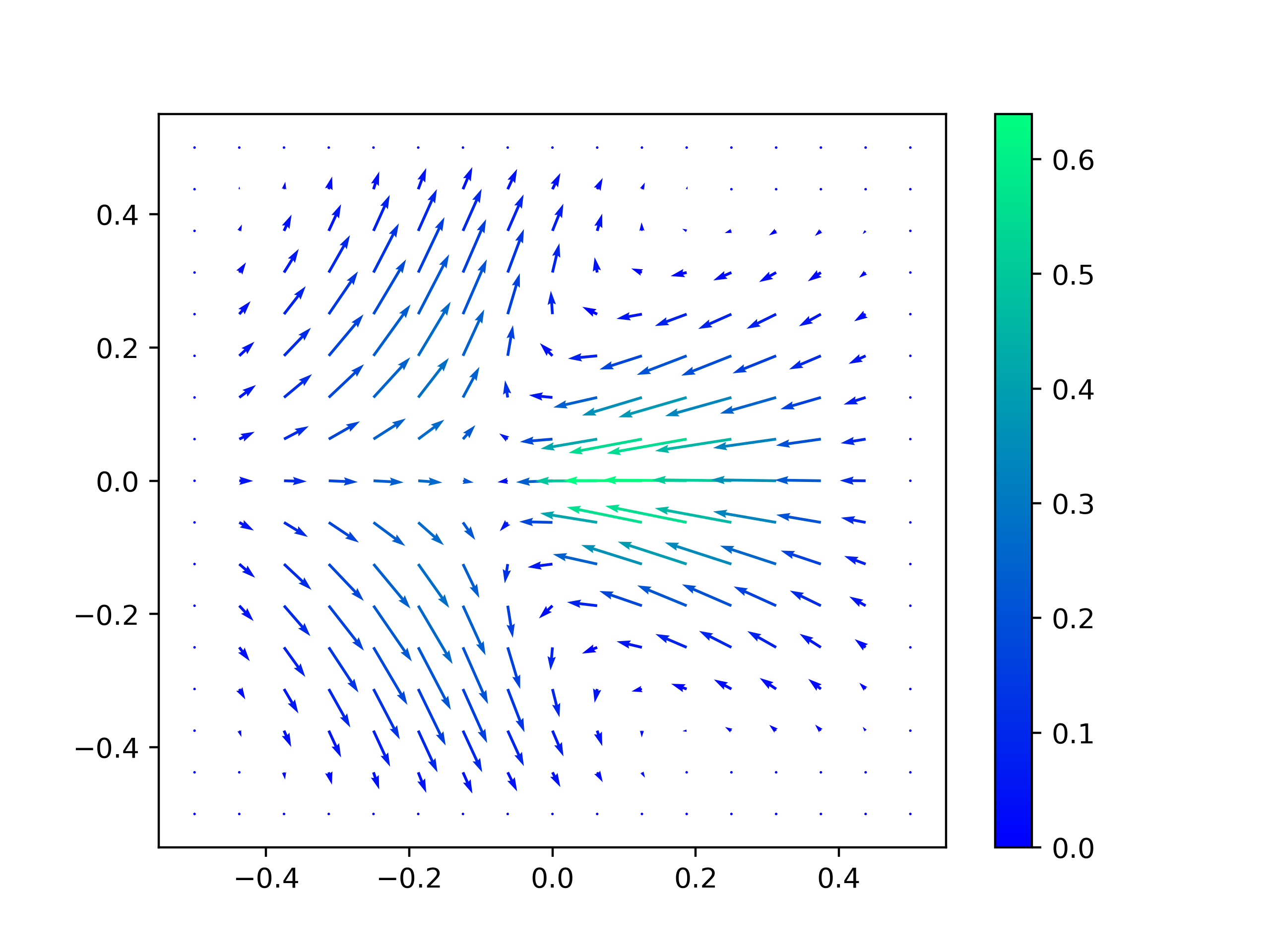}
    \caption{Experiment \ref{experiment_2}, model \eqref{discrete_scheme}: Evolution of the director in the plane $x_3=0$ at time $t=0, 0.03, 0.05, 0.1$ (from left to right, from top to bottom), evolution of the energy (bottom left), velocity field in the plane $x_3=0$ at time $t= 0.05$ (bottom right).
    The colour marks the (absolute) value of the z-component.
    }
    \label{fig:annihilation_general}
\end{figure}
\subsection{Annihilation of two defects in a rotating flow}
The setting of the experiment equals the setting of the one before except for the choice of the initial condition and boundary condition of the velocity.
Instead we choose 
\begin{align*}
    v_0 &= 10 (- x_2, x_1, 0)^T \text{ for all }x \in \Omega, \\
    v &= 10 (- x_2, x_1, 0)^T \text{ for all } x \in \partial \Omega.
\end{align*}
Note that the inhomogeneous Dirichlet boundary conditions for the velocity are a deviation from our so far considered setting such that we can only expect the energy law \eqref{discrete_energy_equ} to hold with a modified right-hand side accounting for the boundary conditions (see Fig. \ref{fig:annihilation_rot}).
In Fig. \ref{fig:annihilation_rot} we can observe that in the simplified model the defects are swirled anticlockwise around the center by the velocity field before annihilation. This is in qualitative agreement with the results obtained by \cite{lasarzik_main, Liu_Walkington} and also holds for the more general model (see Fig. \ref{fig:annihilation_rot_gen}).

\begin{figure}[t]
    \centering
    \includegraphics[width= .49 \textwidth]{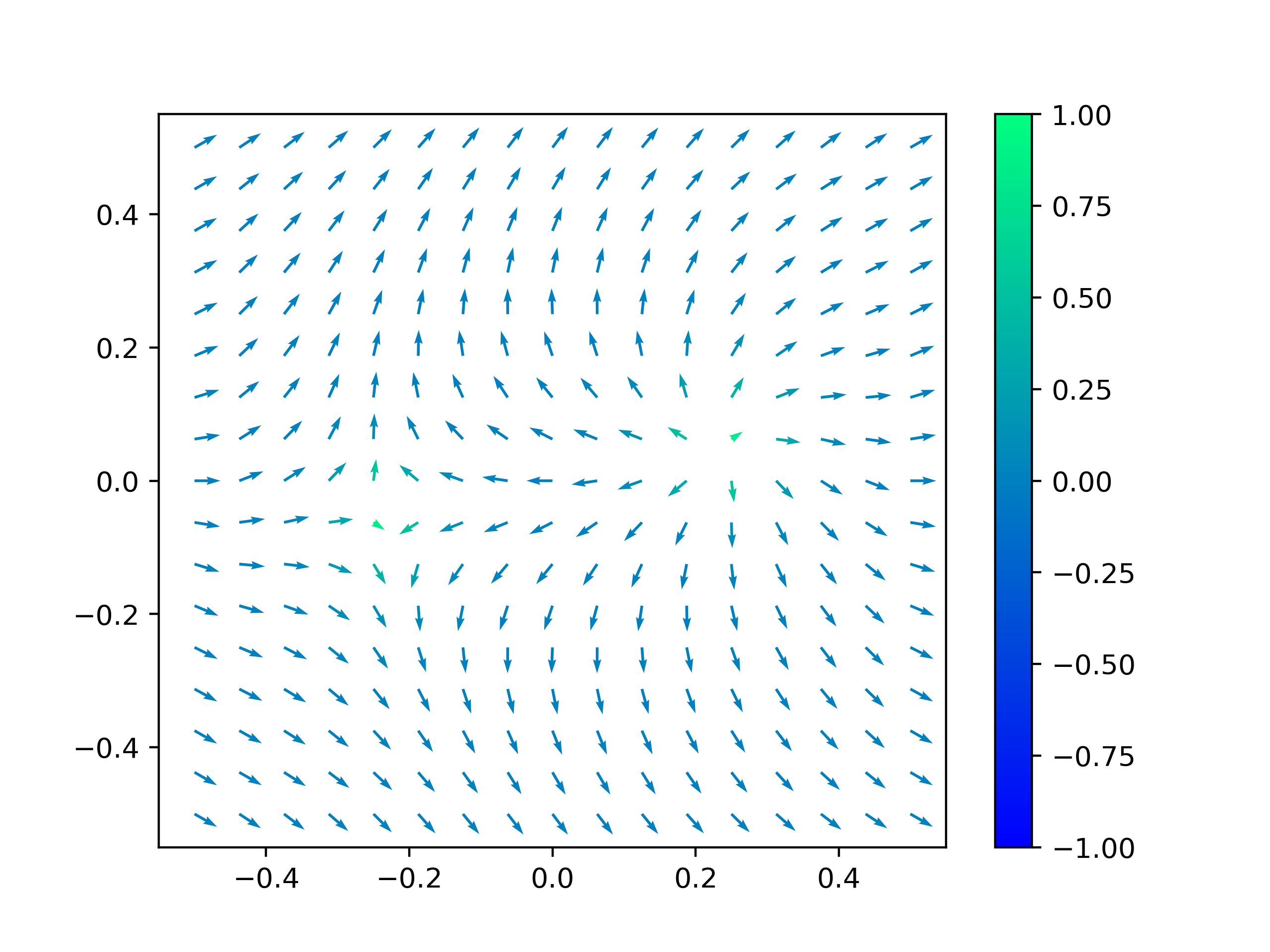}
    \includegraphics[width= .49 \textwidth]{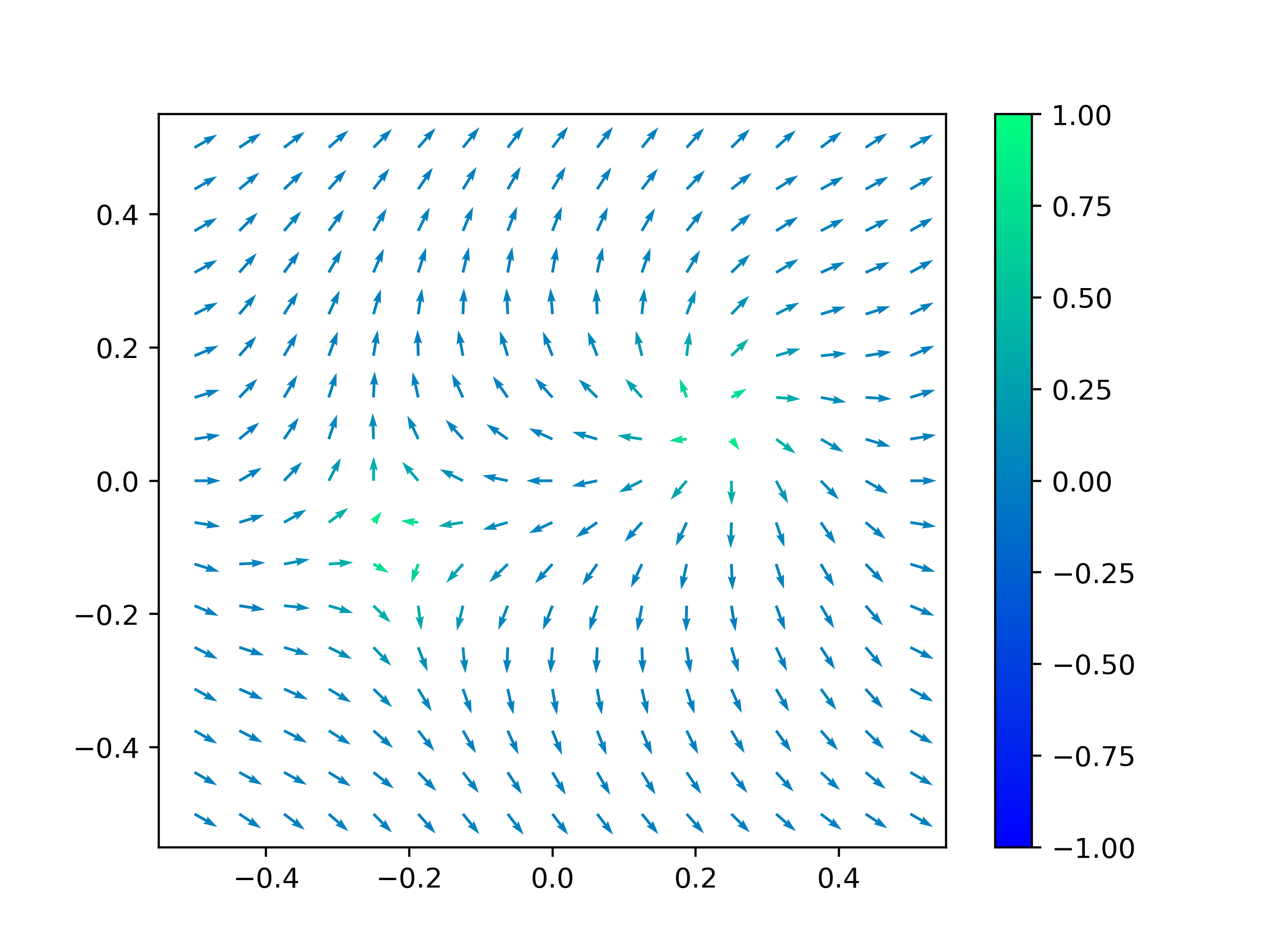}
    \includegraphics[width= .49 \textwidth]{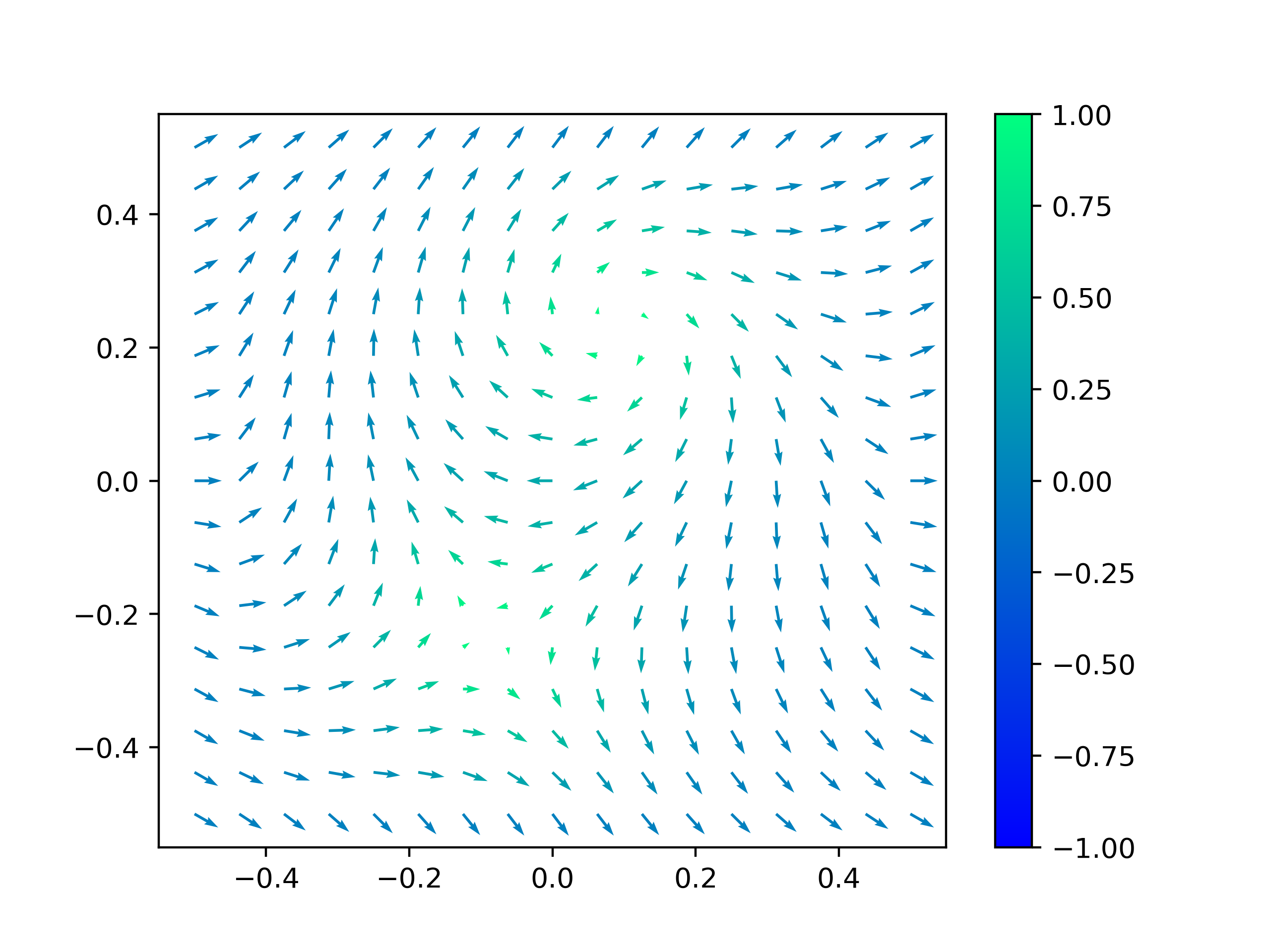}
    \includegraphics[width= .49 \textwidth]{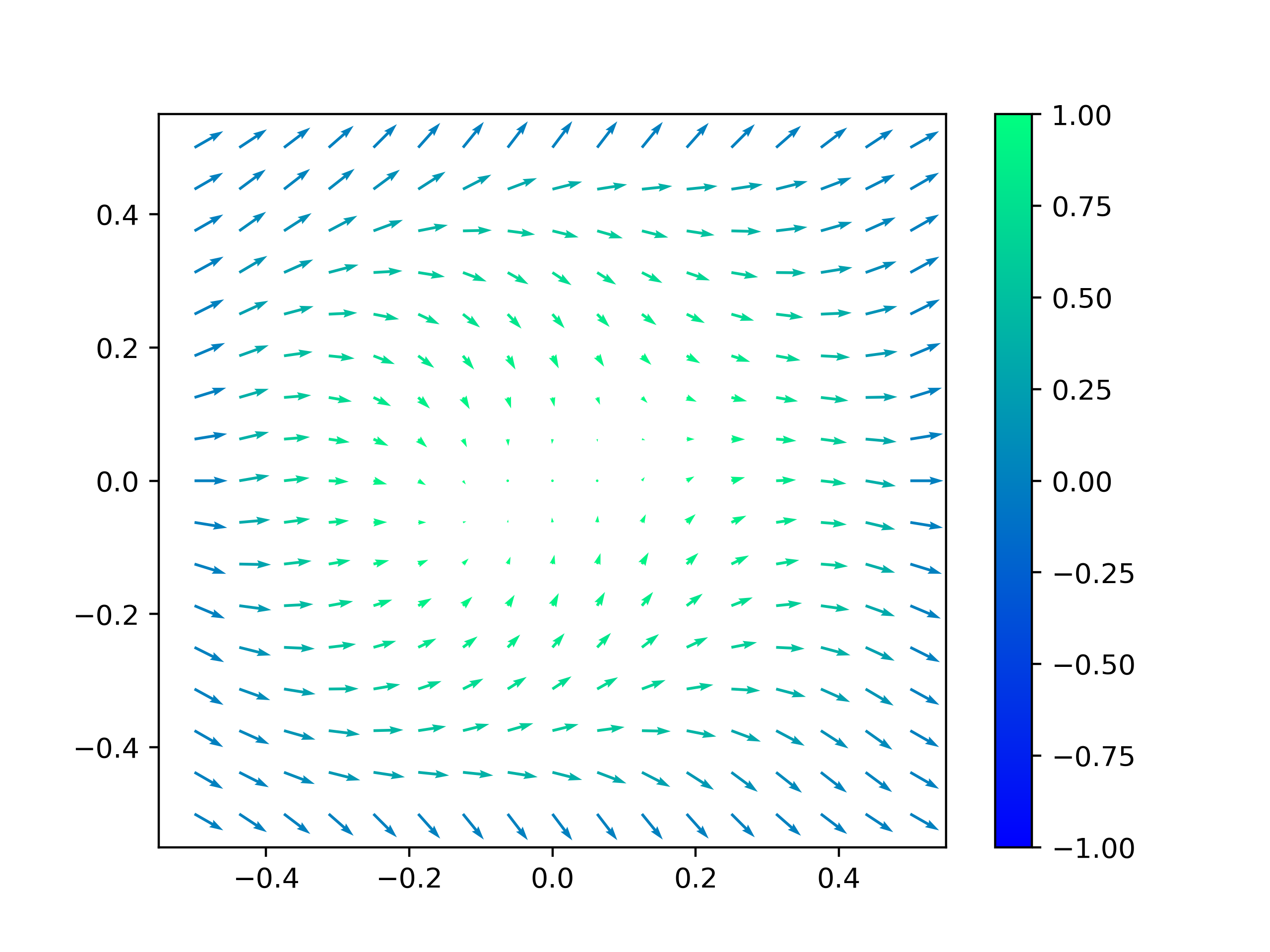}
     \includegraphics[width= .49 \textwidth]{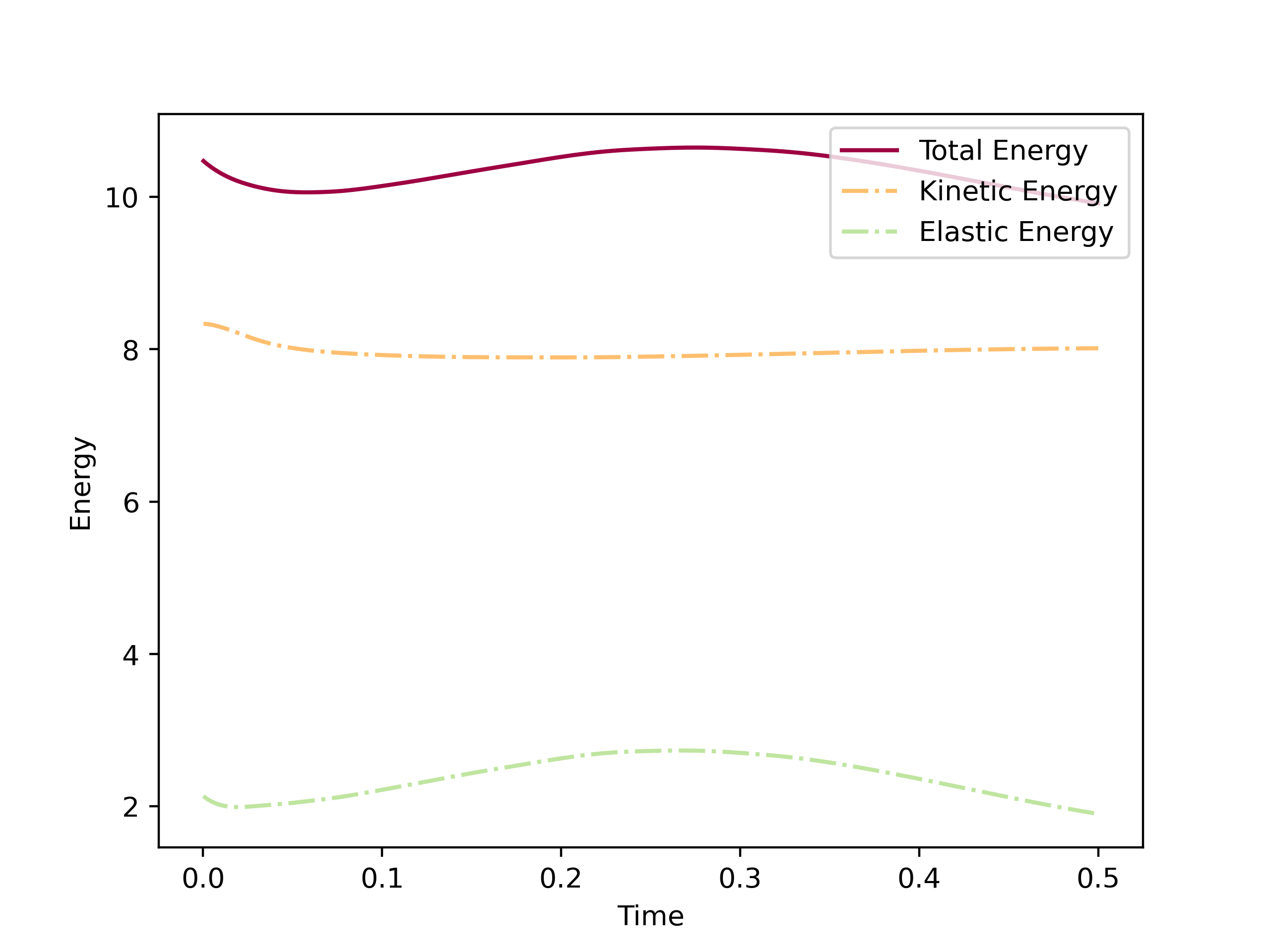}
    \includegraphics[width= .49 \textwidth]{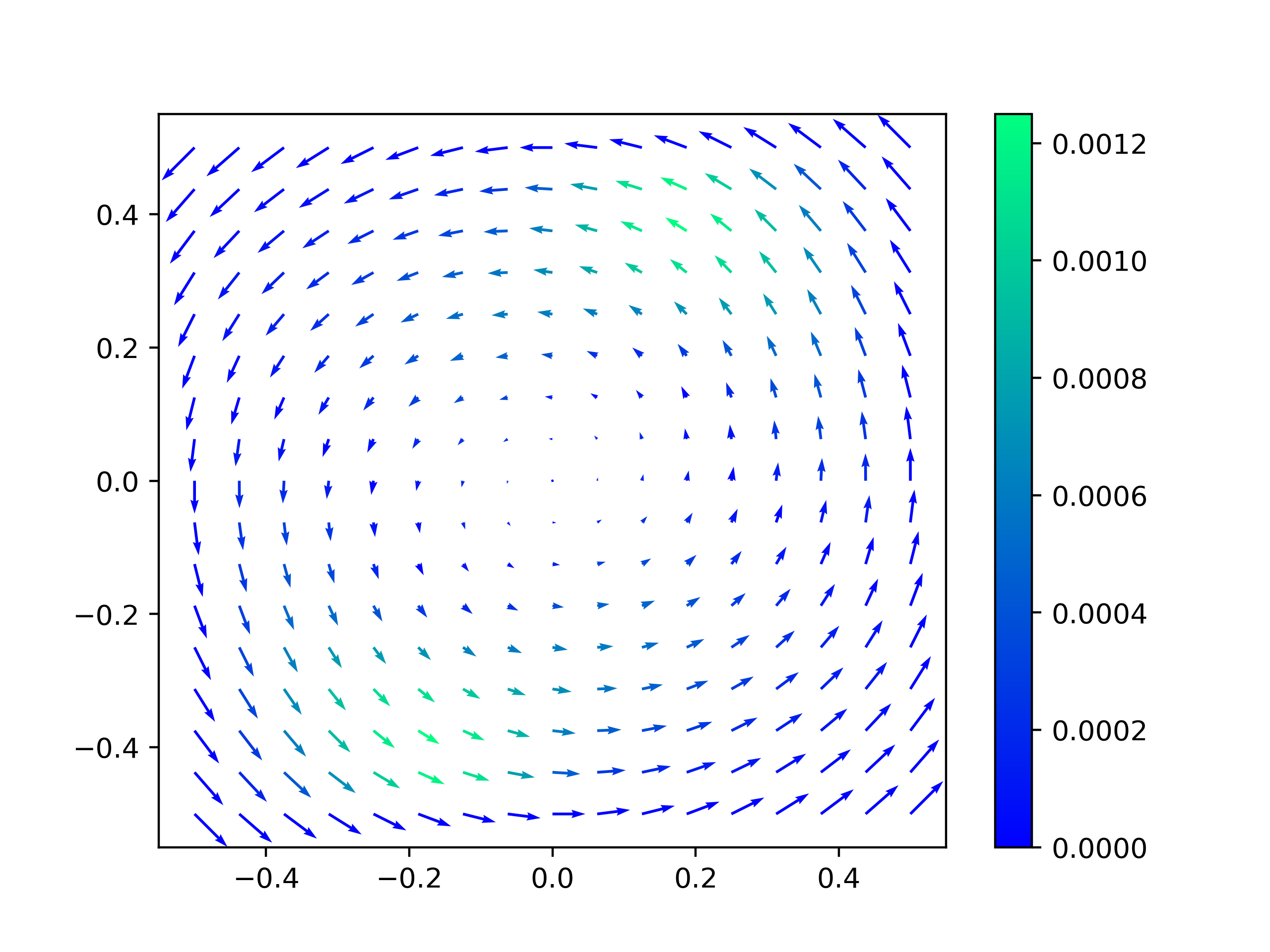}
    \caption{Annihilation in a rotating flow, simplified model \eqref{def:simple_model}: Evolution of the director in the plane $x_3=0$ at time $t= 0.03, 0.05, 0.15, 0.5$, evolution of the energy, velocity field in the plane $x_3=0$ at time $t= 0.15$ (from top left to bottom right).
        The colour marks the (absolute) value of the z-component.
}
    \label{fig:annihilation_rot}
\end{figure}
\begin{figure}[t]
    \centering
    \includegraphics[width= .49 \textwidth]{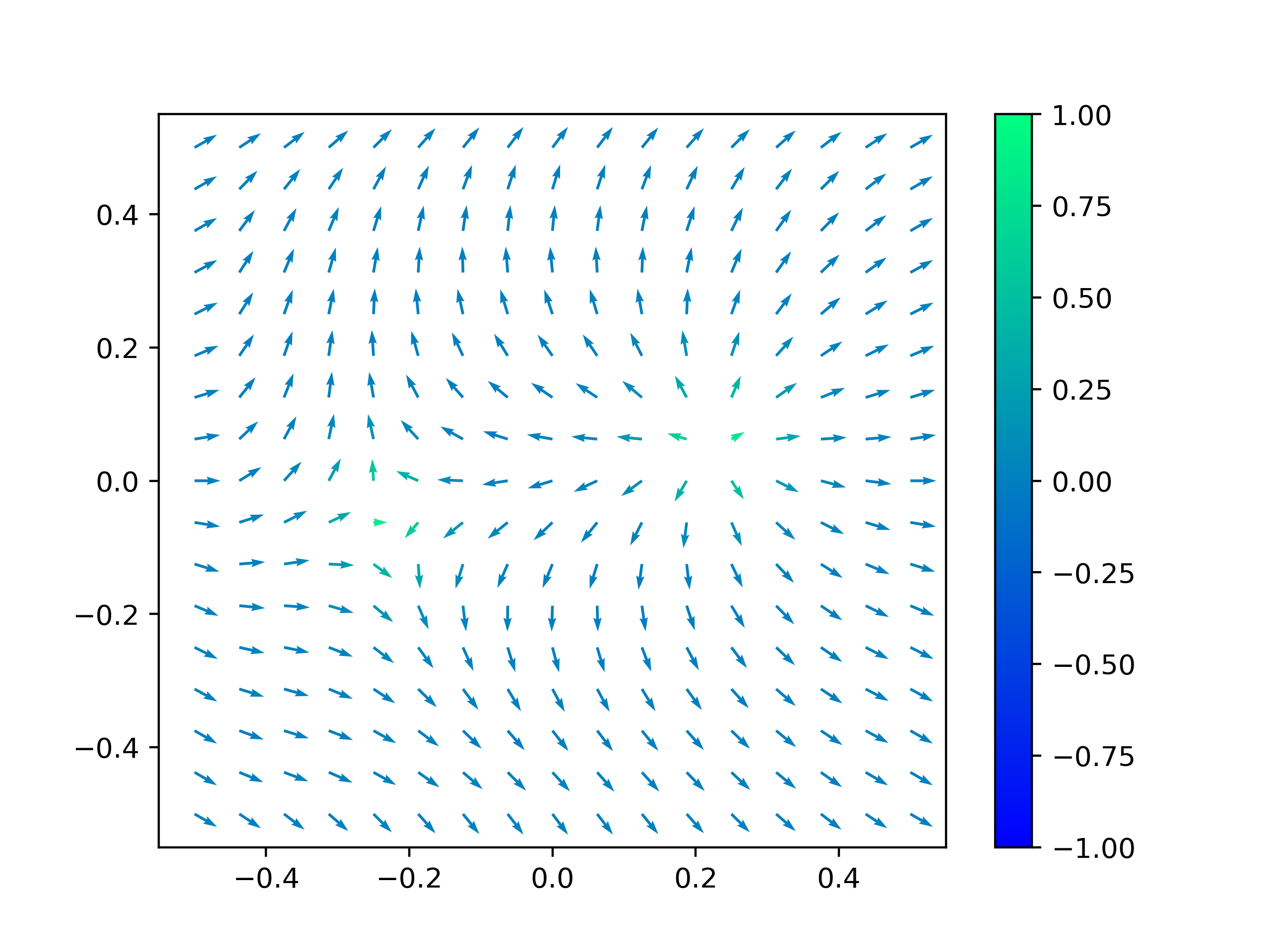}
    \includegraphics[width= .49 \textwidth]{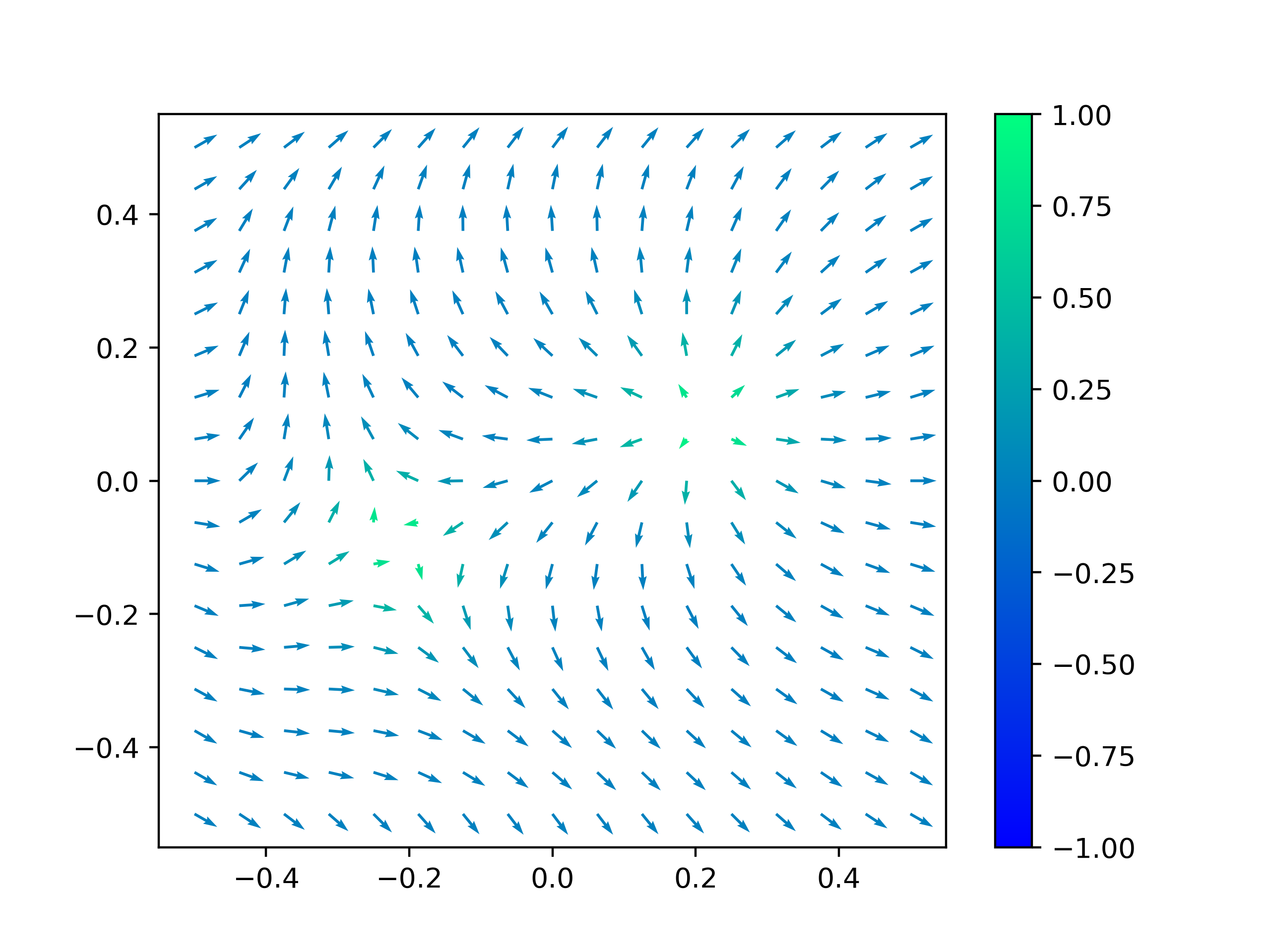}
    \includegraphics[width= .49 \textwidth]{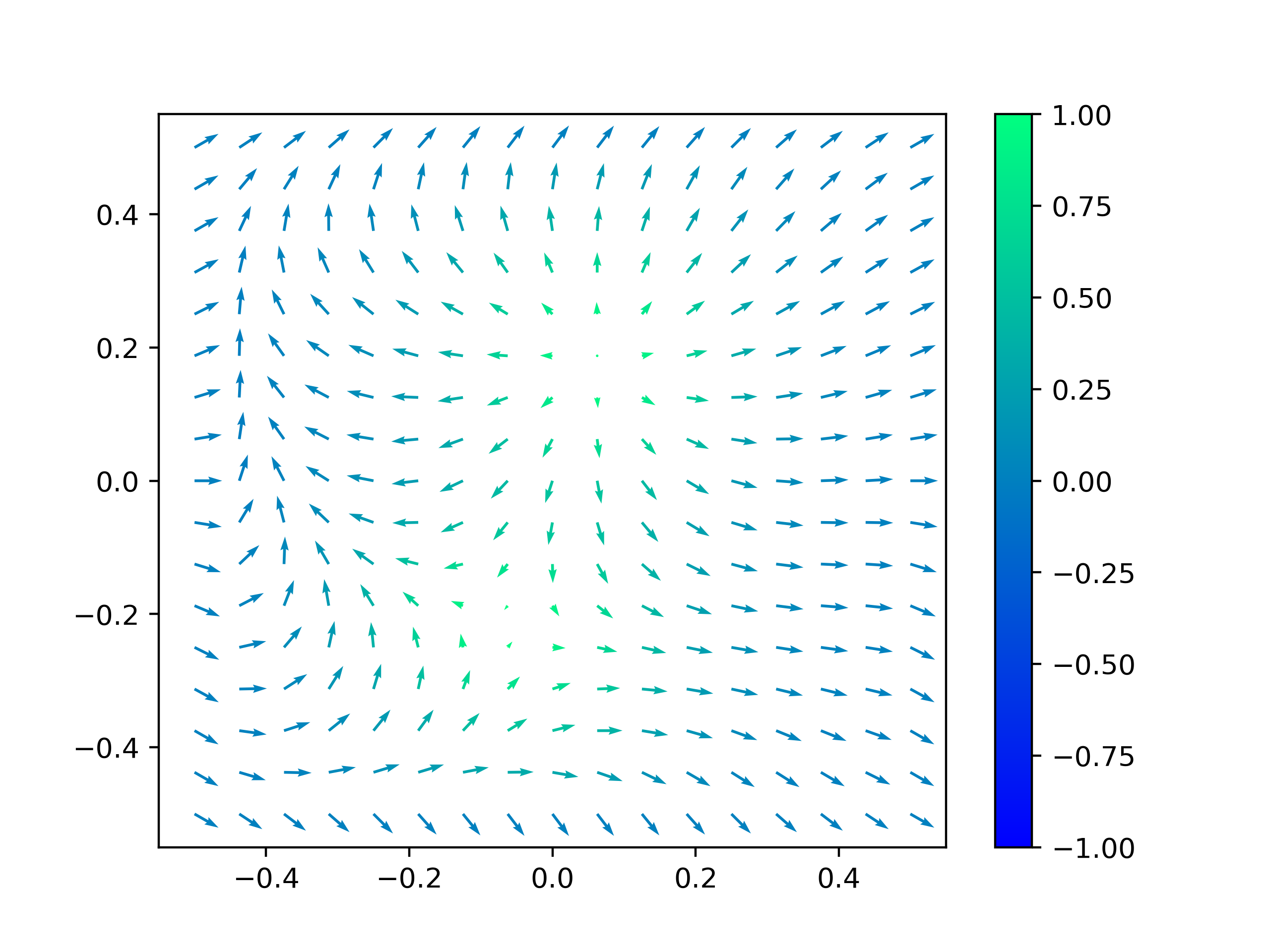}
    \includegraphics[width= .49 \textwidth]{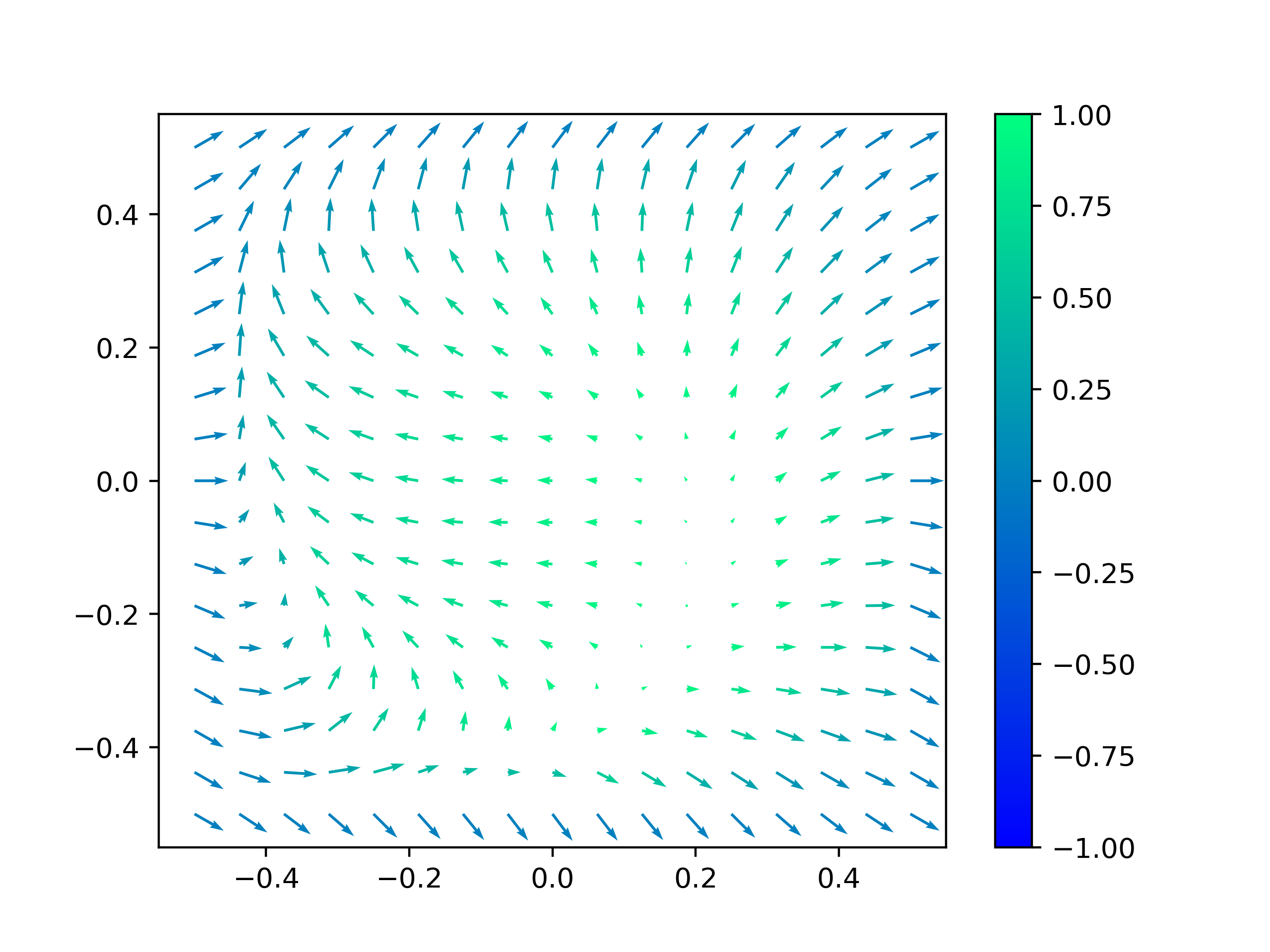}
     \includegraphics[width= .49 \textwidth]{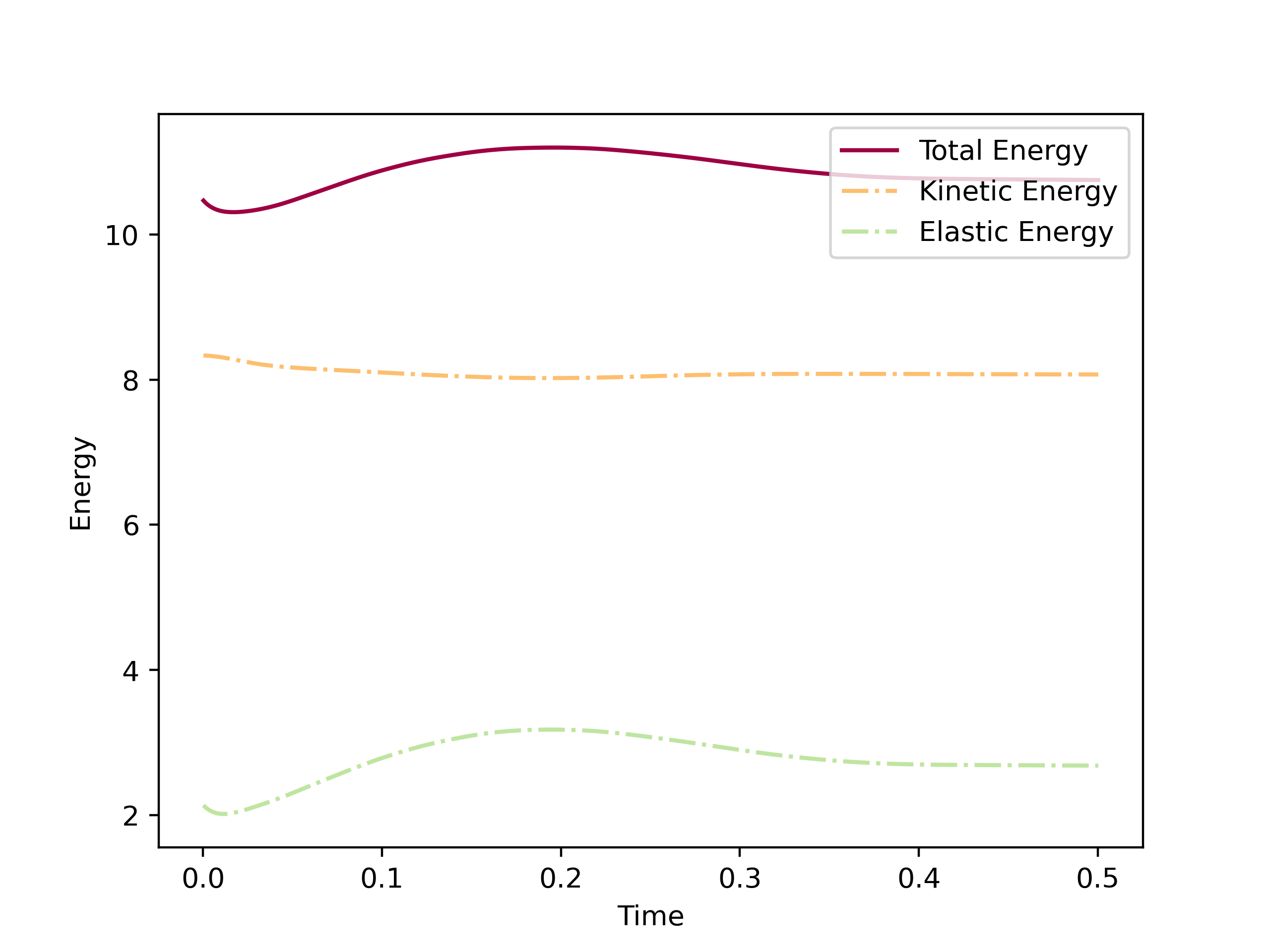}
    \includegraphics[width= .49 \textwidth]{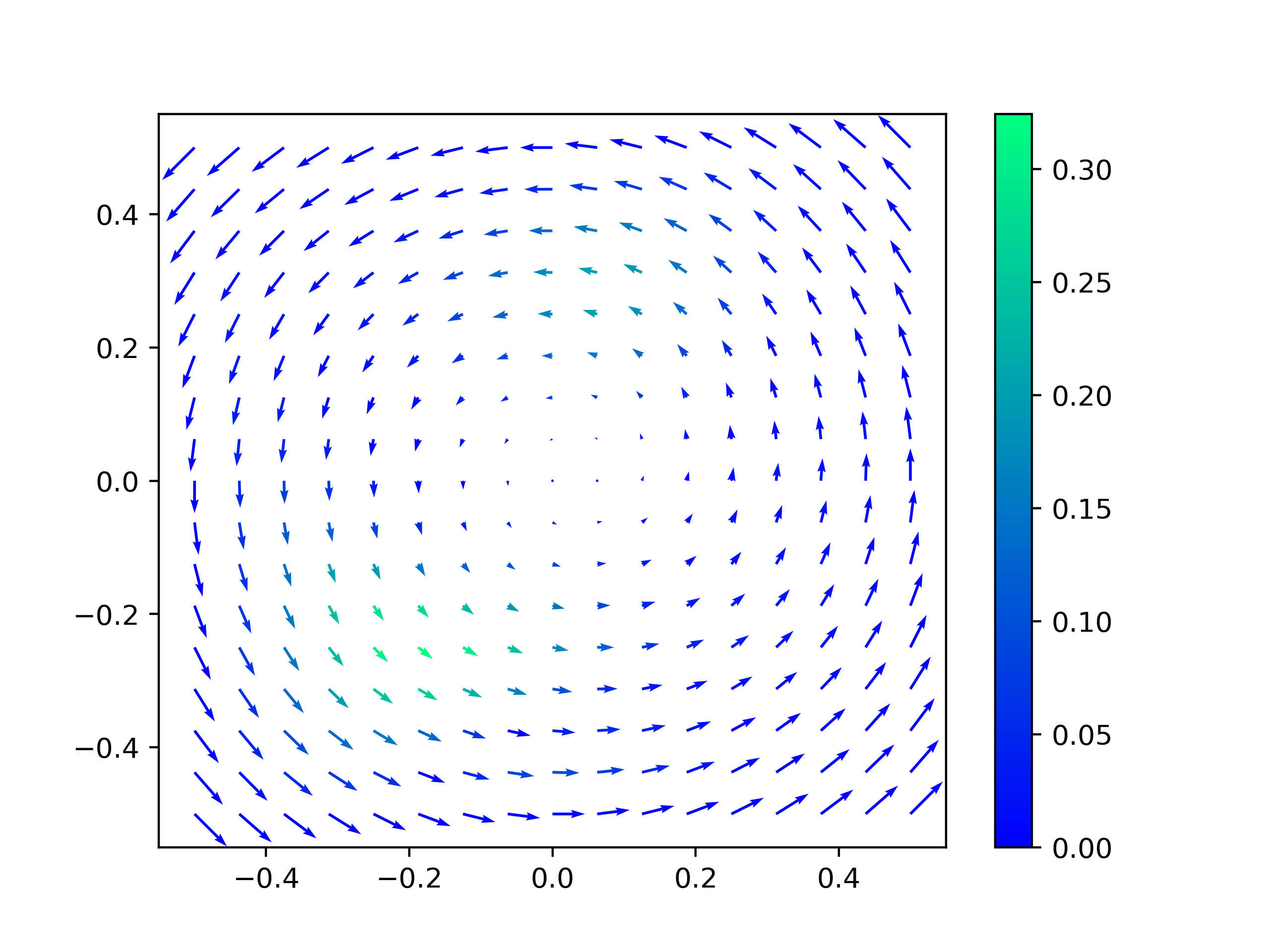}
    \caption{Annihilation in a rotating flow, model \eqref{discrete_scheme}: Evolution of the director in the plane $x_3=0$ at time $t= 0.03, 0.05, 0.15, 0.5$, evolution of the energy, velocity field in the plane $x_3=0$ at time $t= 0.15$ (from top left to bottom right).    The colour marks the (absolute) value of the z-component.
}
    \label{fig:annihilation_rot_gen}
\end{figure}

\section*{Acknowledgements}
The second author acknowledges financial support received in the form of a Ph.D. scholarship from the Friedrich-Naumann-Foundation for Freedom (dt.: Friedrich-Naumann-Stiftung für die Freiheit) with funds from the Federal Ministry of Education and Research (BMBF)  and funding by the Deutsche Forschungsgemeinschaft (DFG, German Research Foundation) under Germany's Excellence Strategy – The Berlin Mathematics Research Center MATH+ (EXC-2046/1, project ID: 390685689). 
This publication is in parts based on the second author's master thesis \cite{thesis_max}.
%
\small

\end{document}